\documentclass[smallextended]{svjour3}

\usepackage{ifthen}

\newboolean{UseBibTex}
\setboolean{UseBibTex}{false}

\usepackage[ngerman,english]{babel}
\usepackage[T1]{fontenc}
\usepackage[latin1]{inputenc}
\usepackage[numbers,comma,sort&compress]{natbib}
\usepackage{verbatim} 
\usepackage{multirow}
\usepackage{hyperref}

\usepackage{longtable}

\newif\ifpdf
\ifx\pdfoutput\undefined
\usepackage{graphicx}
\pdffalse
\else
\pdftrue
\usepackage{graphicx}
\usepackage{epstopdf}
\fi

\usepackage{fancyhdr}

\usepackage{amsmath,amsthm,amsfonts,amssymb,amscd,xspace}
\newtheorem{presumption}[theorem]{Assumption} 
\newtheorem{algorithm}[theorem]{Algorithm}     
\usepackage[centerlast,small,sc]{caption}

\usepackage[numbers,comma,sort&compress]{natbib}
\usepackage{url}
\newcommand{\komma}{\textnormal{~,}}
\newcommand{\punkt}{\textnormal{~.}}
\newcommand{\diff}{\mathrm{d}}

\newcommand{\xor}{\textnormal{or}}

\newcommand{\lambdamin}{\lambda_{\min}}
\newcommand{\lambdamax}{\lambda_{\max}}
\newcommand{\mathalf}{\frac{1}{2}}
\newcommand{\matquarter}{\frac{1}{4}}

\newcommand{\WpBound}{\chi_1}
\newcommand{\WpInverseBound}{C_0}

\newcommand{\kappaBound}{\chi_0}

\newcommand{\Fscal}{c}

\renewcommand{\forall}{\textnormal{for all }}

\newcommand{\zeroVector}[1]{0}

\newcommand{\zeroMatrix}[1]{0}

\newcommand{\Sym}[1]{\mathbb{R}_{\mathrm{sym}}^{#1\times#1}}

\newcommand{\Rpos}{\mathbb{R}_{\geq0}}

\newcommand{\subhessian}[2]{\partial^2{#1}{(#2)}}

\usepackage{color} 
\definecolor{gold}{rgb}{0.85,.66,0}

\newcommand{\GenaueAngabeOne}{}
\newcommand{\GenaueAngabeTwo}{p.~25~ff,~Chapter~3}  
\newcommand{\GenaueAngabeThree}{}
\newcommand{\GenaueAngabeFour}{}
\newcommand{\GenaueAngabeFive}{}
\newcommand{\GenaueAngabeSix}{}

\newcommand{\refh}[1]{\textnormal{(\ref{#1})}}   
\newcommand{\refH}[1]{\textnormal{\ref{#1}}}     

\newcommand{\textH}[1]{\textnormal{#1}}

\def\fullTitle{{A feasible second order bundle algorithm for nonsmooth, nonconvex optimization problems with inequality constraints: I. Derivation and convergence}}

\def\AuthorOne{{Hannes Fendl}}
\def\AuthorThree{{Hermann Schichl}}

\ifpdf
\hypersetup{
    pdftitle  = {\fullTitle},
    pdfauthor = {\AuthorOne and \AuthorThree},
}
\fi

\ifpdf
	\usepackage[color]{changebar}
	\newcommand{\cbstartDVI}{\cbstart}
	\newcommand{\cbendDVI}{\cbend}

	\nochangebars  
	
	\newcommand{\cbeqDVItwo}{{}}
 
	\newcommand{\cbstartMIFFLIN}{\cbcolor{green}\cbstart}  
	\newcommand{\cbendMIFFLIN}{\cbend}      
	
	\newcommand{\cbmathMIFFLIN}{~\colorbox{green}{\color{green} MIF}}  
	\renewcommand{\cbmathMIFFLIN}{{}}
  
	\renewcommand{\cbstartDVI}{{}}          
	\renewcommand{\cbendDVI}{{}}            
	
	\newcommand{\cbstartBLUE}{\cbcolor{blue}\cbstart}
	\newcommand{\cbendBLUE}{\cbend}
	\newcommand{\cbmathBLUE}{}                                    
	
\else
  \newcommand{\cbstartDVI}{{}}
	\newcommand{\cbendDVI}{{}}

	\newcommand{\cbeqDVItwo}{{}}
	
	\newcommand{\cbstartMIFFLIN}{{}}  
	\newcommand{\cbendMIFFLIN}{{}}    
\fi

\newlength{\widebarargwidth}
\newlength{\widebarwidth}
\newlength{\widebarargheight}
\newlength{\widebarargdepth}
\DeclareRobustCommand{\widebar}[1]{%
  \settowidth{\widebarargwidth}{\ensuremath{#1}}%
  \settoheight{\widebarargheight}{\ensuremath{#1}}%
  \settodepth{\widebarargdepth}{\ensuremath{#1}}%
  \addtolength{\widebarargwidth}{-0.2\widebarargheight}%
  \addtolength{\widebarargwidth}{-0.2\widebarargdepth}%
  \makebox[0pt][l]{\addtolength{\widebarargheight}{0.3ex}%
    \hspace{0.2\widebarargheight}%
    \hspace{0.2\widebarargdepth}%
    \hspace{0.5\widebarargwidth}%
    \setlength{\widebarwidth}{0.6\widebarargwidth}%
    \addtolength{\widebarwidth}{0.3ex}%
    \makebox[0pt][c]{\rule[\widebarargheight]{\widebarwidth}{0.1ex}}}%
  {#1}}

\newcommand{\widebarMatrix}{\bar}

\newcounter{PAGEREFERENCEcounter:Paper:BundleAlgorithms}

\newcommand{\emptyh}[1]{}

\begin{document}

\title{\fullTitle}
\author{\AuthorOne\thanks{This research was supported by the Austrian Science Found (FWF) Grant Nr.~P22239-N13.} \and \AuthorThree}
\institute{Faculty of Mathematics, University of Vienna, Austria\\
  Oskar-Morgenstern-Pl.~1, A-1090 Wien, Austria\\
\email{hermann.schichl@univie.ac.at}}


\maketitle

\begin{abstract}
This paper extends the SQP-approach of the well-known bundle-Newton method for nonsmooth unconstrained minimization to the nonlinearly constrained case. Instead of using a penalty function or a filter or an improvement function to deal with the presence of constraints, the search direction is determined by solving a convex quadratically constrained quadratic program to obtain good iteration points.
Furthermore, global convergence of the method is shown under certain mild assumptions.
\end{abstract}



\section{Introduction}
Nonsmooth optimization addresses to solve the optimization problem
\begin{equation}
\begin{split}
&\min{f(x)}\\
&\textnormal{ s.t. }F_i(x)\leq0~~~\forall i=1,\dots,m\komma
\label{AOB:OptProb}
\end{split}
\end{equation}
where $f,F_i:\mathbb{R}^n\longrightarrow\mathbb{R}$ are locally Lipschitz continuous.
Since $F_i(x)\leq0$ for all $i=1,\dots,m$ if and only if $F(x):=\max_{i=1,\dots,m}{\Fscal_iF_i(x)}\leq0$ with constants $\Fscal_i>0$ and since $F$ is still locally Lipschitz continuous (cf., e.g., \citet[p.~969,~Theorem~6~(a)]{Mifflin}, we can always assume $m=1$ in (\ref{AOB:OptProb}). Since we do not take scaling problems of the constraints into account in this paper, we choose $\Fscal_i=1$ for all $i=1,\dots,m$ and therefore we always consider the nonsmooth optimization problem with a single nonsmooth constraint
\begin{equation}
\begin{split}
&\min{f(x)}\\
&\textnormal{ s.t. }F(x)\leq0\komma
\label{BundleSQP:OptProblem}
\end{split}
\end{equation}
where $F:\mathbb{R}^n\longrightarrow\mathbb{R}$ is locally Lipschitz continuous, instead of (\ref{AOB:OptProb}).

Since locally Lipschitz continuous functions are differentiable almost everywhere, both $f$ and $F$ may have kinks and therefore
already
the attempt to solve an unconstrained nonsmooth optimization problem by a smooth solver (e.g., by a line search algorithm or by a trust region method) by just replacing the gradient by a subgradient, fails in general (cf., e.g., \citet[p.~461-462]{Zowe}):
If
$g$ is an element of the subdifferential $\partial f(x)$,
then the search direction $-g$ does not need to be a direction of descent (contrary to the behavior of the gradient of a differentiable function).
Furthermore,
it
can happen that $\lbrace x_k\rbrace$ converges towards a minimizer $\hat{x}$, although the sequence of gradients $\lbrace\nabla f(x_k)\rbrace$ does not converge towards $\zeroVector{n}$ and therefore we cannot identify $\hat{x}$ as a minimizer.
Moreover,
it
can happen that $\lbrace x_k\rbrace$ converges towards a point $\hat{x}$, but $\hat{x}$ is not stationary for $f$.
The reason for these problems
is that if
$f$ is not differentiable at $x$, then the gradient $\nabla f$ is discontinuous at $x$ and therefore $\nabla f(x)$ does not give any information about the behavior of $\nabla f$ in a neighborhood of $x$.

Not surprisingly, like in smooth optimization, the presence of constraints adds additional complexity, since constructing a descent sequence whose limit satisfies the constraints is (both theoretically and numerically) much more difficult than achieving this aim without the requirement of satisfying any restrictions.

Methods that are able to solve nonsmooth optimization problems are, e.g., bundle algorithms which force a descent of the objective function by using local knowledge of the function, the R-algorithm by \citet{Shor} or stochastic algorithms that try to approximate the subdifferential. In the following we will present a few implementations of these methods.
~\\
\textbf{Bundle algorithms}.\refstepcounter{PAGEREFERENCEcounter:Paper:BundleAlgorithms}\label{Paper:BundleAlgorithms} Bundle algorithms are iterative methods for solving nonsmooth optimization problems. They only need to compute one element $g$ of the subdifferential $\partial f(x)$ per iteration, which in practice is easily computable by algorithmic differentiation (cf., e.g., \citet{Griewank}). For computing the search direction, they collect information about the function (e.g., subgradients) from previous iterations. This collected information is referred to as ``the bundle''.
As in smooth optimization, convex nonsmooth optimization is much easier than nonconvex nonsmooth optimization as well in theory as in practice because
convex functions only have global minimizers and
the cutting plane approximation of a convex function always yields an underestimation which in particular simplifies convergence analysis.
A good introduction to nonsmooth optimization which treats the convex, unconstrained case in great detail is \citet[p.~106~ff]{Bonn}. Moreover, very detailed standard references for nonsmooth nonconvex optimization are
\citet{KiwielBook} and \citet{MakelaBook}, which both in particular discuss constrained problems extensively.

Now we give a brief overview over a few bundle algorithms. We start this overview with the following bundle algorithms that support nonconvex constraints:
The multiobjective proximal bundle method for nonconvex nonsmooth optimization (MPBNGC) by \citet{MPBNGC} is a first order method that uses the improvement function
$h_{x_k}(x):=\max{\big(f(x)-f(x_k),F(x)\big)}$
for the handling of the constraints. Further details about the proximal bundle method can be found in \citet{MakelaBook}.
The algorithms in \citet{MifflinLaxenburg,Mifflin,MifflinNB} support a nonconvex objective function as well as nonconvex constraints (cf.~Remark \ref{AN:AOB:remark:Mifflin}).
NOA by \citet{KiwielNOA} is a nonsmooth optimization algorithm that handles nonconvex constraints by using a penalty function or an improvement function, while in the special case of convex constraints it offers an alternative treatment by the constraint linearization technique by \citet{KiwielCLS}.
The limited memory bundle algorithm for inequality constrained nondifferentiable optimization by \citet{Karmitsa} combines LMBM by \citet{KarmitsaLMBM} with the feasible directions interior point technique by \citet{Herskovits,HerskovitsSantos} for dealing with the constraints. The search direction is determined by solving a linear system.

In addition a few bundle algorithms can only handle convex constraints:
The bundle trust algorithm by \citet{Schramm,SchrammOfficialPublication}, which also supports a nonconvex objective function, handles the constraints by using the constraint linearization technique by \citet{KiwielCLS}.
The bundle filter algorithm by \citet{FleLey} is only applicable to convex optimization problems and it computes the search direction by solving
a linear program.
The bundle-filter method for nonsmooth convex constrained optimization by \citet{Sag1} is based on the improvement function.
The infeasible bundle method for nonsmooth convex constrained optimization by \citet{Sag2} is also based on the improvement function,
but it uses neither a penalty function nor a filter.

Moreover, there are some bundle algorithms that support at most linear constraints:
The variable metric bundle method PVAR by \citet{LuksanPVARconvex,LuksanPVAR} can solve nonsmooth linearly constrained problems with a nonconvex objective function.
The implementation PBUN of the proximal bundle method by \citet{LuksanPBUN,VlcekPBUN,LuksanPBUNandBNEW} optimizes a nonconvex objective function, where the feasible set is given by linear constraints.
The proximal bundle method by \citet{KiwielRestrictedStep}, which is based on a restricted step concept, can handle a nonconvex objective function and linear constraints.
The focus of the limited memory bundle method LMBM by \citet{KarmitsaLMBM,KarmitsaLBMBpublication1,KarmitsaLBMBpublication2} is the treatment of large-scale nonsmooth nonconvex unconstrained optimization problems.
This is done by combining ideas from the variable metric bundle method
\citet{LuksanPVARconvex,LuksanPVAR}
and limited memory variable metric methods by, e.g, \citet{Byrd}. Its bound constraint version is presented in \citet{KarmitsaLBMBpublication1BC,KarmitsaLBMBpublication2BC}.

All algorithms mentioned above only use first order information of the objective function and the constraints as input. Nevertheless, there are some very interesting
bundle methods, since they are Newton-like methods (at least in some sense) and which only support the handling of linear constraints yet as far as I know (except for putting the objective function and the constraints into a penalty function with a fixed penalty parameter and then applying the unconstrained algorithm to the penalty function):
The quasi-Newton bundle-type method for nondifferentiable convex optimization by \citet{MSQ} generalizes the idea of Quasi-Newton methods to nonsmooth optimization
and it converges superlinearly for strongly convex functions (and some additional technical assumptions).
The bundle-Newton method for nonsmooth unconstrained minimization by \citet{Luksan} supports a nonconvex objective function,
it is based on an SQP-approach, and it is the only method for solving nonsmooth optimization problems that I know which uses Hessian information. Furthermore, its rate of convergence is superlinear for strongly convex, twice times continuously differentiable functions. Moreover, a description of the implementation PNEW of the bundle-Newton method can be found in \citet{LuksanPBUNandBNEW}).

In this paper we extend
the bundle-Newton method to a second order bundle algorithm
for nonsmooth, nonconvex inequality constraints by using additional quadratic information:
We use second order information of the constraint (cf.~(\ref{BundleSQP:OptProblem})).
Furthermore,
we
use the SQP-approach of the bundle-Newton method for computing the search direction for the constrained case and combine it with the idea of quadratic constraint approximation, as it is used, e.g., in the sequential quadratically constrained quadratic programming method by \citet{Solodov} (this method is not a bundle method), in the hope to obtain good feasible iterates, where we only accept strictly feasible points as serious steps. Therefore, we have to solve a strictly feasible convex QCQP for computing the search direction
\cbstartMIFFLIN
(Note that this approach also yields a generalization of the original bundle-Newton method in the unconstrained case).
\cbendMIFFLIN
Using such a QCQP for computing the search direction yields a line search condition for accepting infeasible points as trial points (which is different to that in, e.g., \citet{MifflinNB}).
One of the most important properties of the convex QP (that is used to determine the search direction) with respect to a bundle method is its strong duality (e.g., for a meaningful termination criterion, for global convergence,\dots) which is also true in the case of strictly feasible convex QCQPs
(cf.~Subsection \ref{section:GlobalConvergence}).

For Numerical results we refer the reader to \citet[\GenaueAngabeOne]{HannesPaperA}. Proofs that are presented in this paper can be looked up in explicit detail in \citet[\GenaueAngabeTwo]{HannesDissertation}.\\
\textbf{Other algorithms for nonsmooth optimization}. There exist several other methods for solving nonsmooth optimization problems that are not based on the bundle approach or that are no bundle algorithms in the sense as described
on page \pageref{Paper:BundleAlgorithms}.
A few representatives of these methods that support at most linear constraints are:
The algorithm PMIN by \citet{LuksanPMIN}, which is based on \citet{LuksanPMINpublication}, solves linearly constrained minimax optimization problems, i.e.~the objective function must be maximum of twice times continuously differentiable functions.
The robust gradient sampling algorithm for nonsmooth nonconvex optimization by \citet{BLO} approximates the whole subdifferential at each iteration (cf.~\citet{BLOapprox}) and does not make null steps.
The MATLAB-code HANSO by \citet{OvertonHANSO} combines ideas from BFGS algorithms (cf.~\citet{BFGSbasisForHANSO}) and from the gradient sampling algorithm by \citet{BLO} for solving nonsmooth unconstrained optimization problems.
The derivative-free bundle method (DFBM) by \citet{BagirovDFBM}, where ``derivate-free'' means that no derivate information is used explicitly, can solve linearly constrained nonsmooth problems. The subgradients are approximated by finite differences in this algorithm (cf.~\citet{BagirovApproximation}). DFBM is an essential part of the programming library for global and non-smooth optimization GANSO by \citet{BagirovGANSO}.
The discrete gradient method DGM for nonsmooth nonconvex unconstrained optimization by \citet{BagirovDGM} is a bundle-like method that does not compute subgradients, but approximates them by discrete gradients.
The quasisecant method QSM for minimizing nonsmooth nonconvex functions by \citet{BagirovQSM} combines ideas both from bundle methods and from the gradient sampling method by \citet{BLO}.

Furthermore, we want to mention the following solver for nonsmooth convex optimization problems:
The oracle based optimization engine OBOE by \citet{VialOBOE} is based on the analytic center cutting plane method by \citet{NesterovForOBOE}, which is an interior point framework.

Finally, we list a few algorithms that can also handle nonconvex constraints:
The robust sequential quadratic programming algorithm extends the gradient sampling algorithm by \citet{CO} for nonconvex, nonsmooth constrained optimization.
SolvOpt by \citet{SolvOptPublication} is an implementation of the R-algorithm by \citet{Shor}. It handles the constraints by automatically adapting the penalty parameter.
\texttt{ralg} by \citet{Kroshko} is another implementation of the R-algorithm by \citet{Shor} that is only available in (the interpreted programming language) Python. The constraints are handled by a filter technique.
\begin{remark}
\citet{KarmitsaComparison} gives a brief, excellent description of the main ideas (including very good readable pseudo code) of many of the unconstrained methods resp.~the unconstrained versions of the methods which we mentioned
above (for further information visit the online decision tree for nonsmooth optimization software by \citet{KarmitsaDecisionTree}).
\end{remark}
The paper is organized as follows: In Section \ref{Paper:OptimizationTheory} we recall the basics of an SQP-method which is a common technique in smooth optimization and we summarize the most important facts about nonsmooth optimization theory. In Section \ref{Paper:DerivationOfTheMethod} we give the theoretical foundation of our second order bundle algorithm and afterwards we present the algorithm and the line search in detail. Finally, we show the convergence of the line search and the global convergence of the algorithm in Section \ref{Paper:Convergence}.

Throughout the paper we use the following notation:
We denote the non-negative real numbers by $\Rpos:=\lbrace x\in\mathbb{R}:~x\geq0\rbrace$. We denote the space of all symmetric $n\times n$-matrices by $\Sym{n}$. For $x\in\mathbb{R}^n$ we denote the Euclidean norm of $x$ by $\lvert x\rvert$,
and for $A\in\mathrm{Sym(n)}$ we denote the spectral norm of $A$ by $\lvert A\rvert$. Furthermore, we denote the smallest resp.~the largest eigenvalue of a positive definite matrix $A\in\mathbb{R}^{n\times n}$ by $\lambdamin(A)$ resp.~$\lambdamax(A)$.
\cbstartDVI
Therefore, if $A$ is positive definite, we have
\begin{equation}
\lvert A\rvert=\sqrt{\lambdamax(A)}
\label{GolubVanLoan:Satz:SpektralNormVonPositivDefiniterMatrix}
\end{equation}
(cf., e.g., \citet[p.~394,~Follow~up~of~Theorem~8.1.2]{GolubVanLoanMatrixCompuations}).
\cbendDVI

\section{Optimization theory}
\label{Paper:OptimizationTheory}
In the following section  we summarize the basics of an SQP-method, since we will approximate a nonsmooth problem by a sequence of smooth problems to derive our algorithm in Section \ref{Paper:DerivationOfTheMethod} and hence we will need some facts about smooth optimization, and we present the most important facts about nonsmooth optimization theory.
\subsection{Smooth optimality conditions \& SQP}
\begin{theorem}
Let $f,F_i:\mathbb{R}^n\longrightarrow\mathbb{R}$ (with $i=1,\dots,m$) be continuously differentiable and $\hat{x}\in\mathbb{R}^n$ be a solution of the smooth optimization problem
\begin{equation}
\begin{split}
&\min{f(x)}\\
&\textnormal{ s.t. }F_i(x)\leq0~~~\forall i=1,\dots,m\punkt
\label{SQP:Def:OptProblem}
\end{split}
\end{equation}
Then there exist $\kappa\geq0$ and $\lambda\geq\zeroVector{m}$ with
\begin{equation}
  \begin{aligned}
    &\kappa\nabla f(\hat{x})^T+\sum_{i=1}^m{\nabla F_i(\hat{x})^T\lambda_i}=\zeroVector{n}\komma\\
    &\lambda_iF_i(\hat{x})= 0~\forall i=1,\dots,m\komma\\
    &\kappa=1~\xor~(\kappa=0\textnormal{, }\lambda\not=\zeroVector{m})\punkt
\label{AOB:KJUnglgsNBSystemKomplementaereVektoren}
  \end{aligned}
\end{equation}
If all occurring functions are convex, then
the existence of a strictly feasible $x$ (i.e.~$F(x)<0$) always guarantees $\kappa=1$,
and
the conditions \refh{AOB:KJUnglgsNBSystemKomplementaereVektoren} are sufficient (for a feasible $\hat{x}$ being a minimizer of \refh{SQP:Def:OptProblem}).
\end{theorem}
\begin{proof}
Combine, e.g., \citet[p.~19,~4.1~Theorem]{SchichlOptBed} and \citet[p.~243,~5.5.3~KKT~optimality~conditions]{Boyd}.
\qedhere
\end{proof}

One possibility to find a solution of the optimization problem (\ref{SQP:Def:OptProblem}) is using an SQP-method (sequential quadratic programming). An SQP-method minimizes the quadratic approximation of the Lagrangian $L:\mathbb{R}^n\times\Rpos^m\longrightarrow\mathbb{R}$ given by $L(x,\lambda):=f(x)+\sum_{i=1}^m{F_i(x)\lambda_i}$
subject to linearizations of the constraints and then it uses the obtained minimizer as the new iteration point (or it performs a line search between the current iteration point and the obtained minimizer to determine the new iteration point). Since quadratic information is necessary for this approach, we demand $f,F_i:\mathbb{R}^n\longrightarrow\mathbb{R}$ (with $i=1,\dots,m$) to be $C^2$ in this
subsection.
\begin{proposition}
\label{AN:proposition:SQP-step}
Let
the matrix $\nabla F(x)\in\mathbb{R}^{m\times n}$ (gradient of the constraints) have full rank (``Constraint qualification'')
and let the Hessian of the Lagrangian with respect to the $x$-components
$\nabla_{xx}^2L(x,\lambda)=\nabla^2f(x)+\sum_{i=1}^m\nabla^2F_i(x)\lambda_i$
be positive definite on the tangent space of the constraints, i.e.~$d^T\nabla_{xx}^2L(x,\lambda)d>0$ for all $d\in\mathbb{R}^n$ with $d\not=\zeroVector{n}$ and $\nabla F(x)d=\zeroVector{m}$
(cf.~\textH{\citet[p.~531,~Assumption~18.1]{Nocedal}}).
Then the \textbf{SQP-step} for optimization problem \refh{SQP:Def:OptProblem} is given by the solution of the QP
\begin{equation}
\begin{split}
f(x)\hspace{3pt}+
                 &\min_{d}{\nabla f(x)d+\tfrac{1}{2}d^T\nabla_{xx}^2L(x,\lambda)d}\\
                 &\textnormal{ s.t. }F_i(x)+\nabla F_i(x)d\leq0~~~\forall i=1,\dots,m\punkt
\label{SQP:Satz:QP}
\end{split}
\end{equation}
\end{proposition}
\begin{proof}
Straightforward calculations.
\qedhere
\end{proof}
\begin{remark}
\label{SQP:Bem:LinearNBkonsistent}
A difficulty of an infeasible SQP-method (e.g., SNOPT by \citet{Gill}) --- i.e.~infeasible iteration points $x_k$ may occur --- is that the linear constraints of the QP (\ref{SQP:Satz:QP}) can be infeasible (cf., e.g, \citet[p.~535,~18.3~Algorithmic~development]{Nocedal}).
Note that this difficulty does not arise for a feasible SQP-method (e.g., FSQP by \citet{Tits2}) --- i.e.~only feasible iteration points $x_k$ are accepted --- as then $d=\zeroVector{n}$ is always feasible for the QP (\ref{SQP:Satz:QP}). Nevertheless, in this case it can be difficult to obtain feasible points that make good progress towards a solution (cf.~Remark \ref{BundleSQP:remark:Maratos}).
\end{remark}

\subsection{Nonsmooth Optimality conditions}
We gather information on the optimality conditions of the nonsmooth optimization problem
(\ref{AOB:OptProb})
with locally Lipschitz continuous functions $f,F_i:\mathbb{R}^n\longrightarrow\mathbb{R}$ for $i=1,\dots,m$. For this purpose, we closely follow the exposition in \citet{Borwein}.
\begin{definition}
Let $U\subseteq\mathbb{R}^n$ be open and $f:\mathbb{R}^n\longrightarrow\mathbb{R}$.
We define the \textbf{Clarke directional derivative} in $x\in U$ in direction $d\in\mathbb{R}^n$ by
\begin{equation*}
f^0(x,d):=
\limsup_{h\rightarrow\zeroVector{n},t\downarrow0}
\frac{f(x+h+td)-f(x+h)}{t}
\end{equation*}
and we define the \textbf{subdifferential} $\partial f(x)\subseteq\mathbb{R}^n$ of $f$ in $x\in U$ by
\begin{equation*}
\partial f(x)
:=
\mathrm{ch}\lbrace g\in\mathbb{R}^n:~g^Td\leq f^0(x,d)~\forall d\in\mathbb{R}^n\rbrace
\komma
\end{equation*}
where $\mathrm{ch}$ denotes the convex hull of a set. The elements of $\partial f(x)$ are called \textbf{subgradients}.
We define the set $\subhessian{f}{x}\subseteq\Sym{n}$ of the substitutes for the Hessian of $f$ at $x$ by
\begin{equation}
\subhessian{f}{x}
:=
\left\lbrace
\begin{array}{ll}
\lbrace G\rbrace & \textnormal{if the Hessian }G\textnormal{ of }f\textnormal{ at }x\textnormal{ exists}\\
\Sym{n}          & \textnormal{else}\punkt
\end{array}
\right.
\label{Def:subhessian}
\end{equation}
\end{definition}
We summarize the most important properties of the Clarke directional derivative and the subdifferential. The following two results are taken from \citet{Borwein}.
\begin{proposition}
\label{PAPER:LocallyBoundedUpperSemiContinous}
The subdifferential $\partial f(x)$ is non-empty, convex and compact. Furthermore, $\partial f:\mathbb{R}^n\longrightarrow\mathcal{P}(\mathbb{R}^n)$, where $\mathcal{P}(\mathbb{R}^n)$ denotes the power set of $\mathbb{R}^n$, is locally bounded and upper semicontinuous.
\end{proposition}
\begin{theorem}[First order nonsmooth optimality conditions]
\label{AOB:Satz:KarushJohnUnglgsNBNichtGlattAlternativform}
Let $\hat{x}$ be a local minimizer of \refh{AOB:OptProb} and $f,F_i:\mathbb{R}^n\longrightarrow\mathbb{R}$ (with $i=1,\dots,m$) be Lipschitz continuous in a neighborhood of $\hat{x}$.
Then there exists $\kappa\geq0$ and $\lambda\geq0$ with
\begin{equation*}
  \begin{aligned}
    &\zeroVector{n}\in\kappa\partial f(\hat{x})+\sum_{i=1}^m\lambda_i\partial F_i(\hat{x})\komma\\
    &\lambda_iF_i(\hat{x})= 0~\forall i=1,\dots,m\komma\\
    &\kappa=1~\xor~(\kappa=0\textnormal{, }\lambda\not=\zeroVector{m})\punkt
  \end{aligned}
\end{equation*}
Furthermore, if there exists a direction $d\in\mathbb{R}^n$ that satisfies the (nonsmooth) constraint qualification
\begin{equation}
F_j^{\circ}(\hat{x},d)<0\textnormal{ for all }j\in\lbrace1,\dots,m\rbrace\textnormal{ with }F_j(\hat{x})=0
\komma
\label{AOB:Satz:KarushJohnUnglgsNBNichtGlattAlternativform:CQ}
\end{equation}
then we can always set $\kappa=1$.
\end{theorem}
\begin{corollary}
\label{AOB:Satz:KarushJohnUnglgsNBNichtGlattAlternativformCorollary}
Let the constraint qualification \refh{AOB:Satz:KarushJohnUnglgsNBNichtGlattAlternativform:CQ} be satisfied for \refh{BundleSQP:OptProblem}, then the optimality condition for \refh{BundleSQP:OptProblem} reads as follows: There exists $\lambda\geq0$ with
\begin{equation}
            \zeroVector{n}\in \partial f(\hat{x})+\lambda\partial F(\hat{x})\komma\quad
\lambda         F(\hat{x})=   0                                             \komma\quad
                F(\hat{x})\leq0                                             \punkt
\label{AOB:Satz:KarushJohnUnglgsNBNichtGlattAlternativformGlgsSystem}
\end{equation}
\end{corollary}
\begin{proof}
Inserting into Theorem \ref{AOB:Satz:KarushJohnUnglgsNBNichtGlattAlternativform} with $m=1$.
\qedhere
\end{proof}
\begin{remark}
\label{AN:AOB:remark:Mifflin}
The algorithms in \citet{MifflinLaxenburg,Mifflin,MifflinNB} (for solving nonlinearly constrained nonsmooth optimization problems) use a fixed point theorem about certain upper semicontinuous point to set mappings by \citet{Merrill} as optimality condition which is different to an approach with the optimality conditions in Theorem \ref{AOB:Satz:KarushJohnUnglgsNBNichtGlattAlternativform} or Corollary \ref{AOB:Satz:KarushJohnUnglgsNBNichtGlattAlternativformCorollary}.
\end{remark}

\section{Derivation of the method}
\label{Paper:DerivationOfTheMethod}
In this section we discuss the theoretical basics of our second order bundle algorithm and we give a detailed presentation of the algorithm and the line search.
\subsection{Theoretical basics}
We assume in this section that the functions $f,F:\mathbb{R}^n\longrightarrow\mathbb{R}$ are locally Lipschitz continuous, $g_j\in\partial f(y_j)$, $\hat{g}_j\in\partial F(y_j)$, and
let $G_j$ and $\hat{G}_j$ be approximations of elements in $\subhessian{f}{y_j}$ and $\subhessian{F}{y_j}$ (cf.~(\ref{Def:subhessian})), respectively.

Our goal is to determine a local minimizer for the nonsmooth optimization problem (\ref{BundleSQP:OptProblem})
\begin{equation*}
\begin{split}
&\min_{x\in\mathbb{R}^n}{f(x)}\\
&\textnormal{ s.t. }F(x)\leq0\komma
\end{split}
\end{equation*}
and therefore we want to find
a point that satisfies the first order optimality conditions (\ref{AOB:Satz:KarushJohnUnglgsNBNichtGlattAlternativformGlgsSystem}). To attain this goal, we will propose an extension to the bundle-Newton method for nonsmooth unconstrained minimization by \citet{Luksan}: If we are in the optimization problem (\ref{BundleSQP:OptProblem}) at the iteration point $x_k\in\mathbb{R}^n$ (with iteration index $k$), we want to compute the next trial point (i.e.~the search direction) by approximating both the objective function $f$ and the constraint $F$ at $x_k$ by a piecewise quadratic function and then perform a single SQP-step, as defined in Proposition \ref{AN:proposition:SQP-step}, to the resulting optimization problem.
\begin{definition}
Let $J_k\subseteq\lbrace1,\dots,k\rbrace$.
We define a quadratic approximation of $f$ resp.~$F$ in $y_j\in\mathbb{R}^n$ with damping parameter $\rho_j$ resp.~$\hat{\rho}_j\in[0,1]$ for $j\in J_k$ by
\begin{equation}
\begin{split}
f_j^{\sharp}(x)&:=f(y_j)+      g_j^T(x-y_j)+\tfrac{1}{2}\rho_j(x-y_j)^TG_j(x-y_j)
\\
F_j^{\sharp}(x)&:=F(y_j)+\hat{g}_j^T(x-y_j)+\tfrac{1}{2}\hat{\rho}_j(x-y_j)^T\hat{G}_j(x-y_j)
\label{Def:frauteCOMPACT}
\end{split}
\end{equation}
and the corresponding gradients by
\begin{equation}
g_j^{\sharp}(x)
:=
\nabla f_j^{\sharp}(x)^T
=
g_j+\rho_jG_j(x-y_j)
\komma\quad
\hat{g}_j^{\sharp}(x)
:=
\nabla F_j^{\sharp}(x)^T
=
\hat{g}_j+\hat{\rho}_j\hat{G}_j(x-y_j)
\label{Luksan:Satz:fjrauteGradientCOMPACT}
\punkt
\end{equation}
We define the piecewise quadratic approximation of $f$ resp.~$F$ in $x_k\in\mathbb{R}^n$ by
\begin{equation}
f_k^{\square}(x):=\max_{j\in J_k}{f_j^{\sharp}(x)}\komma\quad
F_k^{\square}(x):=\max_{j\in J_k}{F_j^{\sharp}(x)}\punkt     \label{Def:fquadratCOMPACT}
\end{equation}
\end{definition}
Hence we approximate the objective function $f$ at $x_k$ by $f_k^{\square}$ and the constraint $F$ at $x_k$ by $F_k^{\square}$ in the optimization problem (\ref{BundleSQP:OptProblem}) and then we perform a single SQP-step to the resulting optimization problem
\begin{equation}
\begin{split}
&\min_{x\in\mathbb{R}^n}{f_k^{\square}(x)}\\
&\textnormal{ s.t. }F_k^{\square}(x)\leq0\punkt
\label{BundleSQP:Idee:Optimierungsproblemfquadrat}
\end{split}
\end{equation}
It is important to observe here that the local model for the nonsmooth problem
(\ref{BundleSQP:OptProblem}) is the piecewise quadratic nonsmooth problem
(\ref{BundleSQP:Idee:Optimierungsproblemfquadrat}). This problem in turn can, however, be
equivalently written as a smooth QCQP.
\begin{proposition}
The SQP-step $(d,\hat{v})\in\mathbb{R}^{n+1}$ for \refh{BundleSQP:Idee:Optimierungsproblemfquadrat} is given by the solution of the QP
\begin{equation}
\begin{split}
&f(x_k)+\min_{d,\hat{v}}{\hat{v}+\tfrac{1}{2}d^TW^kd}\\
&\hspace{39pt}\textnormal{ s.t. }-\big(f(x_k)-f_j^k\big)+d^Tg_j^k\leq\hat{v}    \hspace{32pt}~~~\forall j\in J_k\\
&\hspace{39pt}\hphantom{\textnormal{ s.t. }}F(x_k)-\big(F(x_k)-F_j^k\big)+d^T\hat{g}_j^k\leq0~~~\forall j\in J_k\komma
\label{BundleSQP:Satz:SQPIterationsschrittProblem}
\end{split}
\end{equation}
where
\begin{equation}
\begin{split}
f_j^k&:=f _j^{\sharp}(x_k)
\komma\quad
g_j^k:=g _j^{\sharp}(x_k)\overset{\refh{Luksan:Satz:fjrauteGradientCOMPACT}}{=}
       g_j+\rho_jG_j(x_k-y_j)
\\
F _j^k&:=F _j^{\sharp}(x_k)
\komma\quad
\hat{g}_j^k:=\hat{g}_j^{\sharp}(x_k)\overset{\refh{Luksan:Satz:fjrauteGradientCOMPACT}}{=}
             \hat{g}_j+\hat{\rho}_j\hat{G}_j(x_k-y_j)
\\
W^k
&:=
\sum_{j\in J_{k-1}}{\lambda_j^{k-1}\rho_jG_j}
+
\sum_{j\in J_{k-1}}{\mu_j^{k-1}\hat{\rho}_j\hat{G}_j}
\label{BundleSQP:Def:gjkCOMPACT}
\end{split}
\end{equation}
and $\lambda_j^{k-1}$ resp.~$\mu_j^{k-1}$ denote the Lagrange multipliers with respect to $f$ resp.~$F$ at iteration $k-1$ for $j\in J_{k-1}$.
\end{proposition}
\begin{proof}
We rewrite (\ref{BundleSQP:Idee:Optimierungsproblemfquadrat}) as a smooth optimization problem by using (\ref{Def:fquadratCOMPACT}).
If we are at the iteration point $(x_k,u_k)\in\mathbb{R}^n\times\mathbb{R}$ with $u_k:=f(x_k)$ in this
smooth reformulation,
then,  according to (\ref{SQP:Satz:QP}) as well as using (\ref{BundleSQP:Def:gjkCOMPACT}), the SQP-step for this problem is given by the solution of the QP (\ref{BundleSQP:Satz:SQPIterationsschrittProblem}).
\qedhere
\end{proof}
Since $f_j^{\sharp}$ resp.~$F_j^{\sharp}$ are only global underestimators for convex $f$ resp.~$F$ and $\rho_j=\hat{\rho}_j=0$ and since $f_k^{\square}$ resp.~$F_k^{\square}$ approximate $f$ resp.~$F$ only well for trial points close to $x_k$, we decrease the activity of non local information (e.g., non local subgradients) by
the following definition.
\begin{definition}
We define the
localized
approximation errors of $f$ resp.~$F$ by
\begin{equation}
\alpha_j^k:=\max{\big(\lvert f(x_k)-f_j^k\rvert,\gamma_1(s_j^k)^{\omega_1}\big)}\komma\quad
     A_j^k:=\max{\big(\lvert F(x_k)-F_j^k\rvert,\gamma_2(s_j^k)^{\omega_2}\big)}\komma
\label{BundleSQP:Def:alphajkCOMPACT} 
\end{equation}
where
\begin{equation}
s_j^k
:=
\vert y_j-x_j\vert+\sum_{i=j}^{k-1}\vert x_{i+1}-x_i\vert
\label{Luksan:Def:Lokalitaetsmass}
\end{equation}
denotes a locality measure for $j=1,\dots,k$ with fixed parameters $\gamma_i>0$ and $\omega_i\geq1$ for $i=1,2$.
\end{definition}
\begin{proposition}
The locality measure $s_j^k$ has the following properties
\begin{equation}
s_j^k+\vert x_{k+1}-x_k\vert
=
s_j^{k+1}
\komma\quad
s_j^k
\geq
\vert y_j-x_k\vert~~\forall j=1,\dots,k
\punkt
\label{Luksan:Satz:LokalitaetsmassMonotonCOMPACT}
\end{equation}
\end{proposition}
\begin{proof}
Straightforward calculations.
\qedhere
\end{proof}
Like the bundle-Newton method by \citet{Luksan}, our algorithm uses a convex search direction problem and therefore we modify (\ref{BundleSQP:Satz:SQPIterationsschrittProblem}) in the following sense.
\begin{proposition}
If we generalize \refh{BundleSQP:Satz:SQPIterationsschrittProblem} by using the
localized
approximation errors \refh{BundleSQP:Def:alphajkCOMPACT} and replacing $W^k$ by a positive definite modification $\widebar{W}_p^k$ (e.g., the Gill-Murray factorization by \textH{\citet{GillMurray}}),
then the generalized version of \refh{BundleSQP:Satz:SQPIterationsschrittProblem} reads
\begin{equation}
\begin{split}
&f(x_k)+\min_{d,\hat{v}}{\hat{v}+\tfrac{1}{2}d^T\widebar{W}_p^kd}\\
&\hspace{39pt}\textnormal{ s.t. }                 -\alpha_j^k+d^Tg_j^k      \leq\hat{v}\hspace{30pt}~~~\forall j\in J_k\\
&\hspace{39pt}\hphantom{\textnormal{ s.t. }}F(x_k)-A_j^k     +d^T\hat{g}_j^k\leq0                   ~~~\forall j\in J_k
\punkt
\label{BundleSQP:Satz:ZuSQPIterationsschrittAequivalentesProblem}
\end{split}
\end{equation}
\end{proposition}
\begin{proof}
Replace $f(x_k)-f_j^k$ by $\alpha_j^k$, $F(x_k)-F_j^k$ by $A_j^k$ and $W^k$ by $\widebar{W}_p^k$ in (\ref{BundleSQP:Satz:SQPIterationsschrittProblem}).
\qedhere
\end{proof}
\begin{remark}
\label{BundleSQP:remark:Maratos}
The standard SQP approach for smooth optimization problems suffers from the Maratos effect \citet{Maratos}, which, in general, prevents infeasible SQP-methods from getting a descent in the merit function and feasible SQP-methods from finding (good) feasible points
(cf.~\citet[p.~1003]{Tits09} and Example \ref{example:Test28}). Some well known techniques for avoiding the Maratos effect are replacing the merit function by an augmented Lagrangian, using second order corrections, using a watchdog technique (which is a non-monotone line search)
(cf., e.g., \citet[p.~440,~15.5~The~Maratos~effect]{Nocedal}),
or a quadratic approximation of the constraints (cf.~\citet{Solodov}). We will choose the quadratic constrained approximation approach to avoid the Maratos effect, which makes the search direction problem slightly more difficult to solve than a QP, but, as we will see, still guarantees strong duality which is necessary for proving convergence of our bundle method.
\end{remark}
\begin{example}
\label{example:Test28}
Consider the optimization problem (\ref{BundleSQP:OptProblem}) with $f,F:\mathbb{R}^2\longrightarrow\mathbb{R}$, where $f(x):=x_2$ and $F(x):=x_1^2-x_2$. Then this problem has the (global) minimizer $\hat{x}=\zeroVector{2}$. Furthermore, it is smooth and consequently its SQP-direction, which is obtained by solving the QP (\ref{SQP:Satz:QP}), at the iteration $k=0$ at the iteration point $(x_k,\lambda_k):=(-1,1+10^{-8},1)$, which implies that $x_k$ is close to the boundary, is  given by $d_k=(1,-2)$. Since we have for $t\in[0,1]$ that $F(x_k+td_k)\leq0$ if and only if $t\leq10^{-4}$, a feasible SQP-method can only make a tiny step towards the solution $\hat{x}$ on the standard SQP-direction in this example, and similar observations can be made for any other point $x_k$ with $k\not=0$ that is close to the boundary (Note that the objective function $f$ has no impact on the Hessian of the Lagrangian in the QP (\ref{SQP:Satz:QP}) in this example).
\end{example}
\cbstartMIFFLIN
Remark \ref{BundleSQP:remark:Maratos} leads to the following idea:
Let $\widebar{G}_j^k,\widebar{\hat{G}}_j^k\in\Sym{n}$ be positive definite (e.g., positive definite modifications of $G_j\in\subhessian{f}{y_j}$ resp.~$\hat{G}_j\in\subhessian{F}{y_j}$; also cf.~Remark \ref{BundleSQPmitQCQP:GlobaleKonvergenz:Bem:Thm3.8}). Then we can try to determine the search direction by solving the convex QCQP
\begin{equation}
\begin{split}
&f(x_k)+\min_{d,\hat{v}}\hat{v}+\tfrac{1}{2}d^T\widebar{W}_p^kd\\
&\hspace{39pt}\textnormal{ s.t. }
-\alpha_j^k+d^Tg_j^k
+\tfrac{1}{2}d^T\widebar{G}_j^kd
\leq\hat{v}
\hspace{-53pt}
~~\hspace{86pt}\forall j\in J_k\\
&\hspace{39pt}\hphantom{\textnormal{ s.t. }}
F(x_k)-A_j^k+d^T\hat{g}_j^k+\tfrac{1}{2}d^T\widebar{\hat{G}}_j^kd\leq0~~~\forall j\in J_k
\label{Luksan:Alg:QCQPTeilproblem}
\end{split}
\end{equation}
instead of the QP (\ref{BundleSQP:Satz:ZuSQPIterationsschrittAequivalentesProblem}), i.e.~instead of just demanding that the first order approximations
are feasible, we demand that the first order approximations
must be the more feasible, the more we move away from $x_k$.
\cbendMIFFLIN
\begin{example}
\label{BundleSQP:example:CauseToAcceptOnlyStrictlyFeasibleIterationsPoints}
We consider the optimization problem (\ref{BundleSQP:OptProblem}) with
$f(x):=x_2$,
and
$F(x):=\max{(\min{(F_1(x),F_2(x))},F_3(x))}$,
where
$F_1(x):= x_1^2+x_2^2$,
$F_2(x):=-x_1+x_2^2$, and
$F_3(x):= x_1-2$,
and we assume that we are at the iteration point
$x_k:=\zeroVector{2}$.

Since
$\hat{F}(x):=\max{\big(F_2(x),F_3(x)\big)}$
is convex, and since an easy examination yields that
$F(x)\leq0
\Longleftrightarrow
\hat{F}(x)\leq0$,
the feasible set of our optimization problem (\ref{BundleSQP:OptProblem}) is convex. Therefore, the linearity of $f$ implies that our optimization problem has the unique minimizer $\hat{x}:=(2,-\sqrt{2})$.

The quadratic approximation of $F$ with respect to $x_k$ in the QCQP (\ref{Luksan:Alg:QCQPTeilproblem}) reads
$F_1(x_k+d)\leq0$, i.e.~$d=\zeroVector{2}$ is the only feasible point for the QCQP (\ref{Luksan:Alg:QCQPTeilproblem}) and therefore its solution, although $x_k=\zeroVector{2}$ is not a stationary point for our optimization problem (for this consider $f$), resp.~much less a minimizer (since $\hat{x}$ is the unique minimizer of our optimization problem).
As it can be seen, e.g., from considering the restriction of $F$ to $x_2=0$,
the reason for the occurrence of $d=\zeroVector{2}$ at $x_k$ is the nonconvexity of $F$ (which is a result of the presence of the $\min$-function in $F$), although the feasible set
is convex.

Notice that if we substitute 
$F$
by
$\hat{F}$
in the constraint of our optimization problem, which yields the same feasible set,
the difficulty which we described above does not occur.
\end{example}
\begin{remark}
If $F(x_k)\leq0$, ((\ref{BundleSQP:Satz:ZuSQPIterationsschrittAequivalentesProblem}) as well as) (\ref{Luksan:Alg:QCQPTeilproblem}) is always feasible and therefore we do not have to deal with infeasible search direction problems as they occur in infeasible SQP-methods (cf.~Remark \ref{SQP:Bem:LinearNBkonsistent}). Nevertheless, we have to demand $F(x_k)<0$, since otherwise it can happen that $d_k=\zeroVector{n}$ is the only feasible point and therefore the solution of (\ref{Luksan:Alg:QCQPTeilproblem}), but $x_k$ is not stationary for (\ref{BundleSQP:OptProblem}) as Example \ref{BundleSQP:example:CauseToAcceptOnlyStrictlyFeasibleIterationsPoints} showed. This is similar to difficulties arising in smooth problems at saddle points of the constraints.
\end{remark}
Now we state the dual search direction problem which plays an important role for proving the global convergence of the method
(cf.~Subsection \ref{section:GlobalConvergence}).
\begin{proposition}
The dual problem of the QCQP \refh{Luksan:Alg:QCQPTeilproblem} is given by
\begin{equation}
\begin{split}
&f(x_k)-\min_{\lambda,\mu}
\tfrac{1}{2}
\Big\lvert H_k(\lambda,\mu)\Big(\sum_{j\in J_k}\lambda_jg _j^k+\mu_j\hat{g}_j^k\Big)\Big\rvert^2
+\sum_{j\in J_k}\lambda_j\alpha_j^k+\mu_jA_j^k\\
&\hspace{0.65\textwidth}-\big(\sum_{j\in J_k}\mu_j\big)F(x_k)
\\
&\hspace{37pt}\textnormal{ s.t. }
\lambda_j\geq0\textnormal{, }
\mu_j\geq0~~\forall j\in J_k\textnormal{, }
\sum_{j\in J_k}\lambda_j=1\textnormal{, }
\label{Luksan:Alg:QCQPTeilproblem:DualesProblem}
\end{split}
\end{equation}
where
\cbstartMIFFLIN
$
H_k(\lambda,\mu)
:=
\big(\widebar{W}_p^k+\sum_{j\in J_k}
\lambda_j\widebar{G}_j^k+
\mu_j\widebar{\hat{G}}_j^k\big)^{-\mathalf}
$.
\cbendMIFFLIN
If $F(x_k)<0$, then the duality gap is zero, and, furthermore, if we denote the minimizer of the dual problem (\refh{Luksan:Alg:QCQPTeilproblem:DualesProblem}) by $(\lambda^k,\mu^k)$, then the minimizer $(d_k,\hat{v}_k)$ of the primal QCQP \refh{Luksan:Alg:QCQPTeilproblem} satisfies
\cbstartMIFFLIN
\begin{align*}
d_k
&=
-(\widebar{W}_p^k+\sum\limits_{j\in J_k}
\lambda_j^k\widebar{G}_j^k+
\mu_j^k\widebar{\hat{G}}_j^k)^{-1}
\big(
\sum_{j\in J_k}\lambda_j^kg_j^k+\mu_j^k\hat{g}_j^k
\big)
\\
\hat{v}_k
&=
\big(\sum_{j\in J_k}\lambda_j^kg_j^k\big)^Td_k
-\sum_{j\in J_k}\lambda_j^k\alpha_j^k
+\tfrac{1}{2}d_k^T\big(\sum\limits_{j\in J_k}{\lambda_j^k\widebar{G}_j^k}\big)d_k
\\
&=
-d_k^T\widebar{W}_p^kd_k
-\tfrac{1}{2}d_k^T\big(\sum_{j\in J_k}
\lambda_j^k\widebar{G}_j^k+
\mu_j^k\widebar{\hat{G}}_j^k\big)d_k
-\sum_{j\in J_k}\lambda_j^k\alpha_j^k
+\mu_j^kA_j^k\\
&\hspace{0.6\textwidth}
-\big(\sum_{j\in J_k}\mu_j^k\big)\big(-F(x_k)\big)
\leq
0
\punkt
\end{align*}
\cbendMIFFLIN
\end{proposition}
\begin{proof}
\cbstartMIFFLIN
The Lagrangian of (\ref{Luksan:Alg:QCQPTeilproblem}) is given by $L(d,\hat{v},\lambda,\mu):=\hat{v}+\tfrac{1}{2}d^T\widebar{W}_p^kd+\sum_{j\in J_k}{\lambda_jF_j^1(d,\hat{v})}+\sum_{j\in J_k}{\mu_jF_j^2(d,\hat{v})}$, where
$F_j^1(d,\hat{v}):=-\alpha_j^k+d^Tg_j^k+\tfrac{1}{2}d^T\widebar{G}_j^kd-\hat{v}$
and $F_j^2(d,\hat{v}):=F(x_k)-A_j^k+d^T\hat{g}_j^k+\tfrac{1}{2}d^T\widebar{\hat{G}}_j^kd$. Consequently, the equality constraint of the dual problem reads
\begin{equation}
\Big(\widebar{W}_p^k+\sum_{j\in J_k}{
\lambda_j\widebar{G}_j^k+
\mu_j\widebar{\hat{G}}_j^k
}\Big)d
+\sum_{j\in J_k}{\lambda_jg_j^k+\mu_j\bar{\hat{g}}_j^k}
=0
\textnormal{ , }
\sum_{j\in J_k}{\lambda_j}=1
\punkt
\label{Luksan:Alg:QCQPTeilproblem:DualesProblem:Bew1}
\end{equation}
\cbendMIFFLIN
Rewriting $\tfrac{1}{2}d^T\widebar{W}_p^kd=-\tfrac{1}{2}d^T\widebar{W}_p^kd+d^T\widebar{W}_p^kd$ in $L$, scooping $d$ in the latter summand and $\hat{v}$, these terms vanish according to (\ref{Luksan:Alg:QCQPTeilproblem:DualesProblem:Bew1}). Now, expressing $d$ in (\ref{Luksan:Alg:QCQPTeilproblem:DualesProblem:Bew1}) and inserting it into $L$ yield the desired form of the dual objective function.

Since the primal problem is convex and (because of the assumption $F(x_k)<0$) strictly feasible, strong duality holds due to \citet[Section~5.2.3]{Boyd}. Therefore the optimal primal and dual objective function values coincide and we can express $\hat{v}_k$ using this equality. Using (\ref{Luksan:Alg:QCQPTeilproblem:DualesProblem:Bew1}), the optimality conditions for the QCQP (\ref{Luksan:Alg:QCQPTeilproblem}) and straightforward calculations yield the desired formulas for $\hat{v}_k$.
\qedhere
\end{proof}

\subsection{Presentation of the algorithm}
The method described in Algorithm \ref{BundleSQPmitQCQP:Alg:GesamtAlgMitQCQP} works according to the following scheme: After choosing a strictly feasible starting point $x_1\in\mathbb{R}^n$ and setting up a few positive definite matrices, we compute the
localized
approximation errors. Then we solve a convex QCQP to determine the search direction, where the
quadratic constraints of the QCQP serve to obtain preferably feasible points that yield a good descent. After computing the aggregated data and the predicted descent as well as testing the termination criterion, we perform a line search (s.~Algorithm \ref{BundleSQP:AlgNB:LinesearchMitQCQP}) on the ray given by the search direction. This yields a trial point $y_{k+1}$ that has the following property: Either $y_{k+1}$ is strictly feasible and the objective function achieves sufficient descent (serious step) or $y_{k+1}$ is strictly feasible and the model of the objective function changes sufficiently (null step with respect to the objective function) or $y_{k+1}$ is not strictly feasible and the model of the constraint changes sufficiently (null step with respect to the constraint). Afterwards we update the iteration point $x_{k+1}$ and the information stored in the bundle. Now, we repeat this procedure until the termination criterion is satisfied.
\begin{algorithm}
\label{BundleSQPmitQCQP:Alg:GesamtAlgMitQCQP}
\begin{enumerate}
\addtocounter{enumi}{-1}
\item\textit{Initialization:} Choose the following parameters, which will not be changed during the algorithm:
\begin{longtable}{l|l|p{6.6cm}}
\caption{Initial parameters}\\
\multicolumn{1}{c|}{\textnormal{General}}     & 
\multicolumn{1}{c|}{\textnormal{Default}}     & 
\multicolumn{1}{c }{\textnormal{Description}}  \\
\hline
\endfirsthead
\caption[]{Initial parameters (continued)}\\
\multicolumn{1}{c|}{\textnormal{General}}     & 
\multicolumn{1}{c|}{\textnormal{Default}}     & 
\multicolumn{1}{c }{\textnormal{Description}}  \\
\hline
\endhead
$x_1\in\mathbb{R}^n$       &                   & Strictly feasible initial point                    \\
$y_1=x_1$                  &                   & Initial trial point                                \\
$\varepsilon\geq0$         &                   & Final optimality tolerance                         \\
$M\geq2$                   & $M=n+3$           & Maximal bundle dimension                           \\
$t_0\in(0,1)$              & $t_0=0.001$       & Initial lower bound for step size                  \\
                           &                   & of serious step in line search                     \\
$\hat{t}_0\in(0,1)$        & $\hat{t}_0=0.001$ & Scaling parameter for $t_0$                        \\
$m_L\in(0,\tfrac{1}{2})$   & $m_L=0.01$        & Descent parameter for serious step in line search  \\
\cbstartMIFFLIN
$m_R\in(m_L,1)$,           & $m_R=0.5$    & Parameter for change of model of objective\\
$m_f\in[0,1]$              &                   & function for short serious and null steps in line search    \cbendMIFFLIN\\
$m_F\in(0,1)$              & $m_F=0.01$        & Parameter for change of model of constraint        \\
                           &                   & for short serious and null steps in line search    \\
$\zeta\in(0,\tfrac{1}{2})$ & $\zeta=0.01$      & Coefficient for interpolation in line search       \\
$\vartheta\geq1$           & $\vartheta=1$     & Exponent for interpolation in line search          \\
$C_S>0$                    & $C_S=10^{50}$     & Upper bound of the distance between $x_k$ and $y_k$\\
$C_G>0$                    & $C_G=10^{50}$     & Upper bound of the norm of the damped              \\
                           &                   & matrices $\lbrace\rho_jG_j\rbrace$ ($\lvert\rho_jG_j\rvert\leq C_G$)\\
$\hat{C}_G>0$              & $\hat{C}_G=C_G$   & Upper bound of the norm of the damped              \\
                           &                   & matrices $\lbrace\hat{\rho}_j\hat{G}_j\rbrace$ ($\lvert\hat{\rho}_j\hat{G}_j\rvert\leq \hat{C}_G$)\\
\cbstartMIFFLIN
$\bar{C}_G>0$              & $\bar{C}_G=C_G$       & Upper bound of the norm of the matrices\\
                           &                       & 
$\lbrace\bar{G}_j^k\rbrace$ and $\lbrace\bar{G}^k\rbrace$ ($\max{(\lvert\bar{G}_j^k\rvert,\lvert\bar{G}^k\rvert)}\leq\bar{C}_G$)
\cbendMIFFLIN\\
$\bar{\hat{C}}_G>0$        & $\bar{\hat{C}}_G=C_G$ & Upper bound of the norm of the matrices\\
                           &                       & 
$\lbrace\bar{\hat{G}}_j^k\rbrace$ and $\lbrace\bar{\hat{G}}^k\rbrace$ ($\max{(\lvert\bar{\hat{G}}_j^k\rvert,\lvert\bar{\hat{G}}^k\rvert)}\leq\bar{\hat{C}}_G$)\\
$i_{\rho}\geq0$            & $i_{\rho}=3$          & Selection parameter for $\rho_{k+1}$ (cf.~Remark \ref{remark:BundleSQPmitQCQP:Parametersimin})\\
$i_l\geq0$                 &                       & Line search selection parameter (cf.~Remark \ref{remark:BundleSQPmitQCQP:Parametersimin})\\
$i_m\geq0$                 &                       & Matrix selection parameter (cf.~Remark \ref{remark:BundleSQPmitQCQP:Parametersimin})\\
$i_r\geq0$                 &                       & Bundle reset parameter (cf.~Remark \ref{remark:BundleSQPmitQCQP:Parametersimin})\\
$\gamma_1>0$               & $\gamma_1=1$      & Coefficient for locality measure for objective function\\
$\gamma_2>0$               & $\gamma_2=1$      & Coefficient for locality measure for constraint        \\
$\omega_1\geq1$            & $\omega_1=2$      & Exponent for locality measure for objective function   \\
$\omega_2\geq1$            & $\omega_2=2$      & Exponent for locality measure for constraint
\end{longtable}
Set the initial values of the data which gets changed during the algorithm:
\begin{align*}
i_n&=\hphantom{\lbrace}0\hphantom{\rbrace}\textnormal{ (\# subsequent null and short steps)}\\
i_s&=\hphantom{\lbrace}0\hphantom{\rbrace}\textnormal{ (\# subsequent serious steps)}       \\
J_1&=          \lbrace 1          \rbrace \textnormal{ (set of bundle indices)}             \punkt
\end{align*}
Compute the following information at the initial trial point
\begin{align}
            f_p^1=f_1^1&=f(y_1)                                     \label{BundleSQPmitQCQP:Alg:fStartwert}\\
            g_p^1=g_1^1&=g(y_1)\in\partial f(y_1)                   \label{BundleSQPmitQCQP:Alg:gStartwert}\\
              G_p^1=G_1& \text{ approximating } G(y_1)\in\subhessian{f}{y_1}               \label{BundleSQPmitQCQP:Alg:GpStartwert}\\
            F_p^1=F_1^1&=F(y_1)<0~~~\textnormal{(}y_1\textnormal{ is strictly feasible according to assumption)}
\label{BundleSQPmitQCQP:Alg:FStartwert}\\
\hat{g}_p^1=\hat{g}_1^1&=\hat{g}(y_1)\in\partial F(y_1)             \label{BundleSQPmitQCQP:Alg:ghatStartwert}\\
  \hat{G}_p^1=\hat{G}_1& \text{ approximating } \hat{G}(y_1)\in\subhessian{F}{y_1}         \label{BundleSQPmitQCQP:Alg:GhatpStartwert}
\end{align}
and set
\begin{align}
\hat{s}_p^1=s_p^1=s_1^1&=0\textnormal{ (locality measure)}
\label{BundleSQPmitQCQP:Alg:sStartwert}\\
                \hat{\rho}_1=\rho_1&=1\textnormal{ (damping parameter)}                                 \nonumber\\
                     \bar{\kappa}^1&=1\textnormal{ (Lagrange multiplier for optimality condition)}      \nonumber\\
                                  k&=1\textnormal{ (iterator)}                                    \punkt\nonumber
\end{align}
\item\textit{Determination of the matrices for the QCQP:}

\texttt{if} (step $k-1$ and $k-2$ were serious steps) $\wedge$ ($\lambda_{k-1}^{k-1}=1$ $\vee$ $\underset{\textnormal{bundle reset}}{\underbrace{i_s>i_r}}$)
\begin{equation}
W=G_k+\bar{\kappa}^k\hat{G}_k
\label{BundleSQPmitQCQP:Alg:WchoiceSuperlinear1}
\end{equation}
\texttt{else}
\begin{equation}
W=G_p^k+\bar{\kappa}^k\hat{G}_p^k
\label{BundleSQPmitQCQP:Alg:WchoiceSuperlinear2}
\end{equation}
\texttt{end}\\
\\
\texttt{if} $i_n\leq i_m+i_l$\\
\hphantom{aaa}$\widebar{W}_p^k=\textnormal{``positive definite modification of }W\textnormal{''}$\\
\texttt{else}\\
\begin{equation}
\widebar{W}_p^k      =\widebar{W}_p^{k-1}
\label{BundleSQPmitQCQP:Alg:KeineModifikationVonGpk}
\end{equation}
\texttt{end}\\
\\
\cbstartDVI
\texttt{if} $i_n<i_m+i_l$
(i.e.~\# of subsequent null and short steps $<$ the fixed number $i_m+i_l$)
\begin{equation}
\begin{split}
(\widebar{G}^k,\widebar{\hat{G}}^k)    &=\textnormal{``positive definite modification of }(G_p^k,\hat{G}_p^k)\textnormal{''}\\
(\widebar{G}_j^k,\widebar{\hat{G}}_j^k)&=\textnormal{``positive definite modification of }(G_j,\hat{G}_j)\textnormal{''}~\forall j\in J_k
\label{BundleSQPmitQCQP:Alg:ModifikationVonGhatbarjk}
\end{split}
\end{equation}
\texttt{else if} $i_n=i_m+i_l$\\
\begin{equation}
\begin{split}
(\widebar{G}^k,\widebar{\hat{G}}^k)    &=\textnormal{``positive definite modification of }(G_p^k,\hat{G}_p^k)\textnormal{''}\\
(\widebar{G}_j^k,\widebar{\hat{G}}_j^k)&=(\widebar{G}^k,\widebar{\hat{G}}^k)~\forall j\in J_k
\label{BundleSQPmitQCQP:Alg:ModifikationVonGhatbarjkBeiGleichheitsCase}
\end{split}
\end{equation}
\texttt{else}
(i.e.~at least $i_m+i_l$ subsequent null and short steps were executed)\\
\begin{equation}
(\widebar{G}^k,\widebar{\hat{G}}^k)    =(\widebar{G}^{k-1},\widebar{\hat{G}}^{k-1})\komma\quad
(\widebar{G}_j^k,\widebar{\hat{G}}_j^k)=(\widebar{G}^{k-1},\widebar{\hat{G}}^{k-1})~\forall j\in J_k
\label{BundleSQPmitQCQP:Alg:KeineModifikationVonGhatbarjk}
\end{equation}
\texttt{end}
\cbendDVI
\item\textit{Computation of the
localized
approximation errors:}
\begin{align}
\alpha_j^k&:=\max{\big(\lvert f(x_k)-f_j^k\rvert,\gamma_1(s_j^k)^{\omega_1}\big)}
\komma\quad
\alpha_p^k:=\max\big(\lvert f(x_k) -f_p^k\rvert,\gamma_1(s_p^k)^{\omega_1}\big)
\label{BundleSQPmitQCQP:Alg:alphajkCOMPACT}
\\
A_j^k&:=\max{\big(\lvert F(x_k)-F_j^k\rvert,\gamma_2(s_j^k)^{\omega_2}\big)}
\komma\quad
A_p^k:=\max\big(\lvert F(x_k) -F_p^k\rvert,\gamma_2(\hat{s}_p^k)^{\omega_2}\big)\punkt
\label{BundleSQPmitQCQP:Alg:AjkCOMPACT}
\end{align}
\item\textit{Determination of the search direction:} Compute the solution $(d_k,\hat{v}_k)\in\mathbb{R}^{n+1}$ of the (convex) QCQP
\begin{equation}
\begin{split}
&\min_{d,\hat{v}}\hat{v}+\tfrac{1}{2}d^T\widebar{W}_p^kd\komma\\
&\textnormal{ s.t. }           -\alpha_j^k+d^Tg_j^k
+\tfrac{1}{2}d^T\widebar{G}_j^kd
\leq\hat{v}~~\hspace{82pt}\textnormal{for }j\in J_k
\cbmathMIFFLIN\\
&\hphantom{\textnormal{ s.t. }}-\alpha_p^k+d^Tg_p^k
+\tfrac{1}{2}d^T\widebar{G}^kd
\leq\hat{v}~~\hspace{82pt}\textnormal{if }i_s\leq i_r
\cbmathMIFFLIN\\
&\hphantom{\textnormal{ s.t. }}F(x_k)-A_j^k+d^T\hat{g}_j^k+\tfrac{1}{2}d^T\widebar{\hat{G}}_j^kd\leq0
~~\textnormal{for }j\in J_k\\
&\hphantom{\textnormal{ s.t. }}F(x_k)-A_p^k+d^T\hat{g}_p^k+\tfrac{1}{2}d^T\widebar{\hat{G}}^kd\leq0
~~\textnormal{if }i_s\leq i_r
\label{BundleSQPmitQCQP:Alg:QPTeilproblem}
\end{split}
\end{equation}
and its corresponding Lagrange multiplier $(\lambda^k,\lambda_p^k,\mu^k,\mu_p^k)\in\Rpos^{2(\lvert J_k\rvert+1)}$, i.e.
\begin{align}
d_k
&=
-H_k^2
\Big(
\sum_{j\in J_k}\lambda_j^k     g _j^k+\lambda_p^k     g _p^k+
\sum_{j\in J_k}    \mu_j^k\hat{g}_j^k+\mu    _p^k\hat{g}_p^k
\Big)
\label{BundleSQPmitQCQP:Alg:dk}
\\
\hat{v}_k
&=
-d_k^T\widebar{W}_p^kd_k
-\tfrac{1}{2}d_k^T\big(\sum_{j\in J_k}
\lambda_j^k\widebar{G}_j^k+\lambda_p^k\widebar{G}^k+
\mu_j^k\widebar{\hat{G}}_j^k+\mu_p^k\widebar{\hat{G}}^k
\big)d_k
-\sum_{j\in J_k}\lambda_j^k\alpha_j^k
\cbmathMIFFLIN
\nonumber
\\
&\hphantom{=}\hspace{4pt}
-\lambda_p^k\alpha_p^k-\sum_{j\in J_k}\mu_j^kA_j^k-\mu_p^kA_p^k
-\big(\sum_{j\in J_k}\mu_j^k+\mu_p^k\big)\big(-F(x_k)\big)
\komma
\label{BundleSQPmitQCQP:Alg:vhatkNEW}
\end{align}
where
\begin{equation}
H_k:=
\big(\widebar{W}_p^k+\sum_{j\in J_k}
\lambda_j^k\widebar{G}_j^k+\lambda_p^k\widebar{G}^k+
\mu_j^k\widebar{\hat{G}}_j^k+\mu_p^k\widebar{\hat{G}}^k
\big)
^{-\mathalf}
\punkt
\cbmathMIFFLIN
\label{BundleSQPmitQCQP:Alg:Hk}
\end{equation}
Set
\begin{equation}
\bar{\kappa}^{k+1}
:=
\sum_{j\in J_k}\mu_j^k+\mu_p^k
\label{BundleSQPmitQCQP:Alg:Defkappabark+1}
\komma\quad
(\kappa_j^k,\kappa_p^k)
:=
\left\lbrace
\begin{array}{ll}
\tfrac{(\mu_j^k,\mu_p^k)}{\bar{\kappa}^{k+1}} & \textnormal{for }\bar{\kappa}^{k+1}>0\\
0                                   & \textnormal{for }\bar{\kappa}^{k+1}=0
\end{array}
\right.
\end{equation}
\texttt{if} $i_s>i_r$\\
\hphantom{aaa}$i_s=0$ (bundle reset)\\
\texttt{end}
\item\textit{Aggregation:} We set for the aggregation of information of the objective function
\begin{align}
(\tilde{g}_p^k,\tilde{f}_p^k,G_p^{k+1},\tilde{s}_p^k)&=\sum_{j\in J_k}\lambda_j^k(g_j^k,f_j^k,\rho_jG_j,s_j^k)+\lambda_p^k(g_p^k,f_p^k,G_p^k,s_p^k )          \label{BundleSQPmitQCQP:Alg:gtildepkCOMPACT}
\\
\tilde{\alpha}_p^k&=\max\big(\vert f(x_k)-\tilde{f}_p^k\vert,\gamma_1(\tilde{s}_p^k)^{\omega_1}\big)
\label{BundleSQPmitQCQP:Alg:alphatildepk}
\end{align}
and for the aggregation of information of the constraint
\begin{align}
(\tilde{\hat{g}}_p^k,\tilde{F}_p^k,\hat{G}_p^{k+1},\tilde{\hat{s}}_p^k)&=\sum_{j\in J_k}\kappa_j^k(\hat{g}_j^k,F_j^k,\hat{\rho}_j\hat{G}_j,s_j^k)+\kappa_p^k(\hat{g}_p^k,F_p^k,\hat{G}_p^k,\hat{s}_p^k)
\label{BundleSQPmitQCQP:Alg:ghattildepkCOMPACT}
\\
\tilde{A}_p^k&=\max\big(\vert F(x_k)-\tilde{F}_p^k\vert,\gamma_2(\tilde{\hat{s}}_p^k)^{\omega_2}\big)
\label{BundleSQPmitQCQP:Alg:Atildepk}
\end{align}
and we set
\begin{align}
v_k
&=
-d_k^T\widebar{W}_p^kd_k
-\tfrac{1}{2}d_k^T\big(\sum_{j\in J_k}
\lambda_j^k\widebar{G}_j^k+\lambda_p^k\widebar{G}^k+
\mu_j^k\widebar{\hat{G}}_j^k+\mu_p^k\widebar{\hat{G}}^k
\big)d_k
\nonumber
\\
&\hphantom{=}\hspace{4pt}
-\tilde{\alpha}_p^k-\bar{\kappa}^{k+1}\tilde{A}_p^k
-\bar{\kappa}^{k+1}\big(-F(x_k)\big)
\cbmathMIFFLIN
\label{BundleSQPmitQCQP:Alg:vk}\\
w_k
&=
\tfrac{1}{2}\vert H_k
(
\tilde{g}_p^k+\bar{\kappa}^{k+1}\tilde{\hat{g}}_p^k
)
\vert^2
+\tilde{\alpha}_p^k+\bar{\kappa}^{k+1}\tilde{A}_p^k
+\bar{\kappa}^{k+1}\big(-F(x_k)\big)
\punkt
\label{BundleSQPmitQCQP:Alg:wk}
\end{align}
\item\textit{Termination criterion:}

\texttt{if} $w_k\leq\varepsilon$\\
\hphantom{aaa}\texttt{stop}\\
\texttt{end}

\item\textit{Line search:} We compute step sizes $0\leq t_L^k\leq t_R^k\leq1$ and $t_0^k\in(0,t_0]$ by using the line search described in Algorithm \ref{BundleSQP:AlgNB:LinesearchMitQCQP}
and we set
\begin{align}
      x_{k+1}&=x_k+t_L^kd_k~~~\textnormal{(is created strictly feasible by the line search)}
                                                       \label{BundleSQPmitQCQP:Alg:xk+1Update}\\
      y_{k+1}&=x_k+t_R^kd_k                            \label{BundleSQPmitQCQP:Alg:yk+1Update}\\
      f_{k+1}&=f(y_{k+1}))\komma~  
      g_{k+1}=g(y_{k+1})\in\partial f(y_{k+1}))\komma\nonumber\\
      G_{k+1}&\text{ approximating }G(y_{k+1})\in\subhessian{f}{y_{k+1}}     \label{BundleSQPmitQCQP:Alg:fk+1Update}             \\
      F_{k+1}&=F(y_{k+1})\komma~
  \hat{g}_{k+1}=\hat{g}(y_{k+1})\in\partial F(y_{k+1})\komma\nonumber\\
  \hat{G}_{k+1}&\text{ approximating }\hat{G}(y_{k+1})\in\subhessian{F}{y_{k+1}}\punkt\nonumber
\end{align}
\item\textit{Update:}

\texttt{if} $i_n\leq i_{\rho}$\\
\begin{equation}
\rho_{k+1}=\min(1,\tfrac{C_G}{\vert G_{k+1}\vert})
\label{BundleSQPmitQCQP:Alg:rhok+1}
\end{equation}
\texttt{else}\\
$\hphantom{aaa}\rho_{k+1}=0$\\
\texttt{end}\\
We set
\begin{equation}
\hat{\rho}_{k+1}=\min(1,\tfrac{\hat{C}_G}{\vert\hat{G}_{k+1}\vert})
\punkt
\label{BundleSQPmitQCQP:Alg:rhohatk+1}
\end{equation}
\texttt{if} $t_L^k\geq t_0^k$ (serious step)\\
\hphantom{aaa}$i_n=0$            \\
\hphantom{aaa}$i_s=i_s+1$        \\
\texttt{else} (no serious step, i.e.~null or short step)
\begin{equation}
i_n=i_n+1
\label{BundleSQPmitQCQP:Alg:inUpdate}
\end{equation}
\texttt{end}\\
\\
Compute the updates of the locality measure
\begin{align}
    s_j^{k+1}&=s_j^k+\vert x_{k+1}-x_k\vert~~~\textnormal{for }j\in J_k\label{BundleSQPmitQCQP:Alg:sjk+1Update} \\
s_{k+1}^{k+1}&=\vert x_{k+1}-y_{k+1}\vert                              \label{BundleSQPmitQCQP:Alg:sk+1k+1Update}\\
    s_p^{k+1}&=\tilde{s}_p^k+\vert x_{k+1}-x_k\vert                    \label{BundleSQPmitQCQP:Alg:spk+1Update}
\\
    \hat{s}_p^{k+1}&=\tilde{\hat{s}}_p^k+\vert x_{k+1}-x_k\vert\punkt
\label{BundleSQPmitQCQP:Alg:shatpk+1Update}
\end{align}
Compute the updates for the objective function approximation
\begin{align}
f_j^{k+1}
&=
f_j^k+g_j^{k\,T}(x_{k+1}-x_k)+\tfrac{1}{2}\rho_j(x_{k+1}-x_k)^TG_j(x_{k+1}-x_k)~~~\textnormal{for }j\in J_k
\nonumber
\\
f_{k+1}^{k+1}
&=
f_{k+1}+g_{k+1}^T(x_{k+1}-y_{k+1})+\tfrac{1}{2}\rho_{k+1}(x_{k+1}-y_{k+1})^TG_{k+1}(x_{k+1}-y_{k+1})
\label{BundleSQPmitQCQP:Alg:fk+1k+1Update}
\\
f_p^{k+1}
&=
\tilde{f}_p^k+\tilde{g}_p^{k\,T}(x_{k+1}-x_k)+\tfrac{1}{2}(x_{k+1}-x_k)^TG_p^{k+1}(x_{k+1}-x_k)
\label{BundleSQPmitQCQP:Alg:fpk+1Update}
\end{align}
and for the constraint
\begin{align}
F_j^{k+1}
&=
F_j^k+\hat{g}_j^{k\,T}(x_{k+1}-x_k)+\tfrac{1}{2}\hat{\rho}_j(x_{k+1}-x_k)^T\hat{G}_j(x_{k+1}-x_k)
~~~\textnormal{for }j\in J_k
\nonumber\\
F_{k+1}^{k+1}
&=
F_{k+1}+\hat{g}_{k+1}^T(x_{k+1}-y_{k+1})+\tfrac{1}{2}\hat{\rho}_{k+1}(x_{k+1}-y_{k+1})^T\hat{G}_{k+1}(x_{k+1}-y_{k+1})
\label{BundleSQPmitQCQP:Alg:Fk+1k+1Update}\\
F_p^{k+1}
&=
\tilde{F}_p^k+\tilde{\hat{g}}_p^{k\,T}(x_{k+1}-x_k)+\tfrac{1}{2}(x_{k+1}-x_k)^T\hat{G}_p^{k+1}(x_{k+1}-x_k)
\punkt
\label{BundleSQPmitQCQP:Alg:Fpk+1Update}
\end{align}
Compute the updates for the subgradient of the objective function approximation
\begin{align}
g_j^{k+1}
&=
g_j^k+\rho_jG_j(x_{k+1}-x_k)~~~\textnormal{for }j\in J_k
\nonumber
\\
g_{k+1}^{k+1}
&=
g_{k+1}+\rho_{k+1}G_{k+1}(x_{k+1}-y_{k+1})
\label{BundleSQPmitQCQP:Alg:gk+1k+1Update}
\\
g_p^{k+1}
&=
\tilde{g}_p^k+G_p^{k+1}(x_{k+1}-x_k)
\label{BundleSQPmitQCQP:Alg:gpk+1Update}
\end{align}
and for the constraint
\begin{align}
\hat{g}_j^{k+1}
&=
\hat{g}_j^k+\hat{\rho}_j\hat{G}_j(x_{k+1}-x_k)
~~~\textnormal{for }j\in J_k
\label{BundleSQPmitQCQP:Alg:ghatjk+1Update}
\\
\hat{g}_{k+1}^{k+1}
&=
\hat{g}_{k+1}+\hat{\rho}_{k+1}\hat{G}_{k+1}(x_{k+1}-y_{k+1})
\label{BundleSQPmitQCQP:Alg:ghatk+1k+1Update}
\\
\hat{g}_p^{k+1}
&=
\tilde{\hat{g}}_p^k+\hat{G}_p^{k+1}(x_{k+1}-x_k)
\punkt
\label{BundleSQPmitQCQP:Alg:ghatpk+1Update}
\end{align}
Choose $J_{k+1}\subseteq\lbrace k-M+2,\dots,k+1\rbrace\cap\lbrace1,2,\dots\rbrace$ with $k+1\in J_{k+1}$.\\
$k=k+1$\\
Go to 1
\end{enumerate}
\end{algorithm}
\begin{remark}
\label{remark:BundleSQPmitQCQP:Parametersimin}
We will see later that for convergence the approximation of element in
$\subhessian f{y}$ and $\subhessian F{y}$ only needs to satisfy mild conditions. The speed of
convergence will, of course, be influenced by the quality of approximation. In our first
implementation of the method \citet{HannesPaperA} we have computed elements of the respective
sets, but update methods similar to L-BFGS are also conceivable.

Like in the original unconstrained bundle-Newton method by \citet{Luksan}, the parameters $i_m$ and $i_r$ as well as the additional parameter $i_l$ are only needed for proving convergence. Since in practice we usually terminate an algorithm, if a maximal number of iterations $\texttt{Nit\_max}$ is exceeded, we always choose
$
i_m=i_n=i_l=\texttt{Nit\_max}+1
$
in our implementation of Algorithm \ref{BundleSQPmitQCQP:Alg:GesamtAlgMitQCQP}.
The case distinction for the choice of $W$ according to (\ref{BundleSQPmitQCQP:Alg:WchoiceSuperlinear1}) resp.~(\ref{BundleSQPmitQCQP:Alg:WchoiceSuperlinear2}) is only necessary for showing the superlinear convergence of the 
original unconstrained bundle-Newton method for strongly convex, twice times continuously differentiable functions (cf.~\citet[p.~385,~Section~4]{Luksan}).
As the choice $i_{\rho}=3$ (cf.~the initialization of Algorithm \ref{BundleSQPmitQCQP:Alg:GesamtAlgMitQCQP}) for the case distinction $i_n\leq i_{\rho}$ for $\rho_{k+1}$ from (\ref{BundleSQPmitQCQP:Alg:rhok+1}) is due to empirical observations in the original unconstrained bundle-Newton method (cf.~\citet[p.~378]{Luksan}), the fact that we make no case distinction for $\hat{\rho}_{k+1}$ from (\ref{BundleSQPmitQCQP:Alg:rhohatk+1}) was also found out numerically.
A numerically meaningful choice of the matrices
\cbstartMIFFLIN
$\widebar{G}_j^k$,
\cbendMIFFLIN
$\widebar{\hat{G}}_j^k$ and $\widebar{\hat{G}}^k$ that occur in (\ref{BundleSQPmitQCQP:Alg:ModifikationVonGhatbarjk}) is discussed in \citet[\GenaueAngabeThree]{HannesPaperA}.
\end{remark}
\begin{proposition}
We have for all $k\geq0$
\begin{align}
\vert H_k
(
\tilde{g}_p^k+\bar{\kappa}^{k+1}\tilde{\hat{g}}_p^k
)
\vert^2
&=
d_k^T\big(
\widebar{W}_p^k
+\sum_{j\in J_k}\lambda_j^k\widebar{G}_j^k+\lambda_p^k\widebar{G}^k
+\sum_{j\in J_k}\mu_j^k\widebar{\hat{G}}_j^k+\mu_p^k\widebar{\hat{G}}^k
\big)d_k
\cbmathMIFFLIN
\label{proposition:ConnectionBetweenHkAndWpkNorm}
\\
w_k
&=
-\tfrac{1}{2}d_k^T\widebar{W}_p^kd_k
-v_k
\punkt
\label{proposition:wkvkZusammenhang}
\end{align}
\end{proposition}
\begin{proof}
Because of
\cbstartMIFFLIN
$H_k^{-2}=\widebar{W}_p^k+\sum_{j\in J_k}
\lambda_j^k\widebar{G}_j^k+\lambda_p^k\widebar{G}^k
\mu_j^k\widebar{\hat{G}}_j^k+\mu_p^k\widebar{\hat{G}}^k$
\cbendMIFFLIN
due to (\ref{BundleSQPmitQCQP:Alg:Hk}) and
$d_k=-H_k^2(\tilde{g}_p^k+\bar{\kappa}^{k+1}\tilde{\hat{g}}_p^k)$
due to (\ref{BundleSQPmitQCQP:Alg:dk}),
(\ref{BundleSQPmitQCQP:Alg:Defkappabark+1}),
(\ref{BundleSQPmitQCQP:Alg:gtildepkCOMPACT}) and
(\ref{BundleSQPmitQCQP:Alg:ghattildepkCOMPACT}), easy calculations yield (\ref{proposition:ConnectionBetweenHkAndWpkNorm}).
Furthermore, (\ref{proposition:wkvkZusammenhang}) holds due to
(\ref{BundleSQPmitQCQP:Alg:wk}),
(\ref{proposition:ConnectionBetweenHkAndWpkNorm}), and
(\ref{BundleSQPmitQCQP:Alg:vk}).
\qedhere
\end{proof}
\begin{remark}
If we consider a nonsmooth unconstrained optimization problem (i.e.~we drop the constraint $F(x)\leq0$ in optimization problem (\ref{BundleSQP:OptProblem}))
\cbstartMIFFLIN
and if we choose $\widebar{G}_j^k=\zeroMatrix{n}$,
\cbendMIFFLIN
then our formula for $v_k$ from (\ref{BundleSQPmitQCQP:Alg:vk}) reduces to the formula for $v_k$ in the unconstrained bundle-Newton method (cf.~\citet[p.~377,~formula~(13)]{Luksan}),
since
$
v_k
=
-\vert H_k
\tilde{g}_p^k
\vert^2
-\tilde{\alpha}_p^k
$
due to
(\ref{BundleSQPmitQCQP:Alg:vk}) and
(\ref{proposition:ConnectionBetweenHkAndWpkNorm}).
\end{remark}

\subsection{Presentation of the line search}
We extend the line search of the bundle-Newton method for nonsmooth unconstrained minimization to the constrained case in the line search described in Algorithm \ref{BundleSQP:AlgNB:LinesearchMitQCQP}. For obtaining a clear arrangement of the line search, we compute data concerning the objective function in \texttt{ComputeObjectiveData} and data concerning the constraint in \texttt{ComputeConstraintData}. Before formulating the line search in detail, we give a brief overview of its functionality:

Starting with the step size $t=1$, we check if the point $x_k+td_k$ is strictly feasible.
If so and if additionally the objective function decreases sufficiently in this point and $t$ is not too small, then we take $x_k+td_k$ as new iteration point in Algorithm \ref{BundleSQPmitQCQP:Alg:GesamtAlgMitQCQP} (serious step).
Otherwise, if the point $x_k+td_k$ is strictly feasible and the model of the objective function changes sufficiently, we take $x_k+td_k$ as new trial point (short/null step with respect to the objective function).
If $x_k+td_k$ is not strictly feasible, but the model of the constraint changes sufficiently (in particular here the quadratic approximation of the constraint comes into play), we take $x_k+td_k$ as new trial point (short/null step with respect to the constraint).
After choosing a new step size $t\in[0,1]$ by interpolation, we iterate this procedure.
\begin{algorithm}[Line search]
\label{BundleSQP:AlgNB:LinesearchMitQCQP}
\begin{enumerate}
\addtocounter{enumi}{-1}
\item\textit{Initialization:} Choose $\zeta\in(0,\tfrac{1}{2})$ as well as $\vartheta\geq1$ and set $t_L=0$ as well as $t=t_U=1$.
\item\textit{Modification of either $t_L$ or $t_U$:}
\begin{align*}
&\texttt{if }F(x_k+td_k)<0                                            \nonumber\\
&\hphantom{iii}\texttt{if }f(x_k+td_k)\leq f(x_k)+m_Lv_k\cdot t       \nonumber\\
\end{align*}
\begin{align}
&\hphantom{iiiiii}t_L=t                                                     \nonumber\\
&\hphantom{iii}\texttt{else if }f(x_k+td_k)>f(x_k)+m_Lv_k\cdot t       \nonumber\\
&\hphantom{iiiiii}t_U=t                                                     \nonumber\\
&\hphantom{iii}\texttt{end}                                                 \nonumber\\
&\texttt{else if }F(x_k+td_k)\geq0                                     \nonumber\\
&\hphantom{iii}t_U=t                                                        \nonumber\\
&\hphantom{iii}t_0=\hat{t}_0t_U                                             \label{linesearch:t0Shrinking}\\
&\texttt{end}                                                               \nonumber\\
&\texttt{if }t_L\geq t_0                                               \nonumber\\
&\hphantom{iii}t_R=t_L                                                      \nonumber\\
&\hphantom{iii}\texttt{return}\textnormal{ (serious step)}                  \nonumber\\
&\texttt{end}                                                               \nonumber
\end{align}
\item\textit{Decision of return}
\begin{align*}
&\hphantom{\texttt{10 }}\texttt{if }i_n<i_l\nonumber\\
&\hphantom{\texttt{10 }}\hphantom{iii}\texttt{if }F(x_k+td_k)<0
\nonumber\\
&\hphantom{\texttt{10 }}\hphantom{iiiiii}[g,G,\dots]=\texttt{ComputeObjectiveData(}t\texttt{,\dots)}\nonumber\\
&\hphantom{\texttt{10 }}\hphantom{iiiiii}\texttt{if }Z=\texttt{true}\nonumber\\
&\hphantom{\texttt{10 }}\hphantom{iiiiiiiii}t_R=t
\nonumber\\
&\hphantom{\texttt{10 }}\hphantom{iiiiiiiii}\texttt{return}\textnormal{ (short/null step: change of model of the objective function)}
\nonumber\\
&\hphantom{\texttt{10 }}\hphantom{iiiiii}\texttt{end}
\nonumber\\
&\hphantom{\texttt{10 }}\hphantom{iii}\texttt{else if }F(x_k+td_k)\geq0
\nonumber\\
&\hphantom{\texttt{10 }}\hphantom{iiiiii}
[\hat{g},\hat{G},\dots]=\texttt{ComputeConstraintData(}t\texttt{,\dots)}\nonumber\\
&\hphantom{\texttt{10 }}\hphantom{iiiiii}\texttt{if }\hat{Z}=\texttt{true}
\nonumber\\
&\hphantom{\texttt{10 }}\hphantom{iiiiiiiii}t_R=t
\nonumber\\
&\hphantom{\texttt{10 }}\hphantom{iiiiiiiii}\texttt{return}\textnormal{ (short/null step: change of model of the constraint)}
\nonumber\\
&\hphantom{\texttt{10 }}\hphantom{iiiiii}\texttt{end}
\nonumber\\
&\hphantom{\texttt{10 }}\hphantom{iii}\texttt{end}
\nonumber\\
&\hphantom{\texttt{10 }}\texttt{else if }i_n\geq i_l\nonumber\\
&\hphantom{\texttt{10 }}\hphantom{iii}[g,G,\dots]=\texttt{ComputeObjectiveData(}t\texttt{,\dots)}\nonumber\\
&\hphantom{\texttt{10 }}\hphantom{iii}\texttt{if }F(x_k+td_k)<0~~\texttt{and}~~~Z=\texttt{true}\nonumber\\
&\hphantom{\texttt{10 }}\hphantom{iiiiii}t_R=t
\nonumber\\
&\hphantom{\texttt{10 }}\hphantom{iiiiii}\texttt{return}\textnormal{ (short/null step: change of model of the objective function)}
\nonumber\\
&\hphantom{\texttt{10 }}\hphantom{iii}\texttt{end}
\nonumber\\
&\hphantom{\texttt{10 }}\texttt{end}\nonumber
\end{align*}
\item\textit{Interpolation:} Choose $t\in[t_L+\zeta(t_U-t_L)^{\vartheta},t_U-\zeta(t_U-t_L)^{\vartheta}]$.
\item\textit{Loop:} \texttt{Go to 1}
\end{enumerate}
\medskip
\noindent\texttt{function }$[g,G,\dots]=$\texttt{ComputeObjectiveData(}$t$\texttt{,\dots)}
\begin{align}
g&=g(x_k+td_k)\in\partial f(x_k+td_k)\nonumber\\
G&=\text{ approximation of }G(x_k+td_k)\in\subhessian{f}{x_k+td_k}\nonumber\\
\rho&=
\left\lbrace
\begin{array}{ll}
\min(1,\tfrac{C_G}{\lvert G\rvert}) & \textnormal{for }i_n\leq 3\\
0                                   & \textnormal{else}
\end{array}
\right.                                                                                                              
\nonumber\\
f&=f(x_k+td_k)+(t_L-t)g^Td_k+\tfrac{1}{2}\rho(t_L-t)^2d_k^TGd_k
\label{Luk:LineareLS1}\\
\beta&=
\max(\lvert f(x_k+t_Ld_k)-f\rvert,\gamma_1\lvert t_L-t\rvert^{\omega_1}\lvert d_k\rvert^{\omega_1})
\label{Luk:LineareLS2}\\
\widebar{G}&=\textnormal{``positive definite modification of }G\textnormal{''}
\cbmathMIFFLIN
\label{Luk:LineareLS3:NEWfeasibleTrialPointPositiveDefiniteModification}
\\
Z&=
-\beta+d_k^T\big(g+\rho(t_L-t)Gd_k\big)\geq
m_Rv_k
+m_f\cdot(-\tfrac{1}{2}d_k^T\widebar{G}d_k)\nonumber\\
&\hspace{0.55\textwidth}\texttt{ and }(t-t_L)\lvert d_k\rvert\leq C_S
\cbmathMIFFLIN
\label{Luk:LineareLS3}
\end{align}
\texttt{function }$[\hat{g},\hat{G},\dots]=$\texttt{ComputeConstraintData(}$t$\texttt{,\dots)}
\begin{align}
\hat{g}&=\hat{g}(x_k+td_k)\in\partial F(x_k+td_k)\nonumber\\
\hat{G}&=\text{ approximation of }\hat{G}(x_k+td_k)\in\subhessian{F}{x_k+td_k}\nonumber\\%
\hat{\rho}&=
\min(1,\tfrac{\hat{C}_G}{\lvert\hat{G}\rvert})
\nonumber\\
F&=F(x_k+td_k)+(t_L-t)\hat{g}^Td_k+\tfrac{1}{2}\rho(t_L-t)^2d_k^T\hat{G}d_k
\label{Luk:LineareLS1CONSTRAINT}\\
\hat{\beta}&=
\max(\lvert F(x_k+t_Ld_k)-F\rvert,\gamma_2\lvert t_L-t\rvert^{\omega_2}\lvert d_k\rvert^{\omega_2})
\label{Luk:LineareLS2CONSTRAINT}\\
\widebar{\hat{G}}&=\textnormal{``positive definite modification of }\hat{G}\textnormal{''}
\label{Luk:LineareLS3:NEWinfeasibleTrialPointPositiveDefiniteModification}\\
\hat{Z}&=
F(x_k+t_Ld_k)-\hat{\beta}+d_k^T\big(\hat{g}+\hat{\rho}(t_L-t)\hat{G}d_k\big)\geq m_F\cdot(-\tfrac{1}{2}d_k^T\widebar{\hat{G}}d_k)\nonumber\\
&\hspace{0.55\textwidth}\texttt{ and }(t-t_L)\lvert d_k\rvert\leq C_S
\label{Luk:LineareLS3:NEWinfeasibleTrialPoint}
\end{align}
\end{algorithm}
\begin{remark}\label{remark:t0infWithGraphic}
The parameter $i_l$ is only necessary for proving global convergence of Algorithm \ref{BundleSQPmitQCQP:Alg:GesamtAlgMitQCQP} (to be more precise, it is only needed to show that a short or null step which changes the model of the objective function is executed in
Lemma \ref{BundleSQPmitQCQP:GlobaleKonvergenz:Satz:Thm3.8:Bew:wk+1LeqMinQCOMPACT}). If we choose $i_l=0$, then only a change of the model of the objective function yields a short or null step. In fact we have $i_l$ steps in Algorithm \ref{BundleSQPmitQCQP:Alg:GesamtAlgMitQCQP} in which we can use any meaningful criterion for terminating the line search (even for the unconstrained case as it is partially done in the implementation of the original unconstrained bundle-Newton method anyway).

(\ref{linesearch:t0Shrinking}) is due to the following observation:
Consider the line search (Algorithm \ref{BundleSQP:AlgNB:LinesearchMitQCQP}) without (\ref{linesearch:t0Shrinking}) (i.e.~$t_0$ is fixed, e.g., $t_0:=0.5\in(0,1)$, where this large, but legal value for $t_0$ is only chosen to obtain a better graphical illustration in Figure \ref{Figure:BSPt0k}). It can happen (in particular) at the beginning of Algorithm \ref{BundleSQPmitQCQP:Alg:GesamtAlgMitQCQP} that the search direction $d_k$ is bad as we have no knowledge on the behavior of $f$ and $F$ yet. Consequently, the following situation can occur:
The model of the objective function $f$ does not change (e.g., if $f$ is linear on $x_k+td_k$ with $t\in[0,1]$),
and there are no step sizes $t>t_0$ which yield feasible $x_k+td_k$
(this is in particular possible, if we are near the boundary of the feasible set).
\begin{center}
\captionsetup{type=figure}
\includegraphics[width=12cm,height=4cm]{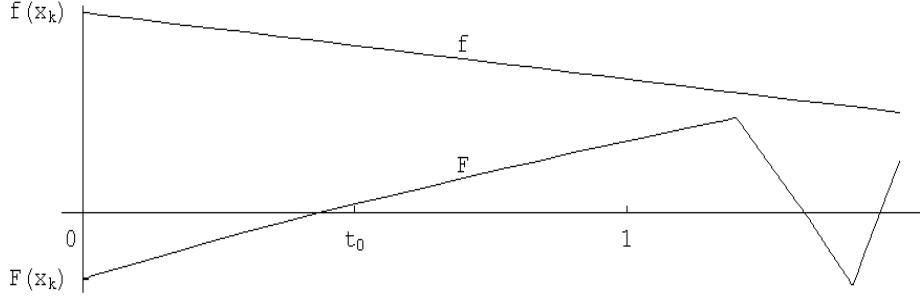}
\captionof{figure}{Line search with fixed $t_0$}
\label{Figure:BSPt0k}
\end{center}
In this situation the line search will not terminate for fixed $t_0$ (in particular in the case $i_n<i_l$ the model of $F$ does not need to even satisfy (\ref{Luk:LineareLS3:NEWinfeasibleTrialPoint}) for infeasible $x_k+td_k$).
Therefore, we need to decrease $t_0$ to have at least one feasible step in the line search for which a descent of $f$ is enough for terminating the line search (similar to the unconstrained case). As the convergence analysis will show, this must not be done too often (cf.~(\ref{Theorem3.8:Presumptiont0infGT0}) and Remark \ref{BundleSQPmitQCQP:GlobaleKonvergenz:Bem:Thm3.8}). Because we use the quadratic terms in the constraint approximation to obtain as much feasibility as possible on the search path $t\mapsto x_k+td_k$ with $t\in[0,1]$ (cf.~the idea that leads to the QCQP (\ref{Luksan:Alg:QCQPTeilproblem})), we expect that this should be true. Indeed, in practice $t_0$ turns out to be only modified at the beginning of Algorithm \ref{BundleSQPmitQCQP:Alg:GesamtAlgMitQCQP} at least many examples of the Hock-Schittkowski collection by \citet{Schittkowski,Schittkowski2} (cf.~\citet[\GenaueAngabeFour]{HannesPaperA}). 
In particular, if $F(x_k+td_k)<0$ for all $t\in[0,1]$ (e.g., if $F$ is constant and negative on $\mathbb{R}^n$ which in fact yields an unconstrained optimization problem), the case (\ref{linesearch:t0Shrinking}) will never occur and therefore $t_0$ will not get changed (this is the reason why $t_0$ is constant in the bundle-Newton method for nonsmooth unconstrained minimization).

The step sizes which the line search returns correspond to the points
$
x_{k+1}=x_k+t_L^kd_k
$
and
$
y_{k+1}=x_k+t    d_k
       =x_k+t_R^kd_k
$.

Only strictly feasible iteration points are accepted in the line search
\begin{equation}
F(x_k+t_L^kd_k)<0
\punkt
\label{BundleSQPmitQCQP:ZulaessigkeitsbedingungFuerIterationspunkt}
\end{equation}
Nevertheless, trial points may be infeasible (if $i_n<i_l$).
\end{remark}
\begin{proposition}
Let
\begin{align}
\hat{\alpha}_p^k
&:=
\sum_{j\in J_k}\lambda_j^k\alpha_j^k+\lambda_p^k\alpha_p^k
\komma\quad
\hat{A}_p^k
:=
\sum_{j\in J_k}\kappa_j^kA_j^k+\kappa_p^kA_p^k
\label{BundleSQPmitQCQP:GlobaleKonvergenz:Lemma3.5:DefalphahatCOMPACT}
\\
\hat{w}_k
&:=
\tfrac{1}{2}\lvert H_k(\tilde{g}_p^k+\bar{\kappa}^{k+1}\tilde{\hat{g}}_p^k)\rvert^2
+\hat{\alpha}_p^k+\bar{\kappa}^{k+1}\hat{A}_p^k+\bar{\kappa}^{k+1}\big(-F(x_k)\big)
\label{BundleSQPmitQCQP:GlobaleKonvergenz:Lemma3.5:Defwhat}
\end{align}
(Note: $\hat{w}_k$ is the optimal function value of the dual problem \refh{Luksan:Alg:QCQPTeilproblem:DualesProblem}). Then we have at iteration $k$ of Algorithm \refH{BundleSQPmitQCQP:Alg:GesamtAlgMitQCQP}
\begin{equation}
\hat{v}_k\leq v_k\leq0\leq w_k\leq\hat{w}_k
\punkt
\label{BundleSQPmitQCQP:Satz:GlobaleKonvergenz:Lemma3.5:wkCOMPACT}
\end{equation}
\end{proposition}
\begin{proof}
For $\gamma>0$ and $\omega\geq1$ the functions $\xi\mapsto\gamma\lvert\xi\rvert^{\omega}$ and $(\xi_1,\xi_2)\mapsto\max{(\xi_1,\xi_2)}$ are convex and therefore we have
$
\gamma\big(\sum_{i=1}^k{t_i\lvert x_i\rvert}\big)^{\omega}
\leq
\sum_{i=1}^k{t_i(\gamma\lvert x_i\rvert^{\omega})}
$
and
$
\max{\big(\sum_{i=1}^k{t_ix_i},\sum_{i=1}^k{t_iy_i}\big)}
\leq
\sum_{i=1}^k{t_i\max{(x_i,y_i)}}
$.
Since $\lambda_j^k\geq0$ for $j\in J_k$ and $\lambda_p^k\geq0$ holds for the solution of the dual problem (\ref{Luksan:Alg:QCQPTeilproblem:DualesProblem}) of the QCQP (\ref{BundleSQPmitQCQP:Alg:QPTeilproblem}), we have $1=\sum_{j\in J_k}\lambda_j^k+\lambda_p^k$ which implies
$f(x_k)=\sum_{j\in J_k}\lambda_j^kf(x_k)+\lambda_p^kf(x_k)$,
and hence
$
\tilde{\alpha}_p^k
\leq
\hat{\alpha}_p^k
$
follows from
(\ref{BundleSQPmitQCQP:Alg:alphatildepk}),
(\ref{BundleSQPmitQCQP:Alg:gtildepkCOMPACT}),
(\ref{BundleSQPmitQCQP:Alg:alphajkCOMPACT}) and
(\ref{BundleSQPmitQCQP:GlobaleKonvergenz:Lemma3.5:DefalphahatCOMPACT}).
If $\bar{\kappa}^{k+1}>0$, we have $1=\sum_{j\in J_k}\kappa_j^k+\kappa_p^k$ due to (\ref{BundleSQPmitQCQP:Alg:Defkappabark+1}) which implies
$F(x_k)=\sum_{j\in J_k}\kappa_j^kF(x_k)+\kappa_p^kF(x_k)$,
and hence
$
\tilde{A}_p^k
\leq
\hat{A}_p^k
$
follows from
(\ref{BundleSQPmitQCQP:Alg:Atildepk}),
(\ref{BundleSQPmitQCQP:Alg:ghattildepkCOMPACT}),
(\ref{BundleSQPmitQCQP:Alg:AjkCOMPACT}) and
(\ref{BundleSQPmitQCQP:GlobaleKonvergenz:Lemma3.5:DefalphahatCOMPACT}).
Consequently, we have
$\bar{\kappa}^{k+1}\tilde{A}_p^k
\leq
\bar{\kappa}^{k+1}\hat{A}_p^k$
for $\bar{\kappa}^{k+1}\geq0$,
which yields
$
0
\leq
\tilde{\alpha}_p^k
+\bar{\kappa}^{k+1}\tilde{A}_p^k
\leq
\hat{\alpha}_p^k
+\bar{\kappa}^{k+1}\hat{A}_p^k
$
due to
(\ref{BundleSQPmitQCQP:Alg:alphatildepk}),
(\ref{BundleSQPmitQCQP:Alg:Defkappabark+1}) and
(\ref{BundleSQPmitQCQP:Alg:Atildepk}).
Now, we obtain the $w_k$-estimate of (\ref{BundleSQPmitQCQP:Satz:GlobaleKonvergenz:Lemma3.5:wkCOMPACT}) due to
(\ref{BundleSQPmitQCQP:Alg:wk}),
(\ref{BundleSQPmitQCQP:GlobaleKonvergenz:Lemma3.5:Defwhat}) and
(\ref{BundleSQPmitQCQP:Alg:Defkappabark+1}).
Because of
(\ref{BundleSQPmitQCQP:GlobaleKonvergenz:Lemma3.5:DefalphahatCOMPACT}) and
(\ref{BundleSQPmitQCQP:Alg:Defkappabark+1})
we have
$0
\geq
-\tilde{\alpha}_p^k
-\bar{\kappa}^{k+1}\tilde{A}_p^k
\geq
-\sum_{j\in J_k}\lambda_j^k\alpha_j^k-\lambda_p^k\alpha_p^k
-\sum_{j\in J_k}\mu_j^kA_j^k-\mu_p^kA_p^k
$,
and, therefore, we obtain the $v_k$-estimate of (\ref{BundleSQPmitQCQP:Satz:GlobaleKonvergenz:Lemma3.5:wkCOMPACT}) by using
(\ref{BundleSQPmitQCQP:Alg:vk}),
(\ref{BundleSQPmitQCQP:Alg:Defkappabark+1}),
(\ref{BundleSQPmitQCQP:Alg:vhatkNEW}),
(\ref{proposition:wkvkZusammenhang}),
the positive definiteness of $\widebar{W}_p^k$ and
(\ref{BundleSQPmitQCQP:Satz:GlobaleKonvergenz:Lemma3.5:wkCOMPACT}).
\qedhere
\end{proof}
\begin{proposition}
If the line search is entered at iteration $k$ of Algorithm \refH{BundleSQPmitQCQP:Alg:GesamtAlgMitQCQP}, then
\begin{equation}
v_k<0
\punkt
\label{proposition:LSaddendum:vkAussage}
\end{equation}
Furthermore, if there occurs a step size $t$ with $F(x_k+td_k)\geq0$ in the line search, then
\begin{equation}
-\tfrac{1}{2}d_k^T\widebar{\hat{G}}_{x_k+td_k}d_k<0
\punkt
\label{proposition:LSaddendum:dkAussage}
\end{equation}
\end{proposition}
\begin{proof}
If the line search is entered at iteration $k$ (cf.~step 6 of Algorithm \ref{BundleSQPmitQCQP:Alg:GesamtAlgMitQCQP}), then no termination occurred at step 5 of Algorithm \ref{BundleSQPmitQCQP:Alg:GesamtAlgMitQCQP} at iteration $k$,
and therefore
we have
$w_k>0$,
which yields
(\ref{proposition:LSaddendum:vkAussage}) due to
(\ref{proposition:wkvkZusammenhang})
and the positive definiteness of $\widebar{W}_p^k$.

Now we show (\ref{proposition:LSaddendum:dkAussage}) by deducing a contradiction:
Suppose (\ref{proposition:LSaddendum:dkAussage}) does not hold, i.e.
$d_k=\zeroVector{N}$
due to (\ref{Luk:LineareLS3:NEWinfeasibleTrialPointPositiveDefiniteModification}). Then, since all iteration points $x_k$ are strictly feasible due to (\ref{BundleSQPmitQCQP:ZulaessigkeitsbedingungFuerIterationspunkt}), we obtain
$
F(x_k+td_k)
=
F(x_k)
<
0
$,
which is a conradiction to the assumption $F(x_k+td_k)\geq0$.
\qedhere
\end{proof}
\begin{proposition}
\begin{enumerate}
\item If the line search (Algorithm \refH{BundleSQP:AlgNB:LinesearchMitQCQP}) terminates with condition \refh{Luk:LineareLS3}, then the old search direction $d_k$ and the old predicted descent $v_k$ (of iteration $k$) are sufficiently infeasible for the new QCQP \refh{BundleSQPmitQCQP:Alg:QPTeilproblem} (at iteration $k+1$) in Algorithm \refH{BundleSQPmitQCQP:Alg:GesamtAlgMitQCQP} (i.e.~the old search direction $d_k$ cannot occur as search direction at iteration $k+1$ and therefore we obtain a different search direction at iteration $k+1$ and consequently a ``meaningful extension of the bundle''.
\item If the line search (Algorithm \refH{BundleSQP:AlgNB:LinesearchMitQCQP}) terminates with condition \refh{Luk:LineareLS3:NEWinfeasibleTrialPoint}, then the old search direction $d_k$ (of iteration $k$) is sufficiently infeasible for the new QCQP \refh{BundleSQPmitQCQP:Alg:QPTeilproblem} (at iteration $k+1$) in Algorithm \refH{BundleSQPmitQCQP:Alg:GesamtAlgMitQCQP} (i.e.~using a QCQP also yields a ``meaningful extension of the bundle'' in the constrained case).
\item The condition $(t-t_L)\lvert d_k\rvert\leq C_S$ in \refh{Luk:LineareLS3:NEWinfeasibleTrialPoint} resp.~\refh{Luk:LineareLS3} corresponds to
\begin{equation}
\lvert y_{k+1}-x_{k+1}\rvert\leq C_S
\punkt
\label{Luksan:Liniensuche:PunktBeschraenktheit}
\end{equation}
\end{enumerate}
\end{proposition}
\begin{proof}~
Because of
$f=f_{k+1}^{k+1}=f_{x_k+td_k}^{\sharp}(x_k+t_Ld_k)$
due to (\ref{Luk:LineareLS1}),
(\ref{BundleSQPmitQCQP:Alg:fk+1k+1Update}) and
(\ref{Def:frauteCOMPACT})
as well as
$\beta=\alpha_{k+1}^{k+1}$
due to
(\ref{Luk:LineareLS2}),
(\ref{BundleSQPmitQCQP:Alg:sk+1k+1Update}) and
(\ref{BundleSQPmitQCQP:Alg:alphajkCOMPACT}),
we obtain
$
-\alpha_{x_k+td_k}^{x_k+t_Ld_k}+d_k^Tg_{x_k+td_k}^{x_k+t_Ld_k}\geq m_Rv_k
+m_f\cdot(-\tfrac{1}{2}d_k\widebar{G}_{x_k+td_k}d_k)
$
by using
(\ref{Luk:LineareLS3}),
(\ref{BundleSQPmitQCQP:Alg:gk+1k+1Update}) and
(\ref{Luksan:Satz:fjrauteGradientCOMPACT}).
Due to the initialization of Algorithm \ref{BundleSQPmitQCQP:Alg:GesamtAlgMitQCQP}, we have $0<m_R<1$
\cbstartMIFFLIN and $0\leq m_f\leq1$. Now, 
(\ref{proposition:LSaddendum:vkAussage})
resp.~(\ref{Luk:LineareLS3:NEWfeasibleTrialPointPositiveDefiniteModification}) imply
$m_Rv_k>v_k$
and
$m_f\cdot(-\tfrac{1}{2}d_k\widebar{G}_{x_k+td_k}d_k)
\geq-\tfrac{1}{2}d_k\widebar{G}_{x_k+td_k}d_k$.
Since the line search (Algorithm \ref{BundleSQP:AlgNB:LinesearchMitQCQP}) terminates with condition (\ref{Luk:LineareLS3}) due to assumption, we obtain that $d_k$ is sufficiently infeasible for the new QCQP (\ref{BundleSQPmitQCQP:Alg:QPTeilproblem}) (with respect to the approximation of the objective function) at iteration $k+1$ due to (\ref{BundleSQPmitQCQP:Satz:GlobaleKonvergenz:Lemma3.5:wkCOMPACT}).

Because of
$F=F_{k+1}^{k+1}=F_{x_k+td_k}^{\sharp}(x_k+t_Ld_k)$
due to
(\ref{Luk:LineareLS1CONSTRAINT}),
(\ref{BundleSQPmitQCQP:Alg:Fk+1k+1Update}) and
(\ref{Def:frauteCOMPACT})
as well as
$\hat{\beta}=A_{k+1}^{k+1}$
due to (\ref{Luk:LineareLS2CONSTRAINT}),
(\ref{BundleSQPmitQCQP:Alg:sk+1k+1Update}) and
(\ref{BundleSQPmitQCQP:Alg:AjkCOMPACT}),
we obtain
$
F(x_k+t_Ld_k)-A_{x_k+td_k}^{x_k+t_Ld_k}+d_k^Tg_{x_k+td_k}^{x_k+t_Ld_k}
\geq
m_F\cdot(-\tfrac{1}{2}d_k\widebar{\hat{G}}_{x_k+td_k}d_k)
$
by using
(\ref{Luk:LineareLS3:NEWinfeasibleTrialPoint}),
(\ref{BundleSQPmitQCQP:Alg:ghatk+1k+1Update}) and
(\ref{Luksan:Satz:fjrauteGradientCOMPACT}).
Due to the initialization of Algorithm \ref{BundleSQPmitQCQP:Alg:GesamtAlgMitQCQP}, we have $0<m_F<1$. Now, 
(\ref{proposition:LSaddendum:dkAussage})
implies
$m_F\cdot(-\tfrac{1}{2}d_k\widebar{\hat{G}}_{x_k+td_k}d_k)
>-\tfrac{1}{2}d_k\widebar{\hat{G}}_{x_k+td_k}d_k$.
Since the line search (Algorithm \ref{BundleSQP:AlgNB:LinesearchMitQCQP}) terminates with condition (\ref{Luk:LineareLS3:NEWinfeasibleTrialPoint}) due to assumption, we obtain that $d_k$ is sufficiently infeasible for the new QCQP (\ref{BundleSQPmitQCQP:Alg:QPTeilproblem}) (with respect to the approximation of the constraint) at iteration $k+1$.

(\ref{Luksan:Liniensuche:PunktBeschraenktheit}) follows from
(\ref{Luk:LineareLS3}) and
(\ref{Luk:LineareLS3:NEWinfeasibleTrialPoint}).
\qedhere
\end{proof}

\section{Convergence}
\label{Paper:Convergence}
In the following section we prove the convergence of the line search and we show the global convergence of the algorithm.
\subsection{Convergence of the line search}
For proving the convergence of the line search (Algorithm \ref{BundleSQP:AlgNB:LinesearchMitQCQP}) we have to identify a large subclass of locally Lipschitz continuous functions, which is the class of weakly upper semismooth functions (that contains, e.g., functions that are the pointwise maximum of finitely many continuously differentiable functions due to \citet[p.~963,~Theorem~2]{Mifflin}).
\begin{definition}
A locally Lipschitz continuous function $f:\mathbb{R}^N\rightarrow\mathbb{R}$ is called weakly upper semismooth, if
\begin{equation}
\limsup_{i\rightarrow\infty}\bar{g}_i^Td
\geq
\liminf_{i\rightarrow\infty}\tfrac{f(x_k+t_id)-f(x_k)}{t_i}
\label{Luksan:Lemma2.3:Semismooth}
\end{equation}
holds for all $x\in\mathbb{R}^N$, $d\in\mathbb{R}^N$, $\lbrace\bar{g}_i\rbrace_i\subset\mathbb{R}^N$ with $\bar{g}_i\in\partial f(x+t_id)$ and $\lbrace t_i\rbrace_i\subset\mathbb{R}_+$ with $t_i\searrow0$.
\end{definition}
\begin{proposition}
\label{proposition:ConvergenceOfLinesearch}
Let $f:\mathbb{R}^N\rightarrow\mathbb{R}$ be weakly upper semismooth, then the line search (Algorithm \refH{BundleSQP:AlgNB:LinesearchMitQCQP}) terminates after finitely many steps with $t_L^k=t_L$, $t_R^k=t$ and $t_0^k>0$.
\end{proposition}
\begin{proof}
If $F(x_k+td_k)<0$ for all $t\in[0,1]$, then this is exactly the same situation as in the line search of the unconstrained bundle-Newton method which terminates after finitely many iterations due to \citet[p.~379,~Proof~of~Lemma~2.3]{Luksan}.
Otherwise, since $F$ is continuous and $F(x_k)<0$, there exists a largest $\tilde{t}>0$ with $F(x_k+d_k\tilde{t})=0$ and $F(x_k+d_ks)<0$ for all $s<\tilde{t}$. Therefore, after sufficiently many iterations in the line search (Algorithm \ref{BundleSQP:AlgNB:LinesearchMitQCQP}) (Note that the interval $[t_L,t_U]$ is shrinking at each iteration of the line search), there only occur $t_L,t_0,t_U$ with $0\leq t_L<t_U<\tilde{t}$ and $0<t_0<t_U<\tilde{t}$ (i.e.~from now on all $x_k+td_k$ with $t\in\lbrace t_L,t_U\rbrace$ are feasible) and consequently $t_0$ (where $x_k+t_0d_k$ is also feasible,) does not change anymore (cf.~(\ref{linesearch:t0Shrinking})). Hence, here we also have exactly the same situation as in the line search of the unconstrained bundle-Newton method, which terminates after finitely many iterations due to \citet[Proof~of~Lemma~2.3]{Luksan},
\cbstartMIFFLIN
where the only difference in the proof is that we need to use the following additional argument to obtain the inequality at the bottom of \citet[p.~379]{Luksan}: Since $m_f\in[0,1]$ due to the initialization of Algorithm \ref{BundleSQPmitQCQP:Alg:GesamtAlgMitQCQP} and since $\widebar{G}$ is positive definite due to (\ref{Luk:LineareLS3:NEWfeasibleTrialPointPositiveDefiniteModification}), the negation of the condition in (\ref{Luk:LineareLS3}) that corresponds to the change of the model of the objective function yields
$
-\beta+d_k^T\big(g+\rho(t_L-t)Gd_k\big)
<
m_Rv_k
+
m_f\cdot(-\tfrac{1}{2}d_k^T\widebar{G}d_k)
\leq
m_Rv_k
$.
\qedhere
\cbendMIFFLIN
\end{proof}
\begin{remark}
The proof of Proposition \ref{proposition:ConvergenceOfLinesearch} only relies on $f$ satisfying (\ref{Luksan:Lemma2.3:Semismooth}), the continuity of $F$ and the strict feasibility of $x_k$. In particular, $F$ does not need to be weakly upper semismooth.
\end{remark}

\subsection{Global convergence}
\label{section:GlobalConvergence}
For investigating the global convergence of Algorithm \ref{BundleSQPmitQCQP:Alg:GesamtAlgMitQCQP} we will follow closely the proof of global convergence of the bundle-Newton method for nonsmooth unconstrained minimization in \citet[p.~380-385,~Section~3]{Luksan} with modifications which concern the constrained case and the use of determining the search direction by solving a QCQP, where we will work out everything in great detail so that it is easy to see which passages of the proof are similar to the unconstrained case resp.~which passages require a careful examination. Therefore, we assume
\begin{equation}
\varepsilon
=
0
\komma\quad
\lambda_j^k
=
0~~~\forall j\not\in J_k
\komma\quad
\mu_j^k
=
0~~~\forall j\not\in J_k
\punkt
\label{BundleSQPmitQCQP:GlobaleKonvergenz:epsilon=0}
\end{equation}
A main difference to the proof of convergence of the unconstrained bundle-Newton method is that here $H_k$ from (\ref{BundleSQPmitQCQP:Alg:Hk}) depends on the Lagrange multipliers
\cbstartMIFFLIN
$(\lambda^k,\lambda_p^k,\mu^k,\mu_p^k)$
\cbendMIFFLIN
of the QCQP (\ref{BundleSQPmitQCQP:Alg:QPTeilproblem}), which implies that so do the search direction $d_k$ from (\ref{BundleSQPmitQCQP:Alg:dk}) (and consequently the new iteration point $x_{k+1}$ from (\ref{BundleSQPmitQCQP:Alg:xk+1Update}) as well as the new trial point $y_{k+1}$ from (\ref{BundleSQPmitQCQP:Alg:yk+1Update})) and
the termination criterion $w_k$ from (\ref{BundleSQPmitQCQP:Alg:wk}) in particular. Furthermore, this dependence does not allow us to achieve the equality $H_{k+1}=H_k$ in the proof of Theorem \ref{BundleSQPmitQCQP:GlobaleKonvergenz:Satz:Thm3.8:Bew:Theorem} in contrast to \citet[top~of~page~385,~Proof~of~Theorem~3.8]{Luksan}, which extends the complexity of the already quite involved proof of the unconstrained bundle-Newton method.

Hence we give a brief overview of the main steps of the proof:
In Proposition \ref{BundleSQPmitQCQP:GlobaleKonvergenzSatz:Lemma3.1} we express the $p$-tilde data (as, e.g., $\tilde{g}_p^k$, $\tilde{\hat{g}}_p^k$,\dots) as convex combinations in which no $p$-data (as, e.g., $g_p^k$, $\hat{g}_p^k$,\dots) occurs.
Afterwards we recall a sufficient condition to identify a vector as an element of the subdifferential in Proposition \ref{Luksan:GlobaleKonvergenzSatz:Lemma3.2}.
In Theorem \ref{AN:Theorem:Lemma3.3NB} we show that if Algorithm \ref{BundleSQPmitQCQP:Alg:GesamtAlgMitQCQP} stops at iteration $k$, then the current iteration point $x_k$ is stationary for the optimization problem (\ref{BundleSQP:OptProblem}).
From then on on we assume that the algorithm does not terminate (cf.~(\ref{BundleSQPmitQCQP:GlobaleKonvergenz:Vor:AlgTerminiertNicht})).
After summarizing some properties of positive definite matrices, we deduce bounds for $\lbrace(\widebar{W}_p^k)^{-1}\rbrace$ and
\cbstartMIFFLIN
$\lbrace\widebar{W}_p^k+\widebar{G}^k+\bar{\kappa}^{k+1}\widebar{\hat{G}}^k\rbrace$
\cbendMIFFLIN
in Corollary \ref{ADDON:Korollar:WpkUndHkBeschraenktheitUndPositivDefinitheit}, which will be essential in the following.
Then, in Proposition \ref{BundleSQPmitQCQP:GlobaleKonvergenz:Satz:Lemma3.6}, we show that if some boundedness assumptions are satisfied and the limit inferior of the sequence $\lbrace\max{(w_k,\lvert x_k-\bar{x}\rvert)}\rbrace$ is zero, where $\bar{x}$ denotes any accumulation point of the sequence of iteration points $\lbrace x_k\rbrace$, then $\bar{x}$ is stationary for the optimization problem (\ref{BundleSQP:OptProblem}), where the proof relies on Carath\'eodory's theorem 
as well as on the local boundedness
and the upper semicontinuity
of the subdifferentials $\partial f$ and $\partial F$.
Due to the negativity of $v_k$,
which holds due to (\ref{BundleSQPmitQCQP:Satz:GlobaleKonvergenz:Lemma3.5:wkCOMPACT}),
we obtain the statement $t_L^kv_k\longrightarrow0$ in Proposition \ref{AN:proposition:Lemma3.5(ii)}.
In Proposition \ref{BundleSQPmitQCQP:GlobaleKonvergenz:Lemma3.7} we show some properties of the shifted sequences $\lbrace x_{k+i}\rbrace$, $\lbrace w_{k+i}\rbrace$ and $\lbrace t_L^{k+i}\rbrace$, where we have to take care of the dependence of
\cbstartMIFFLIN
$(\lambda^k,\lambda_p^k,\mu^k,\mu_p^k)$,
\cbendMIFFLIN
which we noticed before, in the proof.
Then we recall an estimation of a certain quadratic function on the interval $[0,1]$ in Proposition \ref{Luksan:GlobaleKonvergenz:Lemma3.4}.
After recalling the differentiability of matrix valued functions to give a formula for the derivative of the matrix square root in Proposition \ref{Magnus:Satz:InverseMatrixAbleitung}
and after formulating the mean value theorem for vector valued functions on a convex set in
Proposition
\ref{AN:proposition:MeanValueTheoremForVectorValuedFunctions}, we combine these two results to obtain a Lipschitz estimate for the inverse matrix square root in Proposition \ref{KonvergenzErweiterung:Satz:MatrixAbschaetzungMitHigham:ZkAbschaetzungSatz}, which serves as replacement for the property $H_{k+1}=H_k$ of the proof of the unconstrained bundle-Newton method as mentioned above.
Finally, we prove that under some additional boundedness assumptions the limit inferior of the sequence $\lbrace\max{(w_k,\lvert x_k-\bar{x}\rvert)}\rbrace$ is always zero and therefore Proposition \ref{BundleSQPmitQCQP:GlobaleKonvergenz:Satz:Lemma3.6} yields Theorem \ref{BundleSQPmitQCQP:GlobaleKonvergenz:Satz:Thm3.8:Bew:Theorem}
which states that each accumulation point $\bar{x}$ of the sequence of iteration points $\lbrace x_k\rbrace$ is stationary for the optimization problem (\ref{BundleSQP:OptProblem}).
\begin{proposition}
\label{BundleSQPmitQCQP:GlobaleKonvergenzSatz:Lemma3.1}
Assume that Algorithm \refH{BundleSQPmitQCQP:Alg:GesamtAlgMitQCQP} has not stopped before iteration $k$ with $k\geq1$.
Then there exists $\hat{\lambda}_j^k\in\mathbb{R}$ for $j=1,\dots,k$ with
\begin{equation}
\hat{\lambda}_j^k\geq0
\komma\quad
1=\sum_{j=1}^k\hat{\lambda}_j^k
\komma\quad
(G_p^{k+1},\tilde{g}_p^k,\tilde{s}_p^k)=\sum_{j=1}^k\hat{\lambda}_j^k(\rho_jG_j,g_j^k,s_j^k)
\punkt
\label{BundleSQPmitQCQP:GlobaleKonvergenzSatz:Lemma3.1:22:1COMPACT}
\end{equation}
If $\bar{\kappa}^{k+1}>0$, then there exists $\hat{\kappa}_j^k\in\mathbb{R}$ for $j=1,\dots,k$ with
\begin{equation}
\hat{\kappa}_j^k\geq0
\komma\quad
1=\sum\limits_{j=1}^k\hat{\kappa}_j^k
\komma\quad
(\hat{G}_p^{k+1},\tilde{\hat{g}}_p^k,\tilde{\hat{s}}_p^k)=\sum_{j=1}^k\hat{\kappa}_j^k(\hat{\rho}_j\hat{G}_j,\hat{g}_j^k,s_j^k)
\punkt
\label{BundleSQPmitQCQP:GlobaleKonvergenzSatz:Lemma3.1:22:1NBCOMPACT}
\end{equation}
If $\bar{\kappa}^{k+1}=0$, then \refh{BundleSQPmitQCQP:GlobaleKonvergenzSatz:Lemma3.1:22:1NBCOMPACT} holds with
\begin{equation}
\hat{\kappa}_j^k:=0~~~\forall j=1,\dots,k
\label{BundleSQPmitQCQP:GlobaleKonvergenzSatz:Lemma3.1:22:1NB:AllZero}
\punkt
\end{equation}
\end{proposition}
\begin{proof}(by induction)
Since $g_p^1=g_1^1$ due to (\ref{BundleSQPmitQCQP:Alg:gStartwert}) as well as $\alpha_p^1=\alpha_1^1$ due to
(\ref{BundleSQPmitQCQP:Alg:alphajkCOMPACT}),
(\ref{BundleSQPmitQCQP:Alg:fStartwert}) and
(\ref{BundleSQPmitQCQP:Alg:sStartwert}),
as well as $\hat{g}_p^1=\hat{g}_1^1$ due to (\ref{BundleSQPmitQCQP:Alg:ghatStartwert})
as well as $A_p^1=A_1^1$ due to
(\ref{BundleSQPmitQCQP:Alg:AjkCOMPACT}),
(\ref{BundleSQPmitQCQP:Alg:FStartwert}) and
(\ref{BundleSQPmitQCQP:Alg:sStartwert}),
as well as
\cbstartMIFFLIN
$\widebar{G}^1=\widebar{G}_1^1$
\cbendMIFFLIN
and
$\widebar{\hat{G}}^1=\widebar{\hat{G}}_1^1$ due to
(\ref{BundleSQPmitQCQP:Alg:ModifikationVonGhatbarjk}),
(\ref{BundleSQPmitQCQP:Alg:ModifikationVonGhatbarjkBeiGleichheitsCase}) and
\cbstartMIFFLIN
(\ref{BundleSQPmitQCQP:Alg:GpStartwert})
\cbendMIFFLIN
resp.~(\ref{BundleSQPmitQCQP:Alg:GhatpStartwert}),
the aggregated (p-)constraint of the QCQP (\ref{BundleSQPmitQCQP:Alg:QPTeilproblem}) at iteration $k=1$ of Algorithm \ref{BundleSQPmitQCQP:Alg:GesamtAlgMitQCQP} coincides with the corresponding bundle constraint, and therefore we can drop the aggregated (p-)constraint and consequently the dual problem (\ref{Luksan:Alg:QCQPTeilproblem:DualesProblem}) has only two variables $\lambda_1^1$ and $\mu_1^1$, where $\lambda_1^1=1$ must hold, so that the equality constraint of the dual problem (\ref{Luksan:Alg:QCQPTeilproblem:DualesProblem}) is satisfied. Now, if we set $\lambda_p^1=0$ and $\mu_p^1=0$, then the dual solution does not change.

Consequently,
(\ref{BundleSQPmitQCQP:GlobaleKonvergenzSatz:Lemma3.1:22:1COMPACT})
holds due to the same calculations which are performed in \citet[Lemma~3.1]{Luksan}.

Furthermore, we obtain $\bar{\kappa}^2=\mu_1^1$ due to (\ref{BundleSQPmitQCQP:Alg:Defkappabark+1}) and therefore we get $\kappa_j^1=1$ for $\bar{\kappa}^2>0$ and $\kappa_j^1=0$ for $\bar{\kappa}^2=0$ as well as $\kappa_p^1=0$. Summarizing these facts yield that we have at iteration $k=1$ of Algorithm \ref{BundleSQPmitQCQP:Alg:GesamtAlgMitQCQP}
that
$\bar{\kappa}^{k+1}=\mu_1^k$,
$
\kappa_1^k
=
\left\lbrace
\begin{array}{ll}
1 & \textnormal{for }\bar{\kappa}^{k+1}>0\\
0 & \textnormal{for }\bar{\kappa}^{k+1}=0
\end{array}
\right\rbrace
$
and
$\kappa_p^k=0$.

Therefore, the base case is satisfied for $k=1$ with
$
\hat{\kappa}_1^k:=
\left\lbrace
\begin{array}{ll}
1 & \textnormal{for }\bar{\kappa}^{k+1}>0\\
0 & \textnormal{for }\bar{\kappa}^{k+1}=0
\end{array}
\right\rbrace\komma
$
since
(\ref{BundleSQPmitQCQP:GlobaleKonvergenzSatz:Lemma3.1:22:1NBCOMPACT})
holds due to
(\ref{BundleSQPmitQCQP:Alg:ghattildepkCOMPACT}).

Let the induction hypothesis be satisfied (i.e.~we have $\bar{\kappa}^{k+1}>0$ in particular) and define
\begin{equation}
\hat{\kappa}_j^{k+1}
:=
\left\lbrace
\begin{array}{ll}
\kappa_j^{k+1}+\kappa_p^{k+1}\hat{\kappa}_j^k & \textnormal{for }\bar{\kappa}^{k+1}>0\\
\kappa_j^{k+1}                                & \textnormal{for }\bar{\kappa}^{k+1}=0
\end{array}
\right\rbrace
~\textnormal{for }j=1,\dots,k
\komma\quad
\hat{\kappa}_{k+1}^{k+1}:=\kappa_{k+1}^{k+1}
\komma
\label{BundleSQPmitQCQP:GlobaleKonvergenzSatz:Lemma3.1:Defkappahatjk+1COMPACT}
\end{equation}
where $(\mu_j^{k+1},\mu_p^{k+1})$ is
part of
the solution of the dual problem (\ref{Luksan:Alg:QCQPTeilproblem:DualesProblem}) (including the aggregated terms)
and $\kappa_j^{k+1}$ resp.~$\kappa_p^{k+1}$ are set according to (\ref{BundleSQPmitQCQP:Alg:Defkappabark+1}).
The case $\bar{\kappa}^{k+1}=0$ is equivalent to $\kappa_j^k=\kappa_p^k=0$ for all $j=1,\dots,k$
due to (\ref{BundleSQPmitQCQP:Alg:Defkappabark+1}) and therefore we obtain $\tilde{\hat{g}}_p^k=\zeroVector{N}$ and $\hat{G}_p^{k+1}=\zeroMatrix{N}$ due to (\ref{BundleSQPmitQCQP:Alg:ghattildepkCOMPACT}), which implies $\hat{g}_p^{k+1}=\zeroVector{N}$ due to (\ref{BundleSQPmitQCQP:Alg:ghatpk+1Update}). Hence, at iteration $k+1$ in the QCQP (\ref{BundleSQPmitQCQP:Alg:QPTeilproblem}) the aggregated constraint for $F$ reads in the case $i_s\leq i_r$ $F(x_{k+1})-A_p^{k+1}\leq0$. Since this inequality is sharp due to (\ref{BundleSQPmitQCQP:ZulaessigkeitsbedingungFuerIterationspunkt}) and (\ref{BundleSQPmitQCQP:Alg:AjkCOMPACT}), the aggregated constraint for $F$ is inactive at iteration $k+1$. Since Lagrange multipliers for inactive constraints vanish, we obtain at iteration $k+1$ (Note that $\mu_p^{k+1}$ is the Lagrange multiplier corresponding to the aggregated constraint for $F$ at iteration $k+1$ and note that $\bar{\kappa}^{k+2}>0$ is the assumption for what we want to show by the inductive step $k\mapsto k+1$) $\mu_p^{k+1}=0$ which implies
\begin{equation}
\kappa_p^{k+1}
=
0
\komma\quad
\big(
\sum_{j=1}^{k+1}\kappa_j^{k+1}=1
~~~\wedge~~~
(\kappa_j^{k+1}\geq0~~\forall j=1,\dots,k+1)
\big)
\label{BundleSQPmitQCQP:GlobaleKonvergenz:Satz:BEWEISkappapkPLUS1IS0}
\end{equation}
due to $\bar{\kappa}^{k+2}>0$ and (\ref{BundleSQPmitQCQP:Alg:Defkappabark+1}). In the case $i_s>i_r$ (\ref{BundleSQPmitQCQP:GlobaleKonvergenz:Satz:BEWEISkappapkPLUS1IS0}) holds anyway, since then in the dual problem (\ref{Luksan:Alg:QCQPTeilproblem:DualesProblem}) for the QCQP (\ref{BundleSQPmitQCQP:Alg:QPTeilproblem}) the aggregated constraints do not occur and therefore the corresponding Lagrange multiplier can be set to zero.
So, the inductive step $k\mapsto k+1$ for
the first two properties of
(\ref{BundleSQPmitQCQP:GlobaleKonvergenzSatz:Lemma3.1:22:1NBCOMPACT})
holds in the case $\bar{\kappa}^{k+1}>0$ due to
(\ref{BundleSQPmitQCQP:GlobaleKonvergenzSatz:Lemma3.1:Defkappahatjk+1COMPACT}),
(\ref{BundleSQPmitQCQP:Alg:Defkappabark+1}) and
(\ref{Luksan:Alg:QCQPTeilproblem:DualesProblem})
(Note that we assumed that we consider the case $\bar{\kappa}^{k+1}>0$ which implies that we can use the induction hypothesis for
the first two properties of
(\ref{BundleSQPmitQCQP:GlobaleKonvergenzSatz:Lemma3.1:22:1NBCOMPACT}) and note that we have $\bar{\kappa}^{k+2}>0$, since this is the assumption for what we want to show by the inductive step $k\mapsto k+1$) and in the case $\bar{\kappa}^{k+1}=0$ due to
(\ref{BundleSQPmitQCQP:GlobaleKonvergenzSatz:Lemma3.1:Defkappahatjk+1COMPACT}) and
(\ref{BundleSQPmitQCQP:GlobaleKonvergenz:Satz:BEWEISkappapkPLUS1IS0}).
The inductive step for
the third property of
(\ref{BundleSQPmitQCQP:GlobaleKonvergenzSatz:Lemma3.1:22:1NBCOMPACT}) holds in the case $\bar{\kappa}^{k+1}>0$ due to
(\ref{BundleSQPmitQCQP:Alg:ghattildepkCOMPACT}) and
(\ref{BundleSQPmitQCQP:GlobaleKonvergenzSatz:Lemma3.1:Defkappahatjk+1COMPACT}),
and in the case $\bar{\kappa}^{k+1}=0$ due to
(\ref{BundleSQPmitQCQP:Alg:ghattildepkCOMPACT}),
(\ref{BundleSQPmitQCQP:GlobaleKonvergenz:Satz:BEWEISkappapkPLUS1IS0}) and
(\ref{BundleSQPmitQCQP:GlobaleKonvergenzSatz:Lemma3.1:Defkappahatjk+1COMPACT}).
The inductive step for
the fourth property of
(\ref{BundleSQPmitQCQP:GlobaleKonvergenzSatz:Lemma3.1:22:1NBCOMPACT}) holds in the case $\bar{\kappa}^{k+1}>0$ due to
(\ref{BundleSQPmitQCQP:Alg:ghattildepkCOMPACT}),
(\ref{BundleSQPmitQCQP:Alg:ghatpk+1Update}),
(\ref{BundleSQPmitQCQP:Alg:ghatjk+1Update}) and
(\ref{BundleSQPmitQCQP:GlobaleKonvergenzSatz:Lemma3.1:Defkappahatjk+1COMPACT})
and in the case $\bar{\kappa}^{k+1}=0$ due to
(\ref{BundleSQPmitQCQP:Alg:ghattildepkCOMPACT}),
(\ref{BundleSQPmitQCQP:GlobaleKonvergenz:Satz:BEWEISkappapkPLUS1IS0}) and
(\ref{BundleSQPmitQCQP:GlobaleKonvergenzSatz:Lemma3.1:Defkappahatjk+1COMPACT}).
The inductive step for
the fifth property of
(\ref{BundleSQPmitQCQP:GlobaleKonvergenzSatz:Lemma3.1:22:1NBCOMPACT}) holds in the case $\bar{\kappa}^{k+1}>0$ due to
(\ref{BundleSQPmitQCQP:Alg:ghattildepkCOMPACT}),
(\ref{BundleSQPmitQCQP:Alg:shatpk+1Update}),
(\ref{BundleSQPmitQCQP:Alg:sjk+1Update}) and
(\ref{BundleSQPmitQCQP:GlobaleKonvergenzSatz:Lemma3.1:Defkappahatjk+1COMPACT}),
and in the case $\bar{\kappa}^{k+1}=0$ due to
(\ref{BundleSQPmitQCQP:Alg:ghattildepkCOMPACT}),
(\ref{BundleSQPmitQCQP:GlobaleKonvergenz:Satz:BEWEISkappapkPLUS1IS0}) and
(\ref{BundleSQPmitQCQP:GlobaleKonvergenzSatz:Lemma3.1:Defkappahatjk+1COMPACT}).

In the case $\bar{\kappa}^{k+1}=0$ we obtain $\kappa_j^k=0$ for all $j=1,\dots,k$ and $\kappa_p^k=0$ due to
(\ref{BundleSQPmitQCQP:Alg:Defkappabark+1}) and therefore (\ref{BundleSQPmitQCQP:GlobaleKonvergenzSatz:Lemma3.1:22:1NBCOMPACT}) holds due to (\ref{BundleSQPmitQCQP:Alg:ghattildepkCOMPACT}) and (\ref{BundleSQPmitQCQP:GlobaleKonvergenzSatz:Lemma3.1:22:1NB:AllZero}).
\qedhere
\end{proof}
\begin{proposition}
\label{Luksan:GlobaleKonvergenzSatz:Lemma3.2}
If $\bar{x}\in\mathbb{R}^N$ and there exists $\widebar{G}_j\in\Sym{N},~\bar{q},\bar{y}_j\in\mathbb{R}^n,~\bar{g}_j\in\partial f(y_j),~\bar{s}_j,\bar{\lambda}_j\in\mathbb{R}$ for $j=1,\dots,L$, where $L\geq1$, with
\begin{align*}
(\bar{q},0)&=\sum_{j=1}^L\big(\bar{g}_j+\widebar{G}_j(\bar{x}-\bar{y}_j),\bar{s}_j\big)\bar{\lambda}_j
\komma\quad
1=\sum_{j=1}^L\bar{\lambda}_j
\komma\quad
\bar{\lambda}_j\geq0
\komma\quad
\lvert\bar{y}_j-\bar{x}\rvert\leq\bar{s}_j
\komma
\end{align*}
for all $j=1,\dots,L$, then $\bar{q}\in\partial f(\bar{x})$.
\end{proposition}
\begin{proof}
\citet[p.~381,~Proof~of~Lemma~3.2]{Luksan}.
\qedhere
\end{proof}
\begin{theorem}
\label{AN:Theorem:Lemma3.3NB}
If Algorithm \refH{BundleSQPmitQCQP:Alg:GesamtAlgMitQCQP} stops at iteration $k$, then there exists $\bar{\kappa}^{k+1}\geq0$
such that \refh{AOB:Satz:KarushJohnUnglgsNBNichtGlattAlternativformGlgsSystem} holds for $(x_k,\bar{\kappa}^{k+1})$,
i.e.~$x_k$ is stationary for the optimization problem \refh{BundleSQP:OptProblem}.
\end{theorem}
\begin{proof}
Since Algorithm \ref{BundleSQPmitQCQP:Alg:GesamtAlgMitQCQP} stops at iteration $k$, step 5 of the algorithm, (\ref{BundleSQPmitQCQP:GlobaleKonvergenz:epsilon=0}) and (\ref{BundleSQPmitQCQP:Satz:GlobaleKonvergenz:Lemma3.5:wkCOMPACT}) imply $w_k=0$ which is equivalent to
\begin{equation}
\tfrac{1}{2}\lvert H_k(\tilde{g}_p^k+\bar{\kappa}^{k+1}\tilde{\hat{g}}_p^k)\rvert^2=0
~~\wedge~~\tilde{\alpha}_p^k=0
~~\wedge~~\bar{\kappa}^{k+1}\tilde{A}_p^k=0
~~\wedge~~\bar{\kappa}^{k+1}\big(-F(x_k)\big)=0
\label{BundleSQPmitQCQP:GlobaleKonvergenz:Satz:Lemma3.3:Bew1shatGlg3}
\end{equation}
due to (\ref{BundleSQPmitQCQP:Alg:wk}), (\ref{BundleSQPmitQCQP:Alg:alphatildepk}), $\bar{\kappa}^{k+1}\geq0$ and $F(x_k)\leq0$. Using the regularity of $H_k$, (\ref{BundleSQPmitQCQP:Alg:alphatildepk}) and (\ref{BundleSQPmitQCQP:Alg:gtildepkCOMPACT}), we obtain from (\ref{BundleSQPmitQCQP:GlobaleKonvergenz:Satz:Lemma3.3:Bew1shatGlg3})
\begin{equation}
\tilde{g}_p^k+\bar{\kappa}^{k+1}\tilde{\hat{g}}_p^k=\zeroVector{N}
\label{BundleSQPmitQCQP:GlobaleKonvergenz:Satz:Lemma3.3:Bew2}
\komma\quad
\tilde{s}_p^k=0\punkt
\end{equation}
Furthermore, for $\bar{\kappa}^{k+1}>0$ we obtain from (\ref{BundleSQPmitQCQP:GlobaleKonvergenz:Satz:Lemma3.3:Bew1shatGlg3}), (\ref{BundleSQPmitQCQP:Alg:Atildepk}) and
(\ref{BundleSQPmitQCQP:Alg:ghattildepkCOMPACT}) that $\tilde{\hat{s}}_p^k=0$ and hence we have
either $\bar{\kappa}^{k+1}=0$ or $\bar{\kappa}^{k+1}>0~\wedge~\tilde{\hat{s}}_p^k=0$.

We set $\bar{x}:=x_k$, $L:=k$, $\bar{y}_j:=y_j$, $\bar{s}_j:=s_j^k$. Then for $\widebar{G}_j:=\rho_jG_j$,  $\bar{g}_j:=g_j$, $\bar{\lambda}_j:=\hat{\lambda}_j^k$, and $\bar{q}:=\tilde{g}_p^k$ resp.~for $\bar{\kappa}^{k+1}>0$, $\widebar{G}_j':=\hat{\rho}_j\hat{G}_j$, $\bar{g}_j':=\hat{g}_j$, $\bar{\lambda}_j':=\hat{\kappa}_j^k$, and $\bar{q}':=\tilde{\hat{g}}_p^k$ the assumptions of Proposition \ref{Luksan:GlobaleKonvergenzSatz:Lemma3.2} are satisfied (by using roposition \ref{BundleSQPmitQCQP:GlobaleKonvergenzSatz:Lemma3.1}), and therefore we obtain
$\tilde{g}_p^k\in\partial f(x_k)$
and
$\tilde{\hat{g}}_p^k\in\partial F(x_k)$.
Now, using (\ref{BundleSQPmitQCQP:GlobaleKonvergenz:Satz:Lemma3.3:Bew2}) we calculate $\zeroVector{N}\in\partial f(x_k)+\bar{\kappa}^{k+1}\partial F(x_k)$.
\qedhere
\end{proof}
From now on, we demand that Algorithm \ref{BundleSQPmitQCQP:Alg:GesamtAlgMitQCQP} does not stop, i.e.~according to step 5 of Algorithm \ref{BundleSQPmitQCQP:Alg:GesamtAlgMitQCQP} and (\ref{BundleSQPmitQCQP:GlobaleKonvergenz:epsilon=0}) we have for all $k$
\begin{equation}
w_k>0
\punkt
\label{BundleSQPmitQCQP:GlobaleKonvergenz:Vor:AlgTerminiertNicht}
\end{equation}
\cbstartDVI
We summarize some properties of positive (semi)definite matrices.
\begin{proposition}
Let $A,B\in\Sym{N}$ with $B$ positive semidefinite, then
\begin{equation}
A\preceq A+B
\punkt
\label{HornJohnson:Erweiterung:PositivDefinitSummeAbschaetzung}
\end{equation}
If $A$ and $B$ are even positive definite,
then
\begin{equation}
\lvert A^{\mathalf}-B^{\mathalf}\rvert
\leq
\tfrac{1}{
(\lambdamin(A))^{\mathalf}+(\lambdamin(B))^{\mathalf}
}
\lvert A-B\rvert
\komma
\label{KonvergenzErweiterung:Satz:Higham}
\end{equation}
and if additionally $A\preceq B$ holds, then
\begin{equation}
\lvert B^{-1}\rvert
\leq
\lvert A^{-1}\rvert
\punkt
\label{HornJohnson:Erweiterung:InverseSpektralNormAbschaetzung}
\end{equation}
\end{proposition}
\begin{proof}
(\ref{HornJohnson:Erweiterung:PositivDefinitSummeAbschaetzung}) is clear.
(\ref{KonvergenzErweiterung:Satz:Higham}) holds due to \citet[p.~135,~Theorem~6.2]{HighamFunctionsOfMatrices}.
Since $B$ is positive definite due to assumption, $B^{-1}$ is positive definite and since all eigenvalues of a positive definite matrix are positive,
we obtain (\ref{HornJohnson:Erweiterung:InverseSpektralNormAbschaetzung}) due to
the fact that $A\preceq B~\Longleftrightarrow~B^{-1}\preceq A^{-1}$
(cf.~\citet[p.~471,~Corollary~7.7.4(a)]{HornJohnsonMatrixAnalysis}),
the fact that
$A\preceq B$ implies
$\lambda_i(A)
\leq
\lambda_i(B)$
for all $i=1,\dots,N$
(cf.~\citet[p.~471,~Corollary~7.7.4(c)]{HornJohnsonMatrixAnalysis})
and
(\ref{GolubVanLoan:Satz:SpektralNormVonPositivDefiniterMatrix}).
\qedhere
\end{proof}
\begin{proposition}
Let $\lbrace A_k\rbrace$ be a sequence of positive definite matrices $A_k\in\Sym{N}$. Then
\begin{equation}
\lbrace A_k\rbrace\textnormal{ is bounded}
~\Longleftrightarrow~
\lbrace A_k^{\mathalf}\rbrace\textnormal{ is bounded}
\label{ADDON:Satz:AkBeschraenktGdwAkRootBeschraenkt}
\end{equation}
and
\begin{equation}
\lbrace A_k\rbrace\textnormal{ is uniformly positive definite}
~\Longleftrightarrow~
\lbrace A_k^{-1}\rbrace\textnormal{ is bounded}
\punkt
\label{ADDON:Satz:AkGlmpdGdwAkInvereseBeschraenkt}
\end{equation}
\end{proposition}
\begin{proof}
Since $A_k\in\Sym{N}$ is positive definite due to assumption, there exists an eigenvalue decomposition $A_k=Q_k^T\Xi_kQ_k$
with $Q_k\in\mathbb{R}^{N\times N}$ orthogonal and a diagonal matrix $\Xi_k\in\mathbb{R}^{N\times N}$ with positive diagonal elements 
and we define $\mu_k:=\lambdamax(\Xi_k)$.
Since (\ref{GolubVanLoan:Satz:SpektralNormVonPositivDefiniterMatrix}) implies $\lvert A_k\rvert=\mu_k^{\mathalf}$
and
since $A_k^{\mathalf}=Q_k^T\Xi_k^{\mathalf}Q_k$ implies $\lvert A_k^{\mathalf}\rvert=\mu_k^{\matquarter}$, we obtain (\ref{ADDON:Satz:AkBeschraenktGdwAkRootBeschraenkt}).

(\ref{ADDON:Satz:AkGlmpdGdwAkInvereseBeschraenkt}) follows directly from the assumption of the uniform positive definiteness of $\lbrace A_k\rbrace$ and (\ref{GolubVanLoan:Satz:SpektralNormVonPositivDefiniterMatrix}).
\qedhere
\end{proof}
\begin{corollary}
\label{ADDON:Korollar:WpkUndHkBeschraenktheitUndPositivDefinitheit}
If $\lbrace(\widebar{W}_p^k)^{-\mathalf}\rbrace$ is bounded, then $\lbrace(\widebar{W}_p^k)^{-1}\rbrace$ and $\lbrace H_k\rbrace$ are bounded
\begin{equation}
\lvert(\widebar{W}_p^k)^{-1}\rvert
\leq
\WpInverseBound
\label{KonvergenzErweiterung:Satz:InverseVonWkBeschraenkt}
\end{equation}
for all $k\geq1$ with some positive constant $\WpInverseBound>0$.

If $\lbrace\bar{\kappa}^{k+1}\rbrace$ is bounded and $\lbrace(\widebar{W}_p^k)^{-\mathalf}\rbrace$ is uniformly positive definite, then $\lbrace H_k^{-1}\rbrace$ is bounded and
\cbstartMIFFLIN
\begin{equation}
\lvert\widebar{W}_p^k
+\widebar{G}^k
+\bar{\kappa}^{k+1}\widebar{\hat{G}}^k\rvert
\leq
C_1
\komma
\label{KonvergenzErweiterung:Satz:MatrixAbschaetzungMitHigham:MatritzenBeschraenktheit}
\end{equation}
\cbendMIFFLIN
for all $k\geq1$ with some positive constant $C_1>0$.
\end{corollary}
\begin{proof}
Since $(\widebar{W}_p^k)^{-\mathalf}=\big((\widebar{W}_p^k)^{-1}\big)^{\mathalf}$ is bounded due to assumption, $\lbrace(\widebar{W}_p^k)^{-1}\rbrace$ is bounded due to (\ref{ADDON:Satz:AkBeschraenktGdwAkRootBeschraenkt}) and therefore (\ref{KonvergenzErweiterung:Satz:InverseVonWkBeschraenkt}) holds with some positive constant $\WpInverseBound>0$, which is equivalent to the uniform positive definiteness of $\lbrace\widebar{W}_p^k\rbrace$ due to (\ref{ADDON:Satz:AkGlmpdGdwAkInvereseBeschraenkt}). Since $\widebar{W}_p^k\preceq H_k^{-2}$ for all
\cbstartMIFFLIN
$(\lambda^k,\mu^k)\geq\zeroVector{2(\lvert J_k\rvert+1)}$ with $\sum_{j\in J_k}{\lambda_j^k}+\lambda_p^k=1$
\cbendMIFFLIN
due to (\ref{BundleSQPmitQCQP:Alg:Hk}), we obtain $\lvert H_k^{2}\rvert\leq\WpInverseBound$ due to (\ref{HornJohnson:Erweiterung:InverseSpektralNormAbschaetzung}) and (\ref{KonvergenzErweiterung:Satz:InverseVonWkBeschraenkt}), which is equivalent that $\lbrace H_k\rbrace$ is bounded due to (\ref{ADDON:Satz:AkBeschraenktGdwAkRootBeschraenkt}).

Since $\lbrace\bar{\kappa}^{k+1}\rbrace$ is bounded due to assumption, there exists a positive constant $\kappaBound>0$ with $\bar{\kappa}^{k+1}\leq\kappaBound$ for all $k\geq1$ (note that $\bar{\kappa}^{k+1}\geq0$ due to (\ref{BundleSQPmitQCQP:Alg:Defkappabark+1})). Since $\lbrace(\widebar{W}_p^k)^{-\mathalf}\rbrace$ is uniformly positive definite due to assumption, $\lbrace(\widebar{W}_p^k)^{\mathalf}\rbrace$ is bounded due to (\ref{ADDON:Satz:AkGlmpdGdwAkInvereseBeschraenkt}), which is equivalent to $\lbrace\widebar{W}_p^k\rbrace$ being bounded due to (\ref{ADDON:Satz:AkBeschraenktGdwAkRootBeschraenkt}), i.e.~$\lvert\widebar{W}_p^k\rvert\leq\WpBound$ for some positive constant $\WpBound>0$ and for all $k\geq1$. Therefore, we obtain the boundedness of
\cbstartMIFFLIN
$\lvert H_k^{-2}\rvert\leq\WpBound+\bar{C}_G+\kappaBound\bar{\hat{C}}_G$
\cbendMIFFLIN
due to (\ref{BundleSQPmitQCQP:Alg:Hk}), (\ref{BundleSQPmitQCQP:Alg:Defkappabark+1}) and the initialization of Algorithm \ref{BundleSQPmitQCQP:Alg:GesamtAlgMitQCQP}, which is equivalent to $\lbrace H_k^{-1}\rbrace$ being bounded due to (\ref{ADDON:Satz:AkBeschraenktGdwAkRootBeschraenkt}). Furthermore, setting
$C_1:=\WpBound+\bar{C}_G+\kappaBound\bar{\hat{C}}_G$
yields (\ref{KonvergenzErweiterung:Satz:MatrixAbschaetzungMitHigham:MatritzenBeschraenktheit}) due to (\ref{BundleSQPmitQCQP:Alg:Defkappabark+1}) and the initialization of Algorithm \ref{BundleSQPmitQCQP:Alg:GesamtAlgMitQCQP}.
\qedhere
\end{proof}
\cbendDVI
From now on let the following assumption be satisfied.
\begin{presumption}
\label{COMPACT:presumptionForAuxiliaryTheorem:Assumption}
Let \refh{BundleSQPmitQCQP:GlobaleKonvergenz:Vor:AlgTerminiertNicht} be satisfied. Furthermore, let $\lbrace(x_k,\bar{\kappa}^{k+1})\rbrace$ be bounded and assume there exists $\bar{x}\in\mathbb{R}^N$ with $\sigma(\bar{x})=0$, where $\sigma:\mathbb{R}^N\longrightarrow\mathbb{R}$
\begin{equation}
\sigma(x):=\liminf_{k\rightarrow\infty}\max{(w_k,\lvert x_k-x\rvert)}
\punkt
\label{BundleSQPmitQCQP:GlobaleKonvergenz:Def:sigma}
\end{equation}
Moreover, let $\lbrace(\widebar{W}_p^k)^{-\mathalf}\rbrace$ be uniformly positive definite.
\end{presumption}
Next, we present Lemma \ref{COMPACT:lemmaForAuxiliaryTheorem:ConvergenceOfBasicSequences}--\ref{BundleSQPmitQCQP:GlobaleKonvergenz:Satz:Lemma3.6:BewqinSubdifferentialVonf}, which we need for proving Proposition \ref{AN:proposition:Lemma3.5(ii)}.
\begin{lemma}[Convergence of basic sequences]
\label{COMPACT:lemmaForAuxiliaryTheorem:ConvergenceOfBasicSequences}
There exist $K\subset\hat{K}\subset\lbrace1,2,\dots,\rbrace$ and $\bar{\kappa}\in\mathbb{R}$ such that
\begin{align}
x_k&\xrightarrow{\hat{K}}\bar{x}\komma\quad
w_k\xrightarrow{\hat{K}}0      \komma
\label{BundleSQPmitQCQP:GlobaleKonvergenz:Satz:Lemma3.6:BewxkANDwkKhat}
\\
\bar{\kappa}^{k+1}&\xrightarrow{K}\bar{\kappa}\komma
\label{BundleSQPmitQCQP:GlobaleKonvergenz:Satz:Lemma3.6:Bewkappabark+1}
\\
x_k&\xrightarrow{K}\bar{x}\komma\quad
w_k\xrightarrow{K}0       \punkt
\label{BundleSQPmitQCQP:GlobaleKonvergenz:Satz:Lemma3.6:BewxkCOMPACT}
\end{align}
\end{lemma}
\begin{proof}
Since we have $0=\sigma(\bar{x})=\liminf_{k\rightarrow\infty}\max{(w_k,\lvert x_k-\bar{x}\rvert)}$ due to assumption and (\ref{BundleSQPmitQCQP:GlobaleKonvergenz:Def:sigma}) and since $w_k\geq0$ due to (\ref{BundleSQPmitQCQP:Satz:GlobaleKonvergenz:Lemma3.5:wkCOMPACT}), there exist convergent subsequences of $\lbrace w_k\rbrace_{k\geq1}$ and $\lbrace x_k-\bar{x}\rbrace_{k\geq1}$, i.e.~there exists (an infinite set) $\hat{K}\subset\lbrace1,2,\dots,\rbrace$ such that
(\ref{BundleSQPmitQCQP:GlobaleKonvergenz:Satz:Lemma3.6:BewxkANDwkKhat}) holds.
Since $\lbrace\bar{\kappa}^{k+1}\rbrace_k$ is bounded by assumption, all its subsequences are also bounded. Therefore, in particular, its subsequence $\lbrace\bar{\kappa}^{k+1}\rbrace_{k\in\hat{K}}$ is bounded. Consequently, $\lbrace\bar{\kappa}^{k+1}\rbrace_{k\in\hat{K}}$ has an accumulation point, i.e.~there exists (an infinite set) $K\subset\hat{K}$ and $\bar{\kappa}\in\mathbb{R}$ such that
(\ref{BundleSQPmitQCQP:GlobaleKonvergenz:Satz:Lemma3.6:Bewkappabark+1}) holds.
Since $\bar{\kappa}^{k+1}\geq0$ for $k=1,2,\dots$ due to (\ref{BundleSQPmitQCQP:Alg:Defkappabark+1}), we have $\bar{\kappa}\in\Rpos$. Since $K\subset\hat{K}$ and a sequence is convergent, if and only if all of its subsequences converge towards the same limit, (\ref{BundleSQPmitQCQP:GlobaleKonvergenz:Satz:Lemma3.6:BewxkANDwkKhat}) yields
(\ref{BundleSQPmitQCQP:GlobaleKonvergenz:Satz:Lemma3.6:BewxkCOMPACT}).
\qedhere
\end{proof}
\begin{lemma}[Lagrange multipliers]
\label{COMPACT:lemmaForAuxiliaryTheorem:LagrangeMulitpliers}
Let $I:=\lbrace1,2,\dots,N+2\rbrace$ (Note: $\mathrm{card}(I)=n+2$),
$S:=\lbrace (g_j^k,s_j^k):j=1,\dots,k\rbrace\subseteq\mathbb{R}^{N+1}$, and
$\hat{S}:=\lbrace(\hat{g}_j^k,s_j^k):j=1,\dots,k\rbrace\subseteq\mathbb{R}^{N+1}$.
Then for $i\in I$ and $k\geq1$ there exist $\lambda^{k,i},\kappa^{k,i}\in\mathbb{R}$
and $(g^{k,i},s^{k,i})\in S,(\hat{g}^{k,i},\hat{s}^{k,i})\in\hat{S}$ such that
\begin{align}
\tilde{g}_p^k&=    \sum_{i\in I}\lambda^{k,i}g^{k,i}\komma\quad
\tilde{s}_p^k=    \sum_{i\in I}\lambda^{k,i}s^{k,i}\komma\quad
            1=    \sum_{i\in I}\lambda^{k,i}       \komma\quad
\lambda^{k,i}\geq0                                 \punkt
\label{BundleSQPmitQCQP:GlobaleKonvergenz:Satz:Lemma3.6:BewlambdaikCOMPACT}
\\
\tilde{\hat{g}}_p^k&=\sum_{i\in I}\kappa^{k,i}\hat{g}^{k,i}                                   \komma\quad
\tilde{\hat{s}}_p^k=\sum_{i\in I}\kappa^{k,i}\hat{s}^{k,i}                                   \komma\nonumber\\
&\hspace{0.10\textwidth}(1=\sum_{i\in I}\kappa^{k,i}~\wedge~\kappa^{k,i}\geq0)~\xor~(\kappa^{k,i}=0~~~\forall i\in I)\punkt
\label{BundleSQPmitQCQP:GlobaleKonvergenz:Satz:Lemma3.6:BewlambdaikNBCOMPACT}
\end{align}
In particular, we have
\begin{equation}
\sum_{i\in I}{\kappa^{k,i}}
=
\left\lbrace
\begin{array}{ll}
1 & \textnormal{if }\bar{\kappa}^{k+1}>0\\
0 & \textnormal{if }\bar{\kappa}^{k+1}=0\punkt
\end{array}
\right.
\label{BundleSQPmitQCQP:GlobaleKonvergenz:Satz:Lemma3.6:BewlambdaikNB:SummeNachtrag}
\end{equation}
\end{lemma}
\begin{proof}
We have $(\tilde{g}_p^k,\tilde{s}_p^k)\in\mathrm{ch}(S)$ due to (\ref{BundleSQPmitQCQP:GlobaleKonvergenz:Vor:AlgTerminiertNicht}) and
(\ref{BundleSQPmitQCQP:GlobaleKonvergenzSatz:Lemma3.1:22:1COMPACT}). Due to Carath\'eodory's theorem 
(cf., e.g., \citet{NeumaierIV}), for $i\in I$ and $k\geq1$ there exist $(g^{k,i},s^{k,i})\in S$
and $\lambda^{k,i}\in\mathbb{R}$ such that
(\ref{BundleSQPmitQCQP:GlobaleKonvergenz:Satz:Lemma3.6:BewlambdaikCOMPACT}) holds.
Furthermore, we have $(\tilde{\hat{g}}_p^k,\tilde{\hat{s}}_p^k)\in\mathrm{ch}(\hat{S})$ for $\bar{\kappa}^{k+1}>0$ and 
$(\tilde{\hat{g}}_p^k,\tilde{\hat{s}}_p^k)=\zeroVector{n+1}$
for $\bar{\kappa}^{k+1}=0$ due to (\ref{BundleSQPmitQCQP:GlobaleKonvergenz:Vor:AlgTerminiertNicht}) and (\ref{BundleSQPmitQCQP:GlobaleKonvergenzSatz:Lemma3.1:22:1NBCOMPACT}). In the case $\bar{\kappa}^{k+1}>0$ there exist $(\hat{g}^{k,i},\hat{s}^{k,i})\in\hat{S}$, $\kappa^{k,i}\in\mathbb{R}$ for $i\in I$ with $1=\sum_{i\in I}\kappa^{k,i}$, $\kappa^{k,i}\geq0$ and $(\tilde{\hat{g}}_p^k,\tilde{\hat{s}}_p^k)=\sum_{i\in I}\kappa^{k,i}(\hat{g}^{k,i},\hat{s}^{k,i})$ due to Carath\'eodory's theorem 
(cf., e.g., \citet{NeumaierIV}). In the case $\bar{\kappa}^{k+1}=0$ choosing $\kappa^{k,i}:=0$ for all $i\in I$ yields 
$(\tilde{\hat{g}}_p^k,\tilde{\hat{s}}_p^k)=\zeroVector{n+1}=\sum_{i\in I}{\kappa^{k,i}(\hat{g}^{k,i},\hat{s}^{k,i})}$.
Hence,
(\ref{BundleSQPmitQCQP:GlobaleKonvergenz:Satz:Lemma3.6:BewlambdaikNBCOMPACT}) holds,
which immediately implies
(\ref{BundleSQPmitQCQP:GlobaleKonvergenz:Satz:Lemma3.6:BewlambdaikNB:SummeNachtrag}).
\qedhere
\end{proof}
\begin{lemma}[Assignment]
\label{BundleSQPmitQCQP:GlobaleKonvergenz:Satz:Lemma3.6:BewgkiCOMPACT}
There exists $j(k,i)\in\lbrace1,\dots,k\rbrace$ (i.e.~a function $j:\lbrace k\in\mathbb{N}:~k\geq1\rbrace\times I\longrightarrow\lbrace1,\dots,k\rbrace$) with
$g^{k,i}=g_{j(k,i)}^k$,
$s^{k,i}=s_{j(k,i)}^k$,
$\hat{g}^{k,i}=\hat{g}_{j(k,i)}^k$, and
$\hat{s}^{k,i}=\hat{s}_{j(k,i)}^k$.
\end{lemma}
\begin{proof}
Use $(g^{k,i},s^{k,i})\in S$ and $(\hat{g}^{k,i},\hat{s}^{k,i})\in\hat{S}$ for $i\in I$ and $k\geq1$
from Lemma \ref{COMPACT:lemmaForAuxiliaryTheorem:LagrangeMulitpliers}.
\qedhere
\end{proof}
\begin{lemma}[Trial point convergence \& implications]
\label{COMPACT:lemmaForAuxiliaryTheorem:TrialPointConvergenceAndImplications}
For all $i\in I$ there exist $\bar{y}_i\in\mathbb{R}^N$ and (an infinite set) $K_3\subset K_2\subset K_1\subset K$ with
\begin{align}
y_{j(k,i)}
&\xrightarrow{K_1}\bar{y}_i
\punkt
\label{BundleSQPmitQCQP:GlobaleKonvergenz:Satz:Lemma3.6:Bewyki}
\\
(g_{j(k,i)},\hat{g}_{j(k,i)})
&\xrightarrow{K_2}(\bar{g}_i,\bar{\hat{g}}_i)
\in\partial f(\bar{y}_i)\times\partial F(\bar{y}_i)
\label{BundleSQPmitQCQP:GlobaleKonvergenz:Satz:Lemma3.6:BewgjkiCOMPACT}
\\
(\rho_{j(k,i)}G_{j(k,i)},\lambda^{k,i},\hat{\rho}_{j(k,i)}\hat{G}_{j(k,i)},\kappa^{k,i})
&\xrightarrow{K_3}
(\widebar{G}_i,\bar{\lambda}_i,\widebar{\hat{G}}_i,\bar{\kappa}_i)
\punkt
\label{BundleSQPmitQCQP:GlobaleKonvergenz:Satz:Lemma3.6:BewGkiCOMPACT}
\end{align}
\end{lemma}
\begin{proof}
Since $\lvert y_{j(k,i)}\rvert\leq\lvert x_{j(k,i)}\rvert+C_S$ holds for all $i\in I$ and for all $k\geq1$ due to (\ref{Luksan:Liniensuche:PunktBeschraenktheit}), the assumption of the boundedness of $\lbrace x_k\rbrace$ yields that
$
\lbrace y_{j(k,i)}\rbrace_{k\geq1,i\in I}
$ is bounde
and therefore it has a convergent subsequence, i.e.~(\ref{BundleSQPmitQCQP:GlobaleKonvergenz:Satz:Lemma3.6:Bewyki}) holds.
Furthermore, the local boundedness of $\partial f$ resp.~$\partial F$ (cf.~Proposition \ref{PAPER:LocallyBoundedUpperSemiContinous}) 
imply that the sets $B_1:=\lbrace g\in\partial f(y_{j(k,i)}):~y_{j(k,i)}\in\mathbb{R}^N,~k\geq1,~k\in K_1,~i\in I\rbrace$ and $B_2:=\lbrace\hat{g}\in\partial F(y_{j(k,i)}):~y_{j(k,i)}\in\mathbb{R}^N,~k\geq1,~k\in K_1,~i\in I\rbrace$ are bounded. Therefore, $B_1\times B_2$ is bounded and consequently there exists a convergent subsequence $(g_{j(k,i)},\hat{g}_{j(k,i)})\in\partial f(y_{j(k,i)})\times\partial F(y_{j(k,i)})$, i.e.~there exists $(\bar{g}_i,\bar{\hat{g}}_i)\in\mathbb{R}^N\times\mathbb{R}^N$ and (an infinite set) $K_2\subset K_1$ with $(g_{j(k,i)},\hat{g}_{j(k,i)})\xrightarrow{K_2}(\bar{g}_i,\bar{\hat{g}}_i)$. The upper semicontinuity of $\partial f$ resp.~$\partial F$ (cf.~Proposition \ref{PAPER:LocallyBoundedUpperSemiContinous})
and (\ref{BundleSQPmitQCQP:GlobaleKonvergenz:Satz:Lemma3.6:Bewyki}) imply that for all $i\in I$
(\ref{BundleSQPmitQCQP:GlobaleKonvergenz:Satz:Lemma3.6:BewgjkiCOMPACT}) holds.

Since $\rho_{j(k,i)}\in(0,1]$ due to (\ref{BundleSQPmitQCQP:Alg:rhok+1}) and $C_G>0$, we obtain
$\rho_{j(k,i)}\lvert G_{j(k,i)}\rvert\leq C_G$,
which yields the boundedness of the sequence $\lbrace\rho_{j(k,i)}\lvert G_{j(k,i)}\rvert\rbrace$. Due to (\ref{BundleSQPmitQCQP:GlobaleKonvergenz:Satz:Lemma3.6:BewlambdaikCOMPACT}), the sequence $\lbrace\lambda^{k,i}\rbrace$ is bounded. Since $\hat{\rho}_{j(k,i)}\in(0,1]$ due to (\ref{BundleSQPmitQCQP:Alg:rhohatk+1}) and $\hat{C}_G>0$, we obtain
$
\hat{\rho}_{j(k,i)}\lvert\hat{G}_{j(k,i)}\rvert\leq\hat{C}_G
$,
which yields the boundedness of the sequence $\lbrace\hat{\rho}_{j(k,i)}\lvert\hat{G}_{j(k,i)}\rvert\rbrace$. Due to (\ref{BundleSQPmitQCQP:GlobaleKonvergenz:Satz:Lemma3.6:BewlambdaikNBCOMPACT}), the sequence $\lbrace\kappa^{k,i}\rbrace$ is bounded. Therefore, the sequence $\lbrace\rho_{j(k,i)}\lvert G_{j(k,i)}\rvert,\lambda^{k,i},\hat{\rho}_{j(k,i)}\lvert \hat{G}_{j(k,i)}\rvert,\kappa^{k,i}\rbrace$ is bounded. Consequently, there exists a convergent subsequence of $\lbrace\rho_{j(k,i)}\lvert G_{j(k,i)}\rvert,\lambda^{k,i},\hat{\rho}_{j(k,i)}\lvert G_{j(k,i)}\rvert,\kappa^{k,i}\rbrace$, i.e.~for all $i\in I$ there exist $\widebar{G}_i,\widebar{\hat{G}}_i\in\mathbb{R}^{N\times N}$, $\bar{\lambda}_i,\bar{\kappa}_i\in\mathbb{R}$ and (an infinite set) $K_3\subset K_2$ such that
(\ref{BundleSQPmitQCQP:GlobaleKonvergenz:Satz:Lemma3.6:BewGkiCOMPACT}) holds.
\qedhere
\end{proof}
\begin{lemma}[Complementarity condition]
\label{COMPACT:lemmaForAuxiliaryTheorem:ComplementarityCondition}
We have
\begin{align}
\sum_{i\in I}\bar{\lambda}_i\big(\bar{g}_i+\widebar{G}_i(\bar{x}-\bar{y}_i)\big)
+
\bar{\kappa}
\sum_{i\in I}\bar{\kappa}_i\big(\bar{\hat{g}}_i+\widebar{\hat{G}}_i(\bar{x}-\bar{y}_i)\big)
&=\zeroVector{N}
\label{BundleSQPmitQCQP:GlobaleKonvergenz:Satz:Lemma3.6:BewSum=o}
\\
\lambda^{k,i}s^{k,i}&\xrightarrow{K_3}0
\label{BundleSQPmitQCQP:GlobaleKonvergenz:Satz:Lemma3.6:Bewlambdaikski0}
\\
\kappa^{k,i}\hat{s}^{k,i}
&\xrightarrow{K_3}0\textnormal{ if }\bar{\kappa}>0
\punkt
\label{BundleSQPmitQCQP:GlobaleKonvergenz:Satz:Lemma3.6:Bewlambdaikski0NB}
\end{align}
Furthermore, the complementarity condition $\bar{\kappa}F(\bar{x})=0$ holds.
\end{lemma}
\begin{proof}
We calculate
$
\tilde{g}_p^k
\xrightarrow{K_3}
\sum_{i\in I}\bar{\lambda}_i\big(\bar{g}_i+\widebar{G}_i(\bar{x}-\bar{y}_i)\big)
$
and
$
\tilde{\hat{g}}_p^k
\xrightarrow{K_3}
\sum_{i\in I}\bar{\kappa}_i\big(\bar{\hat{g}}_i+\widebar{\hat{G}}_i(\bar{x}-\bar{y}_i)\big)
$
by using
(\ref{BundleSQPmitQCQP:GlobaleKonvergenz:Satz:Lemma3.6:BewlambdaikCOMPACT}),
(\ref{BundleSQPmitQCQP:GlobaleKonvergenz:Satz:Lemma3.6:BewlambdaikNBCOMPACT}),
Lemma \ref{BundleSQPmitQCQP:GlobaleKonvergenz:Satz:Lemma3.6:BewgkiCOMPACT},
(\ref{BundleSQP:Def:gjkCOMPACT}),
(\ref{BundleSQPmitQCQP:GlobaleKonvergenz:Satz:Lemma3.6:BewGkiCOMPACT}),
(\ref{BundleSQPmitQCQP:GlobaleKonvergenz:Satz:Lemma3.6:BewgjkiCOMPACT}),
(\ref{BundleSQPmitQCQP:GlobaleKonvergenz:Satz:Lemma3.6:BewxkCOMPACT}) and
(\ref{BundleSQPmitQCQP:GlobaleKonvergenz:Satz:Lemma3.6:Bewyki})
\cbstartDVI
Since $\lbrace\bar{\kappa}^{k+1}\rbrace$ is bounded and $\lbrace(\widebar{W}_p^k)^{-\mathalf}\rbrace$ is uniformly positive definite (both due to assumption), Corollary \ref{ADDON:Korollar:WpkUndHkBeschraenktheitUndPositivDefinitheit} implies
\cbendDVI
the boundedness of $\lbrace H_k^{-1}\rbrace$. Because of (\ref{BundleSQPmitQCQP:GlobaleKonvergenz:Satz:Lemma3.6:BewxkCOMPACT}), (\ref{BundleSQPmitQCQP:Alg:wk}) and (\ref{BundleSQPmitQCQP:Satz:GlobaleKonvergenz:Lemma3.5:wkCOMPACT}), we have $\lvert H_k(\tilde{g}_p^k+\bar{\kappa}^{k+1}\tilde{\hat{g}}_p^k)\rvert\xrightarrow{K_3}0$, which implies
(\ref{BundleSQPmitQCQP:GlobaleKonvergenz:Satz:Lemma3.6:BewSum=o})
due to the regularity of $H_k$,
(\ref{BundleSQPmitQCQP:GlobaleKonvergenz:Satz:Lemma3.6:Bewkappabark+1}) and
the uniqueness of a limit and $\tilde{\alpha}_p^k\xrightarrow{K_3}0$, which implies
(\ref{BundleSQPmitQCQP:GlobaleKonvergenz:Satz:Lemma3.6:Bewlambdaikski0})
due to
(\ref{BundleSQPmitQCQP:Alg:alphatildepk}),
(\ref{BundleSQPmitQCQP:GlobaleKonvergenz:Satz:Lemma3.6:BewlambdaikCOMPACT}),
Lemma \ref{BundleSQPmitQCQP:GlobaleKonvergenz:Satz:Lemma3.6:BewgkiCOMPACT},
(\ref{BundleSQPmitQCQP:GlobaleKonvergenz:Satz:Lemma3.6:BewlambdaikCOMPACT}) and
(\ref{Luksan:Def:Lokalitaetsmass}),
as well as $\bar{\kappa}^{k+1}F(x_k)\xrightarrow{K_3}0$, which implies $0=\bar{\kappa}F(\bar{x})$ due to (\ref{BundleSQPmitQCQP:GlobaleKonvergenz:Satz:Lemma3.6:Bewkappabark+1}), the continuity of $F$ and (\ref{BundleSQPmitQCQP:GlobaleKonvergenz:Satz:Lemma3.6:BewxkCOMPACT}), as well as $\bar{\kappa}^{k+1}\tilde{A}_p^k\xrightarrow{K_3}0$ which implies for $\bar{\kappa}>0$
that (\ref{BundleSQPmitQCQP:GlobaleKonvergenz:Satz:Lemma3.6:Bewlambdaikski0NB}) holds
due to
(\ref{BundleSQPmitQCQP:GlobaleKonvergenz:Satz:Lemma3.6:Bewkappabark+1}),
(\ref{BundleSQPmitQCQP:Alg:Atildepk}),
(\ref{BundleSQPmitQCQP:GlobaleKonvergenz:Satz:Lemma3.6:BewlambdaikNBCOMPACT}),
Lemma \ref{BundleSQPmitQCQP:GlobaleKonvergenz:Satz:Lemma3.6:BewgkiCOMPACT},
(\ref{BundleSQPmitQCQP:GlobaleKonvergenz:Satz:Lemma3.6:BewlambdaikNBCOMPACT}) and
(\ref{Luksan:Def:Lokalitaetsmass}).
\qedhere
\end{proof}
\begin{lemma}[Subdifferential elements]
\label{BundleSQPmitQCQP:GlobaleKonvergenz:Satz:Lemma3.6:BewqinSubdifferentialVonf}
We have
\begin{equation*}
\sum_{i\in I}\bar{\lambda}_i\big(\bar{g}_i+\widebar{G}_i(\bar{x}-\bar{y}_i)\big)\in\partial f(\bar{x})
\komma
\quad
\left\lbrace
\begin{array}{ll}
\sum\limits_{i\in I}\bar{\kappa}_i\big(\bar{\hat{g}}_i+\widebar{\hat{G}}_i(\bar{x}-\bar{y}_i)\big)\in\partial F(\bar{x}) & \textnormal{if }\bar{\kappa}>0\\
\lbrace\zeroVector{N}\rbrace=\bar{\kappa}\partial F(\bar{x})                                                             & \textnormal{if }\bar{\kappa}=0\punkt
\end{array}
\right.
\end{equation*}
\end{lemma}
\begin{proof}
Since (\ref{BundleSQPmitQCQP:GlobaleKonvergenz:Satz:Lemma3.6:BewlambdaikCOMPACT}) holds for all $k\in K_3$, (\ref{BundleSQPmitQCQP:GlobaleKonvergenz:Satz:Lemma3.6:BewGkiCOMPACT}) implies
$
\sum_{i\in I}\bar{\lambda}_i
=
1
$.
Due to (\ref{BundleSQPmitQCQP:GlobaleKonvergenz:Satz:Lemma3.6:BewGkiCOMPACT}),
we have
$
\lim_{K_3}\sum_{i\in I}{\kappa^{k,i}}
=
\sum_{i\in I}\bar{\kappa}_i
$.
If $\bar{\kappa}>0$, then --- because of (\ref{BundleSQPmitQCQP:GlobaleKonvergenz:Satz:Lemma3.6:Bewkappabark+1})
and since $K_3(\subset K)$ is an infinite set --- there exists $\hat{k}\in K_3$ such that
$
\lvert\bar{\kappa}^{k+1}-\bar{\kappa}\rvert
<
\tfrac{\bar{\kappa}}{2}
$, which implies
$
0
<
\tfrac{\bar{\kappa}}{2}
<
\bar{\kappa}^{k+1}
$
for all $k\in\hat{K}_3$, where
$\hat{K}_3
:=
\left\lbrace
k\in K_3:~k\geq\hat{k}
\right\rbrace
\subseteq
K_3$
is an infinite set. Therefore, we obtain
$
\sum_{i\in I}{\kappa^{k,i}}
=1
$
for all $k\in\hat{K}_3$
due to (\ref{BundleSQPmitQCQP:GlobaleKonvergenz:Satz:Lemma3.6:BewlambdaikNB:SummeNachtrag}),
i.e.~$\lbrace\sum_{i\in I}{\kappa^{k,i}}\rbrace_{k\in\hat{K}_3}$ is constant on $\hat{K}_3$ and hence we have
$
\lim_{\hat{K}_3}\sum_{i\in I}{\kappa^{k,i}}
=
1
\punkt
$
Since
the sequence $\lbrace\sum_{i\in I}{\kappa^{k,i}}\rbrace_{k\in K_3}$ is convergent,
the (infinite) subsequence $\lbrace\sum_{i\in I}{\kappa^{k,i}}\rbrace_{k\in\hat{K}_3}$ (of the sequence $\lbrace\sum_{i\in I}{\kappa^{k,i}}\rbrace_{k\in K_3}$) converges towards $1$
and
a sequence is convergent if and only if all its subsequences converge towards the same limit,
the limit of the sequence $\lbrace\sum_{i\in I}{\kappa^{k,i}}\rbrace_{k\in K_3}$ must be $1$.
Consequently,
we obtatin for $\bar{\kappa}>0$ that
$
\sum_{i\in I}\bar{\kappa}_i
=
1
$.

Due to (\ref{BundleSQPmitQCQP:GlobaleKonvergenz:Satz:Lemma3.6:Bewlambdaikski0}) the sequence $\lbrace\lambda^{k,i}s^{k,i}\rbrace_{k\in K_3}$ is convergent and therefore necessarily bounded, i.e.~there exists $C>0$ with $0\leq s^{k,i}\leq\tfrac{C}{\lambda^{k,i}}$ due to
Lemma \ref{BundleSQPmitQCQP:GlobaleKonvergenz:Satz:Lemma3.6:BewgkiCOMPACT}
as well as (\ref{Luksan:Def:Lokalitaetsmass}) and therefore $\lbrace s^{k,i}\rbrace_{k\in K_3}$ is bounded due to (\ref{BundleSQPmitQCQP:GlobaleKonvergenz:Satz:Lemma3.6:BewGkiCOMPACT})
for
$\bar{\lambda}_i\not=0$,
where at least one such $\bar{\lambda}_i$ exists
because $\sum_{i\in I}\bar{\lambda}_i=1$.
Since the locality measure is monotone due to (\ref{Luksan:Satz:LokalitaetsmassMonotonCOMPACT}), $\lbrace s^{k,i}\rbrace_{k\in K_3}$ is monotone. Consequently, $\lbrace s^{k,i}\rbrace_{k\in K_3}$ is convergent for $\bar{\lambda}_i\not=0$, i.e.~there exists $s_i:=\lim_{K_3}s^{k,i}$. Therefore, (\ref{BundleSQPmitQCQP:GlobaleKonvergenz:Satz:Lemma3.6:Bewlambdaikski0}),
(\ref{BundleSQPmitQCQP:GlobaleKonvergenz:Satz:Lemma3.6:BewGkiCOMPACT}) and $\bar{\lambda}_i\not=0$ imply $s_i=0$. Hence, we obtain for $\bar{\lambda}_i\not=0$
that
$\lvert\bar{x}-\bar{y}_i\rvert=0$
due to
Lemma \ref{BundleSQPmitQCQP:GlobaleKonvergenz:Satz:Lemma3.6:BewgkiCOMPACT},
(\ref{Luksan:Satz:LokalitaetsmassMonotonCOMPACT}),
(\ref{BundleSQPmitQCQP:GlobaleKonvergenz:Satz:Lemma3.6:Bewyki}) and
(\ref{BundleSQPmitQCQP:GlobaleKonvergenz:Satz:Lemma3.6:BewxkCOMPACT}).
For $\bar{\kappa}>0$
the sequence $\lbrace\kappa^{k,i}s^{k,i}\rbrace_{k\in K_3}$ is convergent
due to (\ref{BundleSQPmitQCQP:GlobaleKonvergenz:Satz:Lemma3.6:Bewlambdaikski0NB})
and therefore necessarily bounded, i.e.~there exists $\hat{C}>0$ with $0\leq\hat{s}^{k,i}\leq\tfrac{\hat{C}}{\kappa^{k,i}}$ due to
Lemma \ref{BundleSQPmitQCQP:GlobaleKonvergenz:Satz:Lemma3.6:BewgkiCOMPACT}
as well as (\ref{Luksan:Def:Lokalitaetsmass}) and therefore $\lbrace\hat{s}^{k,i}\rbrace_{k\in K_3}$ is bounded due to (\ref{BundleSQPmitQCQP:GlobaleKonvergenz:Satz:Lemma3.6:BewGkiCOMPACT})
for
$\bar{\kappa}_i\not=0$,
where at least one such $\bar{\kappa}_i$ exists
because $\sum_{i\in I}\bar{\kappa}_i=1$.
Since the locality measure is monotone due to
(\ref{Luksan:Def:Lokalitaetsmass})
and (\ref{Luksan:Satz:LokalitaetsmassMonotonCOMPACT}), $\lbrace\hat{s}^{k,i}\rbrace_{k\in K_3}$ is monotone. Consequently, $\lbrace \hat{s}^{k,i}\rbrace_{k\in K_3}$ is convergent for $\bar{\kappa}_i\not=0$, i.e.~there exists $\hat{s}_i:=\lim_{K_3}\hat{s}^{k,i}$. Therefore, (\ref{BundleSQPmitQCQP:GlobaleKonvergenz:Satz:Lemma3.6:Bewlambdaikski0NB}), (\ref{BundleSQPmitQCQP:GlobaleKonvergenz:Satz:Lemma3.6:BewGkiCOMPACT}) and $\bar{\kappa}_i\not=0$ imply $\hat{s}_i=0$. Hence, we obtain
in the case $\bar{\kappa}>0$
for $\bar{\kappa}_i\not=0$
that
$\lvert\bar{x}-\bar{y}_i\rvert=0$
due to
Lemma \ref{BundleSQPmitQCQP:GlobaleKonvergenz:Satz:Lemma3.6:BewgkiCOMPACT},
(\ref{Luksan:Def:Lokalitaetsmass}),
(\ref{Luksan:Satz:LokalitaetsmassMonotonCOMPACT}), (\ref{BundleSQPmitQCQP:GlobaleKonvergenz:Satz:Lemma3.6:Bewyki}) and (\ref{BundleSQPmitQCQP:GlobaleKonvergenz:Satz:Lemma3.6:BewxkCOMPACT}).
Therefore, if $\bar{\lambda}_i\not=0$ resp.~if $\bar{\kappa}>0$ and $\bar{\kappa}_i\not=0$, then
$
\lvert\bar{x}-\bar{y}_i\rvert=0
$.

If we set
$
\bar{q}
:=
\sum_{i\in I}\bar{\lambda}_i\big(\bar{g}_i+\widebar{G}_i(\bar{x}-\bar{y}_i)\big)
$,
$
\bar{s}_i
:=
\left\lbrace
\begin{array}{ll}
\lvert\bar{x}-\bar{y}_i\rvert & \textnormal{for }\bar{\lambda}_i    =0      \\
0                             & \textnormal{for }\bar{\lambda}_i\not=0
\end{array}
\right\rbrace
$
resp.~if we set
$
\bar{q}'
:=
\sum_{i\in I}\bar{\kappa}_i\big(\bar{\hat{g}}_i+\widebar{\hat{G}}_i(\bar{x}-\bar{y}_i)\big)
$,
$
\bar{s}_i'
:=
\left\lbrace
\begin{array}{ll}
\lvert\bar{x}-\bar{y}_i\rvert & \textnormal{for }\bar{\kappa}_i    =0      \\
0                             & \textnormal{for }\bar{\kappa}_i\not=0
\end{array}
\right\rbrace
$ in the case $\bar{\kappa}>0$,
then the assumptions of Proposition \ref{Luksan:GlobaleKonvergenzSatz:Lemma3.2} are satisfied 
and therefore we obtain
the first two desired results.
Since $F$ is locally Lipschitz continuous, $\partial F(\bar{x})$ is in particular bounded due to 
Proposition \ref{PAPER:LocallyBoundedUpperSemiContinous}
and consequently  we obtain
$\bar{\kappa}\partial F(\bar{x})=\lbrace\zeroVector{N}\rbrace$
in the case $\bar{\kappa}=0$.
\qedhere
\end{proof}
\begin{proposition}
\label{BundleSQPmitQCQP:GlobaleKonvergenz:Satz:Lemma3.6}
Let Assumption \refH{COMPACT:presumptionForAuxiliaryTheorem:Assumption} be satisfied.
Then there exists $\bar{\kappa}\in\Rpos$
such that \refh{AOB:Satz:KarushJohnUnglgsNBNichtGlattAlternativformGlgsSystem} holds for $(\bar{x},\bar{\kappa})$,
i.e.~if the sequence of iteration points and (single) Lagrange multipliers is bounded and the sequence of iteration points has an accumulation point with $\sigma(\bar{x})=0$, then this accumulation point is stationary for the optimization problem \refh{BundleSQP:OptProblem}.
\end{proposition}
\begin{proof}
Due to (\ref{BundleSQPmitQCQP:ZulaessigkeitsbedingungFuerIterationspunkt}),
the continuity of $F$ and
(\ref{BundleSQPmitQCQP:GlobaleKonvergenz:Satz:Lemma3.6:BewxkCOMPACT}),
we obtain $F(\bar{x})\leq0$.
Due to Lemma \ref{COMPACT:lemmaForAuxiliaryTheorem:ComplementarityCondition}, the complementarity condition $\bar{\kappa}F(\bar{x})=0$ holds.
Using (\ref{BundleSQPmitQCQP:GlobaleKonvergenz:Satz:Lemma3.6:BewSum=o}) and
Lemma \ref{BundleSQPmitQCQP:GlobaleKonvergenz:Satz:Lemma3.6:BewqinSubdifferentialVonf},
we calculate $\zeroVector{N}\in\partial f(\bar{x})+\bar{\kappa}\partial F(\bar{x})$.
\qedhere
\end{proof}
\begin{proposition}
\label{AN:proposition:Lemma3.5(ii)}
Let \refh{BundleSQPmitQCQP:GlobaleKonvergenz:Vor:AlgTerminiertNicht} be satisfied. If there exist $\bar{x}\in\mathbb{R}^N$ and $K\subset\lbrace1,2,\dots\rbrace$ with $x\xrightarrow{K}\bar{x}$, then
\begin{equation}
t_L^kv_k
\xrightarrow{K}0\punkt\label{BundleSQPmitQCQP:Satz:GlobaleKonvergenz:Lemma3.5:tLkvk0}
\end{equation}
\end{proposition}
\begin{proof}
\citet[Proof~of~Lemma~3.5(ii)]{Luksan}.
\qedhere
\end{proof}
\begin{proposition}
\label{BundleSQPmitQCQP:GlobaleKonvergenz:Lemma3.7}
Let \refh{BundleSQPmitQCQP:GlobaleKonvergenz:Vor:AlgTerminiertNicht} be satisfied, let the sequence of (symmetric, positive definite matrices) $\lbrace H_k\rbrace$ be bounded and assume that there exists an infinite subset $K\subset\lbrace1,2,\dots\rbrace$ and $\bar{x}\in\mathbb{R}^N$ with
\begin{equation}
x_k
\xrightarrow{K}\bar{x}
\punkt
\label{BundleSQPmitQCQP:GlobaleKonvergenz:Lemma3.7:Vorx}
\end{equation}
Then we have for all $i\geq0$
\begin{equation}
x_{k+i}\xrightarrow{k\xrightarrow{K}\infty}\bar{x}\punkt
\label{BundleSQPmitQCQP:GlobaleKonvergenz:Lemma3.7:Aussagexk+i}
\end{equation}
If additionally $\sigma(\bar{x})>0$ holds, then
we have for all $i\geq0$
\begin{equation}
t_L^{k+i}\xrightarrow{k\xrightarrow{K}\infty}0
\komma
\label{BundleSQPmitQCQP:GlobaleKonvergenz:Lemma3.7:AussagetLk+i}
\end{equation}
and for fixed $\varepsilon_0>0$ and for all fixed $r\geq0$ there exists $\tilde{k}\geq0$ such that
\begin{equation}
w_{k+i}\geq\tfrac{\sigma(\bar{x})}{2}
\komma\quad
t_L^{k+i}<\varepsilon_0
\label{BundleSQPmitQCQP:GlobaleKonvergenz:Lemma3.7:AussagetLk+it0COMPACT}
\end{equation}
for all $k>\tilde{k}$, $k\in K$ and $0\leq i\leq r$.
\end{proposition}
\begin{proof}
We show (\ref{BundleSQPmitQCQP:GlobaleKonvergenz:Lemma3.7:Aussagexk+i}) by induction: The base case holds for $i=0$ due to assumption (\ref{BundleSQPmitQCQP:GlobaleKonvergenz:Lemma3.7:Vorx}). Now, let the induction hypothesis be satisfied for $i\geq0$. We have
\begin{equation}
d_{k+i}
=
H_{k+i}^2(\tilde{g}_p^{k+i}+\bar{\kappa}^{k+i+1}\tilde{\hat{g}}_p^{k+i})
\label{BundleSQPmitQCQP:GlobaleKonvergenz:Satz:Lemma3.7:Bewxk+1:ProofNEW1}
\end{equation}
due to (\ref{BundleSQPmitQCQP:Alg:dk}), (\ref{BundleSQPmitQCQP:Alg:Defkappabark+1}), (\ref{BundleSQPmitQCQP:Alg:gtildepkCOMPACT}) and (\ref{BundleSQPmitQCQP:Alg:ghattildepkCOMPACT}) as well as
\begin{multline}
\tfrac{1}{2}
\vert H_{k+i}
(
\tilde{g}_p^{k+i}+\bar{\kappa}^{k+i+1}\tilde{\hat{g}}_p^{k+i}
)
\vert^2
\leq\\
d_{k+i}^T\widebar{W}_p^{k+i}d_{k+i}
+\tfrac{1}{2}d_{k+i}^T\big(
\!\!\sum_{j\in J_{k+i}}\!\!
\lambda_j^{k+i}\widebar{G}_j^{k+i}+\lambda_p^{k+i}\widebar{G}^{k+i}
\mu_j^{k+i}\widebar{\hat{G}}_j^{k+i}+\mu_p^{k+i}\widebar{\hat{G}}^{k+i}
\big)d_{k+i}
\cbmathMIFFLIN
\label{BundleSQPmitQCQP:GlobaleKonvergenz:Satz:Lemma3.7:Bewxk+1:ProofNEW2}
\end{multline}
due to (\ref{proposition:ConnectionBetweenHkAndWpkNorm}) and the positive definiteness of $\widebar{W}_p^{k+i}$ as well as
\begin{equation}
\alpha_p^{k+i}+\bar{\kappa}^{k+i+1}A_p^{k+i}+\bar{\kappa}^{k+i+1}\big(-F(x_{k+i})\big)
\geq0
\label{BundleSQPmitQCQP:GlobaleKonvergenz:Satz:Lemma3.7:Bewxk+1:ProofNEW2:PAPER}
\end{equation}
due to (\ref{BundleSQPmitQCQP:Alg:alphatildepk}), (\ref{BundleSQPmitQCQP:Alg:Defkappabark+1}), (\ref{BundleSQPmitQCQP:Alg:Atildepk}) and (\ref{BundleSQPmitQCQP:ZulaessigkeitsbedingungFuerIterationspunkt}). Now, using (\ref{BundleSQPmitQCQP:Alg:xk+1Update}), (\ref{BundleSQPmitQCQP:GlobaleKonvergenz:Satz:Lemma3.7:Bewxk+1:ProofNEW1}), (\ref{BundleSQPmitQCQP:GlobaleKonvergenz:Satz:Lemma3.7:Bewxk+1:ProofNEW2}), adding (\ref{BundleSQPmitQCQP:GlobaleKonvergenz:Satz:Lemma3.7:Bewxk+1:ProofNEW2:PAPER}), using
(\ref{BundleSQPmitQCQP:Alg:vk}), the boundedness of $\lbrace H_k\rbrace$ (by assumption), $t_L^{k+i}\in[0,1]$ and (\ref{BundleSQPmitQCQP:Satz:GlobaleKonvergenz:Lemma3.5:tLkvk0}) yields
$
\lvert x_{k+(i+1)}-x_{k+i}\rvert\xrightarrow{K}0
$,
and therefore $\lvert x_{k+(i+1)}-\bar{x}\rvert\xrightarrow{K}0$ follows from the induction hypothesis.

We show (\ref{BundleSQPmitQCQP:GlobaleKonvergenz:Lemma3.7:AussagetLk+i}) by contradiction: Suppose (\ref{BundleSQPmitQCQP:GlobaleKonvergenz:Lemma3.7:AussagetLk+i}) is false, i.e.~there exists $i\geq0$, $\bar{t}>0$, $\bar{K}\subset K$: $t_L^{k+i}\geq\bar{t}$ for all $k\in\bar{K}$. Since $0\leq\bar{t}w_{k+i}\leq-t_L^{k+i}v_{k+i}\xrightarrow{\bar{K}}0$ due to
(\ref{BundleSQPmitQCQP:Satz:GlobaleKonvergenz:Lemma3.5:wkCOMPACT}),
(\ref{BundleSQPmitQCQP:Alg:wk}), (\ref{proposition:ConnectionBetweenHkAndWpkNorm}),
$t_L^{k+i}\in[0,1]$, (\ref{BundleSQPmitQCQP:Alg:vk}) and
(\ref{BundleSQPmitQCQP:Satz:GlobaleKonvergenz:Lemma3.5:tLkvk0}),
we have $w_{k+i}\xrightarrow{\bar{K}}0$ and therefore we obtain $\sigma(\bar{x})=0$ due to
(\ref{BundleSQPmitQCQP:GlobaleKonvergenz:Def:sigma}) and
(\ref{BundleSQPmitQCQP:GlobaleKonvergenz:Lemma3.7:Aussagexk+i}),
which is a contradiction to the assumption $\sigma(\bar{x})>0$.

We show (\ref{BundleSQPmitQCQP:GlobaleKonvergenz:Lemma3.7:AussagetLk+it0COMPACT}): Let $r\geq0$ be fixed and $0\leq i\leq r$. Since we have $\tfrac{\sigma(\bar{x})}{2}<\lim_{K}w_{k+i}$ due to the assumption $\sigma(\bar{x})>0$, (\ref{BundleSQPmitQCQP:GlobaleKonvergenz:Def:sigma}) and (\ref{BundleSQPmitQCQP:GlobaleKonvergenz:Lemma3.7:Aussagexk+i}), because of (\ref{BundleSQPmitQCQP:GlobaleKonvergenz:Lemma3.7:AussagetLk+i}) and because $\varepsilon_0>0$ is a fixed number by assumption, there exist $k_i\geq0$ with $\tfrac{\sigma(\bar{x})}{2}\leq w_{k+i}$ and $t_L^{k+i}<\varepsilon_0$ for all $k>k_i$ with $k\in K$. Now, setting $\tilde{k}:=\max{\lbrace k_i:0\leq i\leq r\rbrace}$ yields (\ref{BundleSQPmitQCQP:GlobaleKonvergenz:Lemma3.7:AussagetLk+it0COMPACT}).
\qedhere
\end{proof}
\begin{proposition}
\label{Luksan:GlobaleKonvergenz:Lemma3.4}
Let $p,g,\Delta\in\mathbb{R}^N$ and $c,u,w,\beta\in\mathbb{R}$, $m\in(0,1)$, $\alpha\geq0$ with
\begin{equation*}
          w=   \tfrac{1}{2}\lvert p\rvert^2+\alpha             \komma\quad
          v=   -(\lvert p\rvert^2+\alpha)                      \komma\quad
-\beta-g^Tp\geq mv                                             \komma\quad
          c=\max{(\lvert g\rvert,\lvert p\rvert,\sqrt{\alpha})}
\end{equation*}
and define $Q:\mathbb{R}\longrightarrow\mathbb{R}$ by
\begin{equation*}
Q(\nu):=\tfrac{1}{2}\lvert\nu g+(1-\nu)(p+\Delta)\rvert^2+\nu\beta+(1-\nu)\alpha
\komma
\end{equation*}
then
\begin{equation*}
\min_{\nu\in[0,1]}{Q(\nu)}
\leq w-w^2\tfrac{(1-m)^2}{8c^2}+4c\lvert\Delta\rvert+\tfrac{1}{2}\lvert\Delta\rvert^2
\punkt
\end{equation*}
\end{proposition}
\begin{proof}
\citet[Lemma~3.4]{Luksan}.
\qedhere
\end{proof}
\cbstartDVI
We introduce the following notation (cf.~\citet[p.~31,~Section~2~resp.~p.~34,~Section~4]{MagnusNeudecker}).
\begin{definition}
Let $A,B\in\mathbb{R}^{N\times N}$.
We define the Frobenius norm of $A$ by
$
\lvert A\rvert_{_F}
:=
\big(\!\sum_{i,j=1}^NA_{ij}^2\big)^{\frac{1}{2}}
$
and we define the vectorization $A_{(:)}$ of $A$
as well as the Kronecker product $A\otimes B$ of $A$ and $B$ by
\begin{equation}
A_{(:)}
:=
\left(
\begin{smallmatrix}
A_{:1}\\
\vdots\\
A_{:N}
\end{smallmatrix}
\right)
\in\mathbb{R}^{N^2}
\komma\quad
A\otimes B
:=
\left(
\begin{smallmatrix}
A_{11}B & \dots & A_{1N}B\\
\vdots  &       & \vdots \\
A_{N1}B & \dots & A_{NN}B
\end{smallmatrix}
\right)
\in\mathbb{R}^{N^2\times N^2}
\punkt
\label{Magnus:Def:KroneckerProdukt}
\end{equation}
\end{definition}
\begin{proposition}
Let
$A,B,C\in\mathbb{R}^{N\times N}$. Then
\begin{equation}
\lvert A\rvert
\leq
\lvert A\rvert_{_F}
\leq
\sqrt{N}\lvert A\rvert
\komma\quad
(ABC)_{(:)}
=
(C^T\otimes A)B_{(:)}
\komma\quad
\lvert A\otimes A\rvert
\leq
N\lvert A\rvert^2
\punkt
\label{KonvergenzErweiterung:Satz:SpektralFrobeniusAbschaetzungCOMPACT}
\end{equation}
\end{proposition}
\begin{proof}
The first property of (\ref{KonvergenzErweiterung:Satz:SpektralFrobeniusAbschaetzungCOMPACT}) holds due to \citet[p.~56,~Section~2.3.2]{GolubVanLoanMatrixCompuations},
the second holds due to \citet[p.~35,~Theorem~2]{MagnusNeudecker},
and the third holds due to (\ref{Magnus:Def:KroneckerProdukt}).
\qedhere
\end{proof}
Now, we introduce differentiability of matrix valued functions (cf.~\citet[p.~107,~Definition~3]{MagnusNeudecker}).
\begin{definition}
Let $A:\mathbb{R}^p\longrightarrow\mathbb{R}^{N\times N}$ and $\mu_0\in\mathbb{R}^p$ be fixed. If there exists $B(\mu_0)\in\mathbb{R}^{N^2\times p}$ with
\begin{equation}
A_{(:)}(\mu_0+\mu)=A_{(:)}(\mu_0)+B(\mu_0)\mu+R_{(:)}(\mu_0,\mu)
\label{Magnus:Def:MatrixFunktionDifferenzierbarkeit}
\end{equation}
for all $\mu\in\mathbb{R}^p$ in a neighborhood of $\mu_0$ and
$
\lim_{\mu\rightarrow0}\tfrac{R_{(:)}(\mu_0,\mu)}{\lvert\mu\rvert}=0
$,
then $A$ is said to be differentiable at $\mu_0$. Furthermore, the $N\times N$-matrix $\diff A(\mu_0,\mu)$ defined by
\begin{equation}
\diff A_{(:)}(\mu_0,\mu):=B(\mu_0)\mu
\in\mathbb{R}^{N^2}
\label{Magnus:Def:MatrixFunktionDifferential}
\end{equation}
is called the (first) differential of $A$ at $\mu_0$ with increment $\mu$ and $B(\mu_0)$ is called the first derivative of $A$ at $\mu_0$.
\end{definition}
\begin{proposition}
\label{Magnus:Satz:InverseMatrixAbleitung}
Let
$
T:=\lbrace Y:~Y\in\mathbb{R}^{N\times N},~\det Y\not=0\rbrace
$
be the set of non-singular $N\times N$-matrices. If $A:\mathbb{R}^p\longrightarrow T$ is $k$ times (continuously) differentiable, then so is $B:\mathbb{R}^p\longrightarrow T$ defined by
$
B(\mu):=A(\mu)^{-1}
$
and
\begin{equation}
\diff B(\mu_0,\mu)
=
-B(\mu_0)\diff A(\mu_0,\mu)B(\mu_0)
\label{Magnus:Satz:InverseMatrixAbleitungAussage}
\punkt
\end{equation}
\end{proposition}
\begin{proof}
\citet[p.~156,~Theorem~3]{MagnusNeudecker}.
\qedhere
\end{proof}
\begin{proposition}
\label{AN:proposition:MeanValueTheoremForVectorValuedFunctions}
Let $f:\Omega\subseteq\mathbb{R}\longrightarrow\mathbb{R}^q$ (with an open interval $\Omega$) be continuously differentiable
and let
$
\omega
:=
\sup_{z\in\Omega}{\lvert f'(z)\rvert}
<
\infty
$,
then
\begin{equation}
\lvert f(y)-f(x)\rvert
\leq
\omega\lvert y-x\rvert
\label{Heuser:Satz:MeanValueVectorValuedFolgerung}
\end{equation}
for all $x,y\in\Omega$ (i.e.~$f$ is Lipschitz continuous on $\Omega$).
\end{proposition}
\begin{proof}
This is a direct consequence of the mean value theorem for vector valued functions (cf., e.g., \citet[p.~278,~167.4~Mittelwertsatz~f\"ur~vektorwertige~Funktionen]{Heuser2}).
\qedhere
\end{proof}
\begin{proposition}
\label{KonvergenzErweiterung:Satz:MatrixAbschaetzungMitHigham:ZkAbschaetzungSatz}
Let $\lbrace\bar{\kappa}^{k+1}\rbrace$ be bounded and let $\lbrace(\widebar{W}_p^k)^{-\mathalf}\rbrace$ be bounded and uniformly positive definite. For $k\geq1$ we define $Z_k:\Rpos\longrightarrow\mathbb{R}^{N\times N}$
\begin{equation}
Z_k(s)
:=
\big(\widebar{W}_p^k
+\widebar{G}^k
+s\widebar{\hat{G}}^k\big)
^{-\mathalf}
\punkt
\cbmathMIFFLIN
\label{KonvergenzErweiterung:Satz:MatrixAbschaetzungMitHigham:DefZk}
\end{equation}
Then we have for all $k\geq1$
\begin{equation}
\lvert Z_k(\bar{\kappa}^{k+2})-Z_k(\bar{\kappa}^{k+1})\rvert
\leq
C_5\lvert\bar{\kappa}^{k+2}-\bar{\kappa}^{k+1}\rvert
\komma\quad
0
\leq
C_5<\infty
\komma
\label{KonvergenzErweiterung:Satz:MatrixAbschaetzungMitHigham:ZkAbschaetzungCOMPACT}
\end{equation}
where
$C_5:=C_2C_4$,
$C_4:=N\WpInverseBound^2C_3$,
$C_3:=N^{\mathalf}\bar{\hat{C}}_G$
and $C_2$ is a positive constant.
\end{proposition}
\begin{proof}
We define for all $k\geq1$
\begin{equation}
Y_k(s)
:=
\big(\widebar{W}_p^k
+\widebar{G}^k
+s\widebar{\hat{G}}^k\big)
^{-1}
\cbmathMIFFLIN
\label{KonvergenzErweiterung:Satz:MatrixAbschaetzungMitHigham:Bew:DefYk}
\end{equation}
and therefore we have $\lvert Y_k(\bar{\kappa}^{k+1})^{-1}\rvert\leq C_1$ for all $k\geq1$ due to (\ref{KonvergenzErweiterung:Satz:MatrixAbschaetzungMitHigham:Bew:DefYk}) and (\ref{KonvergenzErweiterung:Satz:MatrixAbschaetzungMitHigham:MatritzenBeschraenktheit}), which is equivalent to $\lbrace Y_k(\bar{\kappa}^{k+1})\rbrace$ being uniformly positive definite due to (\ref{ADDON:Satz:AkGlmpdGdwAkInvereseBeschraenkt}), i.e.~there exists $\tilde{C}_2>0$ with $\lambdamin(Y_{k}(\bar{\kappa}^{k+1}))\geq\tilde{C}_2$. Consequently, we obtain for all $k\geq1$
\begin{equation*}
\tfrac{1}{
(\lambdamin(Y_{k}(\bar{\kappa}^{k+2})))^{\mathalf}
+
(\lambdamin(Y_{k}(\bar{\kappa}^{k+1})))^{\mathalf}
}
\leq
\tfrac{1}{2}\tilde{C}_2^{-\mathalf}
\end{equation*}
and hence we estimate for all $k\geq1$
\begin{equation}
\lvert Y_k(\bar{\kappa}^{k+2})^{\mathalf}-Y_k(\bar{\kappa}^{k+1})^{\mathalf}\rvert
\leq
C_2\lvert Y_k(\bar{\kappa}^{k+2})-Y_k(\bar{\kappa}^{k+1})\rvert
\komma
\label{KonvergenzErweiterung:Satz:MatrixAbschaetzungMitHigham:Bew:YkAussage}
\end{equation}
due to (\ref{KonvergenzErweiterung:Satz:Higham}),
where we set $C_2:=\tfrac{1}{2}\tilde{C}_2^{-\mathalf}>0$.

Defining
\begin{equation}
X_k(s):=\widebar{W}_p^k+\widebar{G}^k+U_k(s)
\komma\quad
U_k(s):=s\widebar{\hat{G}}^k
\komma\quad
\hat{U}_k:=\widebarMatrix{\hat{G}}_{(:)}^k
\label{KonvergenzErweiterung:Satz:MatrixAbschaetzungMitHigham:Bew:DefXkCOMPACT}
\end{equation}
for $k\geq1$, we calculate $U_{k,(:)}(t-s)=\hat{U}_k(t-s)$ due to (\ref{KonvergenzErweiterung:Satz:MatrixAbschaetzungMitHigham:Bew:DefXkCOMPACT}). Therefore, we have $X_k(t)=X_k(s)+U_k(t-s)$ for all $k\geq1$ and for all $s,t\in\mathbb{R}$ due to (\ref{KonvergenzErweiterung:Satz:MatrixAbschaetzungMitHigham:Bew:DefXkCOMPACT}), which is equivalent to $X_{k,(:)}(t)=X_{k,(:)}(s)+\hat{U}_k(t-s)$. Consequently, (\ref{Magnus:Def:MatrixFunktionDifferenzierbarkeit}) and (\ref{Magnus:Def:MatrixFunktionDifferential}) imply that the differential of $X_k$ at $s$ is given by
\begin{equation}
\diff X_{k,(:)}(s,t-s)=\hat{U}_k(t-s)
\label{KonvergenzErweiterung:Satz:MatrixAbschaetzungMitHigham:Bew:XkDifferential}
\end{equation}
(with $R_{k}(s,t-s)\equiv\zeroMatrix{N}$) and that the derivative of $X_k$ at $s$ is constant, which implies that $X_k$ is continuously differentiable. Furthermore, we estimate for all $k\geq1$
\begin{equation}
\lvert\hat{U}_k\rvert
\leq
C_3
\komma
\label{KonvergenzErweiterung:Satz:MatrixAbschaetzungMitHigham:Bew:C3Abschaetzung}
\end{equation}
due to
(\ref{KonvergenzErweiterung:Satz:MatrixAbschaetzungMitHigham:Bew:DefXkCOMPACT}),
(\ref{KonvergenzErweiterung:Satz:SpektralFrobeniusAbschaetzungCOMPACT}) and
the initialization of Algorithm \ref{BundleSQPmitQCQP:Alg:GesamtAlgMitQCQP}.

Since $\widebar{\hat{G}}^k$ is symmetric and positive definite, we obtain that $U_k(s)$ is symmetric and positive semidefinite for all $s\geq0$ (cf.~(\ref{KonvergenzErweiterung:Satz:MatrixAbschaetzungMitHigham:Bew:DefXkCOMPACT})). Consequently, we have $\widebar{W}_p^k\preceq X_k(s)$ due to
\cbstartMIFFLIN
the symmetry and the positive definiteness of $\widebar{G}^k$,
\cbendMIFFLIN
(\ref{HornJohnson:Erweiterung:PositivDefinitSummeAbschaetzung}) and (\ref{KonvergenzErweiterung:Satz:MatrixAbschaetzungMitHigham:Bew:DefXkCOMPACT}). Therefore, we estimate for all $k\geq1$ and for all $s\geq0$
\begin{equation}
\lvert Y_k(s)\rvert
\leq
\WpInverseBound
\punkt
\label{KonvergenzErweiterung:Satz:MatrixAbschaetzungMitHigham:Bew:SupYkAbschaetzung}
\end{equation}
due to (\ref{KonvergenzErweiterung:Satz:MatrixAbschaetzungMitHigham:Bew:DefYk}),
(\ref{KonvergenzErweiterung:Satz:MatrixAbschaetzungMitHigham:Bew:DefXkCOMPACT}),
(\ref{HornJohnson:Erweiterung:InverseSpektralNormAbschaetzung}) and
(\ref{KonvergenzErweiterung:Satz:InverseVonWkBeschraenkt}).

For $k\geq1$ $Y_k$ we define
\begin{equation}
V_k(s)
:=
\big(Y_k(s)\otimes Y_k(s)\big)\hat{U}_k
\punkt
\label{KonvergenzErweiterung:Satz:MatrixAbschaetzungMitHigham:Bew:DefVk}
\end{equation}
Since $X_k$ is continuously differentiable (cf.~(\ref{KonvergenzErweiterung:Satz:MatrixAbschaetzungMitHigham:Bew:XkDifferential})), Proposition \ref{Magnus:Satz:InverseMatrixAbleitung} yields the continuous differentiability of $Y_k(s)=X_k(s)^{-1}$ due to (\ref{KonvergenzErweiterung:Satz:MatrixAbschaetzungMitHigham:Bew:DefYk}) and
(\ref{KonvergenzErweiterung:Satz:MatrixAbschaetzungMitHigham:Bew:DefXkCOMPACT}),
as well as $\diff Y_{k,(:)}(s,t-s)=V_k(s)(t-s)$ due to
(\ref{Magnus:Satz:InverseMatrixAbleitungAussage}),
(\ref{KonvergenzErweiterung:Satz:SpektralFrobeniusAbschaetzungCOMPACT}),
$Y_k(s)\in\Sym{N}$,
(\ref{KonvergenzErweiterung:Satz:MatrixAbschaetzungMitHigham:Bew:XkDifferential}) and
(\ref{KonvergenzErweiterung:Satz:MatrixAbschaetzungMitHigham:Bew:DefVk})
and therefore (\ref{Magnus:Def:MatrixFunktionDifferential}) implies that $V_k(s)$ is the derivative of $Y_k$ at $s$. Furthermore, we estimate for all $k\geq1$
\begin{equation}
\sup_{s\geq0}{\lvert V_k(s)\rvert}
\leq
C_4
\punkt
\label{KonvergenzErweiterung:Satz:MatrixAbschaetzungMitHigham:Bew:supNormVonVkAbschaetzung}
\end{equation}
due to (\ref{KonvergenzErweiterung:Satz:MatrixAbschaetzungMitHigham:Bew:DefVk}), (\ref{KonvergenzErweiterung:Satz:SpektralFrobeniusAbschaetzungCOMPACT}), (\ref{KonvergenzErweiterung:Satz:MatrixAbschaetzungMitHigham:Bew:C3Abschaetzung}) and (\ref{KonvergenzErweiterung:Satz:MatrixAbschaetzungMitHigham:Bew:SupYkAbschaetzung}).

Since $Y_k$ is continuously differentiable for all $s,t\in S:=\lbrace\xi\in\mathbb{R}:\xi\geq0\rbrace$
(note that $S$ is an interval) and since the derivative of $Y_k$ at $s$ is given by $V_k(s)$ (cf.~(\ref{KonvergenzErweiterung:Satz:MatrixAbschaetzungMitHigham:Bew:DefVk})) and since the norm of the derivative $\lvert V_k(s)\rvert$ is bounded on $S$ due to (\ref{KonvergenzErweiterung:Satz:MatrixAbschaetzungMitHigham:Bew:supNormVonVkAbschaetzung}), we obtain
\begin{equation}
\lvert Y_{k,(:)}(t)-Y_{k,(:)}(s)\rvert
\leq
C_4\lvert t-s\rvert
\label{KonvergenzErweiterung:Satz:MatrixAbschaetzungMitHigham:Bew:YkLipschitzAbschaetzung}
\end{equation}
for all $s,t\in S$ and for all $k\geq1$ due to (\ref{Heuser:Satz:MeanValueVectorValuedFolgerung}).

Now, we estimate for all $k\geq1$
\begin{equation*}
\lvert Z_k(\bar{\kappa}^{k+2})-Z_k(\bar{\kappa}^{k+1})\rvert
\leq
C_5\lvert\bar{\kappa}^{k+2}-\bar{\kappa}^{k+1}\rvert
\end{equation*}
due to
(\ref{KonvergenzErweiterung:Satz:MatrixAbschaetzungMitHigham:DefZk}),
(\ref{KonvergenzErweiterung:Satz:MatrixAbschaetzungMitHigham:Bew:DefYk}),
(\ref{KonvergenzErweiterung:Satz:MatrixAbschaetzungMitHigham:Bew:YkAussage}),
(\ref{KonvergenzErweiterung:Satz:SpektralFrobeniusAbschaetzungCOMPACT}) and
(\ref{KonvergenzErweiterung:Satz:MatrixAbschaetzungMitHigham:Bew:YkLipschitzAbschaetzung}).

Furthermore, we obtain $C_5=C_2N^{\frac{3}{2}}\WpInverseBound^2\bar{\hat{C}}_G$ and therefore
the fact that $C_2$ is a positive constant due to (\ref{KonvergenzErweiterung:Satz:MatrixAbschaetzungMitHigham:Bew:YkAussage}), 
the fact that $N\geq1$ is a fixed finite natural number, combining (\ref{KonvergenzErweiterung:Satz:InverseVonWkBeschraenkt}) with the positive definiteness of $\widebar{W}_p^k$, and the initialization of Algorithm \ref{BundleSQPmitQCQP:Alg:GesamtAlgMitQCQP} yield (\ref{KonvergenzErweiterung:Satz:MatrixAbschaetzungMitHigham:ZkAbschaetzungCOMPACT}).
\qedhere
\end{proof}
\cbendDVI
From now on let the following assumption be satisfied.
\begin{presumption}
\label{COMPACT:presumption:AssumptionForMainTheorem}
Let \refh{BundleSQPmitQCQP:GlobaleKonvergenz:Vor:AlgTerminiertNicht} be satisfied. Furthermore, let the sequence $\lbrace(x_k,\bar{\kappa}^{k+1})\rbrace$ be bounded, let the sequence (of symmetric, positive definite matrices)
\cbstartDVI
$\lbrace(\widebar{W}_p^k)^{-\mathalf}\rbrace$ be bounded as well as uniformly positive definite
\cbendDVI
and let $\bar{x}\in\mathbb{R}^n$ be any accumulation point of $\lbrace x_k\rbrace$,
\cbstartDVI
i.e.~there exists (an infinite set) $K\subset\lbrace1,2,\dots\rbrace$ with
\begin{equation}
x_k\xrightarrow{K}\bar{x}
\komma
\label{BundleSQPmitQCQP:GlobaleKonvergenz:Satz:Thm3.8:Bew:xkGegenxbar}
\end{equation}
and demand
\begin{equation}
\bar{\kappa}^{k+2}-\bar{\kappa}^{k+1}\xrightarrow{K}0
\label{BundleSQPmitQCQP:GlobaleKonvergenz:Satz:Thm3.8:Bew:ADDON:DefDiffbarkappa}
\end{equation}
\cbendDVI
as well as
\begin{equation}
t_0^{\inf}:=\inf_{
k\geq0
}{t_0^k}>0
\label{Theorem3.8:Presumptiont0infGT0}
\end{equation}
(cf.~Remark \refH{remark:t0infWithGraphic}).
\end{presumption}
Next, we present Lemma \ref{COMPACT:lemma:BoundedBasicSequences}--\ref{ASUS:lemma:Contradiction}, which we need for proving Theorem \ref{BundleSQPmitQCQP:GlobaleKonvergenz:Satz:Thm3.8:Bew:Theorem}.
\begin{lemma}[Bounded basic sequences]
\label{COMPACT:lemma:BoundedBasicSequences}
The following boundedness statements hold:
\begin{align}
&
\lbrace y_k\rbrace\textnormal{, }
\lbrace\rho_kG_k\rbrace\textnormal{, }
\lbrace\hat{\rho}_k\hat{G}_k\rbrace\textnormal{ and }
\lbrace g_k\rbrace\textnormal{ are bounded,}
\label{BundleSQPmitQCQP:GlobaleKonvergenz:Satz:Thm3.8:Bew:ykUNDrhokGkUNDgkBeschraenkt}
\\
&
\lbrace H_k\rbrace\textnormal{ is bounded,}
\label{BundleSQPmitQCQP:GlobaleKonvergenz:Satz:Thm3.8:Bew:HkIsBounded}
\\
&
\lbrace g_k^k\rbrace\textnormal{, }\lbrace H_kg_k^k\rbrace\textnormal{ and }\lbrace \alpha_k^k\rbrace\textnormal{ are bounded.}
\label{BundleSQPmitQCQP:GlobaleKonvergenz:Satz:Thm3.8:Bew:gkkUNDHkgkkUNDalphakkBeschraenkt}
\end{align}
\end{lemma}
\begin{proof}
(\ref{BundleSQPmitQCQP:GlobaleKonvergenz:Satz:Thm3.8:Bew:ykUNDrhokGkUNDgkBeschraenkt}) holds
as this statement was shown in the proof of
Lemma \ref{COMPACT:lemmaForAuxiliaryTheorem:TrialPointConvergenceAndImplications},
where only the assumption of the boundedness of $\lbrace x_k\rbrace$ was used, and consequently, this statement is here also true.
Since $\lbrace(\widebar{W}_p^k)^{-\mathalf}\rbrace$ is bounded due to assumption,
(\ref{BundleSQPmitQCQP:GlobaleKonvergenz:Satz:Thm3.8:Bew:HkIsBounded}) holds
due to Corollary \ref{ADDON:Korollar:WpkUndHkBeschraenktheitUndPositivDefinitheit}.
Due to (\ref{BundleSQPmitQCQP:Alg:gk+1k+1Update}), the boundedness of $\lbrace x_k\rbrace$ and (\ref{BundleSQPmitQCQP:GlobaleKonvergenz:Satz:Thm3.8:Bew:ykUNDrhokGkUNDgkBeschraenkt}) resp.~$\lvert H_kg_k^k\rvert\leq\lvert H_k\rvert\cdot\lvert g_k^k\rvert$
and
(\ref{BundleSQPmitQCQP:GlobaleKonvergenz:Satz:Thm3.8:Bew:HkIsBounded})
resp.~(\ref{BundleSQPmitQCQP:Alg:alphajkCOMPACT}), (\ref{BundleSQPmitQCQP:Alg:fk+1k+1Update}), (\ref{BundleSQPmitQCQP:Alg:sk+1k+1Update}), (\ref{BundleSQPmitQCQP:Alg:fk+1Update}), the Cauchy-Schwarz inequality
and the fact that $f$ is continuous on (the whole) $\mathbb{R}^n$,
we obtain (\ref{BundleSQPmitQCQP:GlobaleKonvergenz:Satz:Thm3.8:Bew:gkkUNDHkgkkUNDalphakkBeschraenkt}).
\qedhere
\end{proof}
\begin{lemma}[Bounded aggregate sequences]
We define
\begin{equation}
\tau_k
:=
\tilde{\alpha}_p^k+\bar{\kappa}^{k+1}\tilde{A}_p^k+\bar{\kappa}^{k+1}\big(-F(x_k)\big)
\geq
0
\komma
\label{BundleSQPmitQCQP:GlobaleKonvergenz:Satz:Thm3.8:Bew:ADDON:Deftauk}
\end{equation}
then
\begin{equation}
\lbrace w_k\rbrace\textnormal{, }
\lbrace\tilde{g}_p^k\rbrace\textnormal{, }
\lbrace\tilde{\hat{g}}_p^k\rbrace\textnormal{, }
\lbrace\tilde{\alpha}_p^k\rbrace\textnormal{, }
\lbrace\bar{\kappa}^{k+1}\tilde{A}_p^k\rbrace\textnormal{, }
\lbrace H_k(\tilde{g}_p^k+\bar{\kappa}^{k+1}\tilde{\hat{g}}_p^k)\rbrace\textnormal{ and }
\lbrace\tau_k\rbrace
\textnormal{ are bounded}
\punkt
\label{BundleSQPmitQCQP:GlobaleKonvergenz:Satz:Thm3.8:Bew:wkUNDHkgtildepkUNDgtildepkUNDalphatildepkBeschraenkt}
\end{equation}
\end{lemma}
\begin{proof}
Since $(\lambda,\lambda_p,\mu,\mu_p)\in\mathbb{R}^{2(\lvert J_k\rvert+1)}$ with
\begin{equation}
\lambda_j
:=
\left\lbrace
\begin{array}{ll}
1 & \textnormal{for }j=k\\
0 & \textnormal{for }j\in J_k\setminus\lbrace k\rbrace
\end{array}
\right\rbrace
~,~
\lambda_p
:=0
\komma\quad
\mu_j
:=0~~~\forall j\in J_k
~,~
\mu_p
:=0
\label{BundleSQPmitQCQP:GlobaleKonvergenz:Satz:Thm3.8:Bew:lambdaZulaessig}
\end{equation}
is feasible for the (dual) problem (\ref{Luksan:Alg:QCQPTeilproblem:DualesProblem}) for $k\geq1$ (Note: This problem is written as a minimization problem), we obtain (Note: $\hat{w}_k$ is the optimal function value of (\ref{Luksan:Alg:QCQPTeilproblem:DualesProblem})) due to (\ref{BundleSQPmitQCQP:GlobaleKonvergenz:Lemma3.5:Defwhat}),
(\ref{BundleSQPmitQCQP:Alg:gtildepkCOMPACT}),
(\ref{BundleSQPmitQCQP:GlobaleKonvergenz:Lemma3.5:DefalphahatCOMPACT}),
(\ref{BundleSQPmitQCQP:Alg:ghattildepkCOMPACT}),
(\ref{BundleSQPmitQCQP:Alg:Defkappabark+1}) and
inserting the feasible point from (\ref{BundleSQPmitQCQP:GlobaleKonvergenz:Satz:Thm3.8:Bew:lambdaZulaessig})
that
$
\hat{w}_k\leq\tfrac{1}{2}\lvert H_kg_k^k\rvert^2+\alpha_k^k
$.
Hence,
due to (\ref{BundleSQPmitQCQP:Alg:wk}) and (\ref{BundleSQPmitQCQP:Satz:GlobaleKonvergenz:Lemma3.5:wkCOMPACT}), we estimate
\begin{equation*}
0
\leq
\tfrac{1}{2}\vert H_k(\tilde{g}_p^k+\bar{\kappa}^{k+1}\tilde{\hat{g}}_p^k)\vert^2
+\tilde{\alpha}_p^k
+\bar{\kappa}^{k+1}\tilde{A}_p^k
+\bar{\kappa}^{k+1}\big(-F(x_k)\big)
\leq
\tfrac{1}{2}\lvert H_kg_k^k\rvert^2+\alpha_k^k
\end{equation*}
and therefore (\ref{BundleSQPmitQCQP:GlobaleKonvergenz:Satz:Thm3.8:Bew:gkkUNDHkgkkUNDalphakkBeschraenkt}) as well as the non-negativity of $\tilde{\alpha}_p^k$, $\bar{\kappa}^{k+1}$, $\tilde{A}_p^k$ resp.~$-F(x_k)$ due (\ref{BundleSQPmitQCQP:Alg:alphatildepk}), (\ref{BundleSQPmitQCQP:Alg:Defkappabark+1}), (\ref{BundleSQPmitQCQP:Alg:Atildepk}) resp.~(\ref{BundleSQPmitQCQP:ZulaessigkeitsbedingungFuerIterationspunkt}) imply that  $\lbrace w_k\rbrace$, $\lbrace\tilde{\alpha}_p^k\rbrace$, $\lbrace\bar{\kappa}^{k+1}\tilde{A}_p^k\rbrace$, $\lbrace H_k(\tilde{g}_p^k+\bar{\kappa}^{k+1}\tilde{\hat{g}}_p^k)\rbrace$ and $\lbrace\tau_k\rbrace$
are bounded. Now, consider the proof of
Lemma \ref{COMPACT:lemmaForAuxiliaryTheorem:ComplementarityCondition}:
There we only used the first consequence $x_k\xrightarrow{K}\bar{x}$ of (\ref{BundleSQPmitQCQP:GlobaleKonvergenz:Satz:Lemma3.6:BewxkCOMPACT}) (and this property is also satisfied here due to (\ref{BundleSQPmitQCQP:GlobaleKonvergenz:Satz:Thm3.8:Bew:xkGegenxbar})) of the assumption $\sigma(\bar{x})=0$ for showing the convergence of $\tilde{g}_p^k$ resp.~$\tilde{\hat{g}}_p^k$ on a subsequence. Consequently, $\tilde{g}_p^k$ resp.~$\tilde{\hat{g}}_p^k$ are also bounded here. [The second property ($w_k\xrightarrow{K}0$) of (\ref{BundleSQPmitQCQP:GlobaleKonvergenz:Satz:Lemma3.6:BewxkCOMPACT}) resulting from $\sigma(\bar{x})=0$ there, is first used directly after
proving the boundedness of $\tilde{g}_p^k$ and $\tilde{\hat{g}}_p^k$.
If this property was already used
for showing these boundedness results,
the above implication would be false, since then indeed $\sigma(\bar{x})=0$ (and not only $x_k\xrightarrow{K}\bar{x}$) would be used
for proving the boundedness of $\tilde{g}_p^k$ and $\tilde{\hat{g}}_p^k$,
and the relevant situation
in the proof of Theorem \ref{BundleSQPmitQCQP:GlobaleKonvergenz:Satz:Thm3.8:Bew:Theorem} will be
$\sigma(\bar{x})>0$.]
\qedhere
\end{proof}
\begin{lemma}[$\sigma$ is finite]
\label{BundleSQPmitQCQP:GlobaleKonvergenz:Satz:Thm3.8:Bew:sigmaendlichCOMPACT}
$\sigma(\bar{x})$ is finite.
\end{lemma}
\begin{proof}
This is true due to
(\ref{BundleSQPmitQCQP:GlobaleKonvergenz:Def:sigma}), the assumption of the boundedness of $\lbrace x_k\rbrace$ and (\ref{BundleSQPmitQCQP:GlobaleKonvergenz:Satz:Thm3.8:Bew:wkUNDHkgtildepkUNDgtildepkUNDalphatildepkBeschraenkt}).
\qedhere
\end{proof}
\begin{lemma}[Cauchy sequences]
\label{COMPACT:lemma:CauchySequences}
We have
\begin{align}
s_p^{k+1}-\tilde{s}_p^k
\xrightarrow{K}
0
\komma\quad&
\hat{s}_p^{k+1}-\tilde{\hat{s}}_p^k
\xrightarrow{K}
0
\label{BundleSQPmitQCQP:GlobaleKonvergenz:Satz:Thm3.8:Bew:spk+1-stildepk0COMPACT}
\komma
\\
f(x_{k+1})-f(x_k)
\xrightarrow{K}0
\komma\quad&
F(x_{k+1})-F(x_k)
\xrightarrow{K}0
\komma
\label{BundleSQPmitQCQP:GlobaleKonvergenz:Satz:Thm3.8:Bew:fxk+1-fxk0COMPACT}
\\
f_p^{k+1}-\tilde{f}_p^k\xrightarrow{K}0
\komma\quad&
F_p^{k+1}-\tilde{F}_p^k\xrightarrow{K}0
\komma\quad
\Delta_k\xrightarrow{K}0
\komma
\label{BundleSQPmitQCQP:GlobaleKonvergenz:Satz:Thm3.8:Bew:fpk+1-ftildepk0COMPACT}
\end{align}
where
\begin{equation}
\Delta_k
:=
H_{k+1}\big(
(g_p^{k+1}+\bar{\kappa}^{k+1}\hat{g}_p^{k+1})
-
(\tilde{g}_p^k+\bar{\kappa}^{k+1}\tilde{\hat{g}}_p^k)
\big)
\punkt
\label{BundleSQPmitQCQP:GlobaleKonvergenz:Satz:Thm3.8:Bew:DefDeltak}
\end{equation}
\end{lemma}
\begin{proof}
Since the assumptions of Proposition \ref{BundleSQPmitQCQP:GlobaleKonvergenz:Lemma3.7} for applying (\ref{BundleSQPmitQCQP:GlobaleKonvergenz:Lemma3.7:Aussagexk+i}) are satisfied --- $x_k\xrightarrow{K}\bar{x}$ holds due to (\ref{BundleSQPmitQCQP:GlobaleKonvergenz:Satz:Thm3.8:Bew:xkGegenxbar}), $\sigma(\bar{x})>0$ holds due to Lemma \ref{BundleSQPmitQCQP:GlobaleKonvergenz:Satz:Thm3.8:Bew:sigmaendlichCOMPACT} --- applying (\ref{BundleSQPmitQCQP:GlobaleKonvergenz:Lemma3.7:Aussagexk+i}) for $i=1$ and $i=0$ yields
$
x_{k+1}-x_k\xrightarrow{K}0
$.
Due to
(\ref{BundleSQPmitQCQP:Alg:spk+1Update}) and
(\ref{BundleSQPmitQCQP:Alg:shatpk+1Update}),
we obtain
(\ref{BundleSQPmitQCQP:GlobaleKonvergenz:Satz:Thm3.8:Bew:spk+1-stildepk0COMPACT}).
Because of
(\ref{BundleSQPmitQCQP:GlobaleKonvergenz:Satz:Thm3.8:Bew:xkGegenxbar}) and
the continuity of $f$ and $F$,
we obtain
(\ref{BundleSQPmitQCQP:GlobaleKonvergenz:Satz:Thm3.8:Bew:fxk+1-fxk0COMPACT}).
Due to (\ref{BundleSQPmitQCQP:GlobaleKonvergenz:Vor:AlgTerminiertNicht}) the assumptions of Proposition \ref{BundleSQPmitQCQP:GlobaleKonvergenzSatz:Lemma3.1} are satisfied and therefore we estimate using
(\ref{BundleSQPmitQCQP:GlobaleKonvergenzSatz:Lemma3.1:22:1COMPACT}),
(\ref{BundleSQPmitQCQP:GlobaleKonvergenz:Satz:Thm3.8:Bew:ykUNDrhokGkUNDgkBeschraenkt}) and
(\ref{BundleSQPmitQCQP:Alg:rhok+1})
resp.~(\ref{BundleSQPmitQCQP:GlobaleKonvergenzSatz:Lemma3.1:22:1NBCOMPACT}),
(\ref{BundleSQPmitQCQP:GlobaleKonvergenzSatz:Lemma3.1:22:1NB:AllZero}),
(\ref{BundleSQPmitQCQP:GlobaleKonvergenz:Satz:Thm3.8:Bew:ykUNDrhokGkUNDgkBeschraenkt}) and
(\ref{BundleSQPmitQCQP:Alg:rhohatk+1})
that
$
\lvert     G _p^{k+1}\rvert\leq      C _G
\komma\quad
\lvert\hat{G}_p^{k+1}\rvert\leq \hat{C}_G
$.
Due to
(\ref{BundleSQPmitQCQP:Alg:fpk+1Update}),
the Cauchy-Schwarz inequality and
(\ref{BundleSQPmitQCQP:GlobaleKonvergenz:Satz:Thm3.8:Bew:wkUNDHkgtildepkUNDgtildepkUNDalphatildepkBeschraenkt})
resp.~(\ref{BundleSQPmitQCQP:Alg:Fpk+1Update}),
the Cauchy-Schwarz inequality and
(\ref{BundleSQPmitQCQP:GlobaleKonvergenz:Satz:Thm3.8:Bew:wkUNDHkgtildepkUNDgtildepkUNDalphatildepkBeschraenkt})
resp.~(\ref{BundleSQPmitQCQP:GlobaleKonvergenz:Satz:Thm3.8:Bew:DefDeltak}),
(\ref{BundleSQPmitQCQP:Alg:gpk+1Update}),
(\ref{BundleSQPmitQCQP:Alg:ghatpk+1Update}),
(\ref{BundleSQPmitQCQP:Alg:Defkappabark+1}),
(\ref{BundleSQPmitQCQP:GlobaleKonvergenz:Satz:Thm3.8:Bew:HkIsBounded}) and
the boundedness of
$\lbrace\bar{\kappa}^{k+1} \rbrace$ (by assumption),
we obtain
(\ref{BundleSQPmitQCQP:GlobaleKonvergenz:Satz:Thm3.8:Bew:fpk+1-ftildepk0COMPACT}).
\qedhere
\end{proof}
\begin{lemma}[Zero sequence]
\label{BundleSQPmitQCQP:GlobaleKonvergenz:Satz:Thm3.8:Bew:alphapk+1-alphatildepkCOMPACT}
We have
\begin{equation*}
\big\lvert(\alpha_p^{k+1}-\tilde{\alpha}_p^k)+\bar{\kappa}^{k+1}(A_p^{k+1}-\tilde{A}_p^k)
+\bar{\kappa}^{k+1}\big(F(x_k)-F(x_{k+1})\big)
\big\rvert\xrightarrow{K}0
\punkt
\cbmathBLUE
\end{equation*}
\end{lemma}
\begin{proof}
Because of $0\leq\tilde{s}_p^k\leq\Big(\tfrac{\tilde{\alpha}_p^k}{\gamma_1}\Big)^{\frac{1}{\omega_1}}$ due to (\ref{BundleSQPmitQCQP:Alg:gtildepkCOMPACT}) and (\ref{BundleSQPmitQCQP:Alg:alphatildepk}) and because of the boundedness of $\lbrace\tilde{\alpha}_p^k\rbrace$ due to (\ref{BundleSQPmitQCQP:GlobaleKonvergenz:Satz:Thm3.8:Bew:wkUNDHkgtildepkUNDgtildepkUNDalphatildepkBeschraenkt}), $\tilde{s}_p^k$ is bounded. Since the function $\xi\mapsto\xi^{\omega_1}$ with $\omega_1\geq1$ is Lipschitz continuous on every bounded subset of $\mathbb{R}_+$, there exists $c_L>0$ with
\begin{equation}
\lvert(s_p^{k+1})^{\omega_1}-(\tilde{s}_p^k)^{\omega_1}\rvert
\leq
c_L\lvert s_p^{k+1}-\tilde{s}_p^k\rvert
\punkt
\label{BundleSQPmitQCQP:GlobaleKonvergenz:Satz:Thm3.8:Bew:spk+1-stildepkLipschitz}
\end{equation}
In the case $\bar{\kappa}^{k+1}=0$, we have $\bar{\kappa}^{k+1}\tilde{\hat{s}}_p^k=0$ due to (\ref{BundleSQPmitQCQP:Alg:Defkappabark+1}) and (\ref{BundleSQPmitQCQP:Alg:ghattildepkCOMPACT}). Now consider the case $\bar{\kappa}^{k+1}>0$. Because of $0\leq\bar{\kappa}^{k+1}\tilde{\hat{s}}_p^k\leq(\bar{\kappa}^{k+1})^{\omega_2}\Big(\tfrac{\kappa^{k+1}\tilde{A}_p^k}{\gamma_2}\Big)^{\frac{1}{\omega_2}}$ due to (\ref{BundleSQPmitQCQP:Alg:ghattildepkCOMPACT}) and (\ref{BundleSQPmitQCQP:Alg:Atildepk}) and because of the boundedness of $\lbrace\bar{\kappa}^{k+1}\rbrace$ due to assumption and the boundedness of $\lbrace\bar{\kappa}^{k+1}\tilde{A}_p^k\rbrace$ due to (\ref{BundleSQPmitQCQP:GlobaleKonvergenz:Satz:Thm3.8:Bew:wkUNDHkgtildepkUNDgtildepkUNDalphatildepkBeschraenkt}), $\bar{\kappa}^{k+1}\tilde{\hat{s}}_p^k$ is bounded. Therefore, $\lbrace\bar{\kappa}^{k+1}\tilde{\hat{s}}_p^k\rbrace$ is bounded for all $\bar{\kappa}^{k+1}\geq0$. Since the function $\xi\mapsto\xi^{\omega_2}$ with $\omega_2\geq1$ is Lipschitz continuous on every bounded subset of $\mathbb{R}_+$, there exists $\bar{c}_L>0$ with $\lvert(\bar{\kappa}^{k+1}\hat{s}_p^{k+1})^{\omega_2}-(\bar{\kappa}^{k+1}\tilde{\hat{s}}_p^k)^{\omega_2}\rvert\leq\bar{c}_L\bar{\kappa}^{k+1}\lvert\hat{s}_p^{k+1}-\tilde{\hat{s}}_p^k\rvert$ and hence, using the assumption of the boundedness of $\lbrace\bar{\kappa}^{k+1}\rbrace$ and $\omega_2\geq1$ as well as setting $\hat{c}_L:=\bar{c}_L\sup_{k\geq1}{(\bar{\kappa}^{k+1})^{1+\frac{1}{\omega_2}}}<\infty$, we obtain
\begin{equation}
\bar{\kappa}^{k+1}
\lvert(\hat{s}_p^{k+1})^{\omega_2}-(\tilde{\hat{s}}_p^k)^{\omega_2}\rvert
\leq
\hat{c}_L\lvert\hat{s}_p^{k+1}-\tilde{\hat{s}}_p^k\rvert
\punkt
\label{BundleSQPmitQCQP:GlobaleKonvergenz:Satz:Thm3.8:Bew:spk+1-stildepkLipschitzNB}
\end{equation}
We remind of the formula
$
\lvert\max{(a,b)}-\max{(c,d)}\rvert
\leq
\lvert a-c\rvert+\lvert b-d\rvert
$
for all $a,b,c,d\in\mathbb{R}$.
Therefore,
we have $\lvert\alpha_p^{k+1}-\tilde{\alpha}_p^k\rvert\xrightarrow{K}0$
due to
(\ref{BundleSQPmitQCQP:Alg:alphajkCOMPACT}),
(\ref{BundleSQPmitQCQP:Alg:alphatildepk}),
(\ref{BundleSQPmitQCQP:GlobaleKonvergenz:Satz:Thm3.8:Bew:spk+1-stildepkLipschitz}), (\ref{BundleSQPmitQCQP:GlobaleKonvergenz:Satz:Thm3.8:Bew:fpk+1-ftildepk0COMPACT}),
(\ref{BundleSQPmitQCQP:GlobaleKonvergenz:Satz:Thm3.8:Bew:fxk+1-fxk0COMPACT}) and (\ref{BundleSQPmitQCQP:GlobaleKonvergenz:Satz:Thm3.8:Bew:spk+1-stildepk0COMPACT}).
Furthermore, due to
(\ref{BundleSQPmitQCQP:Alg:AjkCOMPACT}) and
(\ref{BundleSQPmitQCQP:Alg:Atildepk}),
we obtain
\begin{equation*}
\lvert A_p^{k+1}-\tilde{A}_p^k\rvert
\leq
\lvert F_p^{k+1}-\tilde{F}_p^k\rvert
+\lvert F(x_k)-F(x_{k+1})\rvert
+\gamma_2\lvert(\hat{s}_p^{k+1})^{\omega_2}-(\tilde{\hat{s}}_p^k)^{\omega_2}\rvert
\punkt
\end{equation*}
Multiplying this last inequality with $\bar{\kappa}^{k+1}\geq0$ (due to (\ref{BundleSQPmitQCQP:Alg:Defkappabark+1})) and using (\ref{BundleSQPmitQCQP:GlobaleKonvergenz:Satz:Thm3.8:Bew:spk+1-stildepkLipschitzNB}), the boundedness of $\lbrace\bar{\kappa}^{k+1}\rbrace$, (\ref{BundleSQPmitQCQP:GlobaleKonvergenz:Satz:Thm3.8:Bew:fpk+1-ftildepk0COMPACT}), (\ref{BundleSQPmitQCQP:GlobaleKonvergenz:Satz:Thm3.8:Bew:fxk+1-fxk0COMPACT}) and (\ref{BundleSQPmitQCQP:GlobaleKonvergenz:Satz:Thm3.8:Bew:spk+1-stildepk0COMPACT}) yields $\bar{\kappa}^{k+1}\lvert A_p^{k+1}-\tilde{A}_p^k\rvert\xrightarrow{K}0$
\cbstartBLUE
and $\bar{\kappa}^{k+1}\lvert F(x_k)-F(x_{k+1})\rvert\xrightarrow{K}0$.
\cbendBLUE
Therefore, using (\ref{BundleSQPmitQCQP:Alg:Defkappabark+1}), we obtain
the desired result.
\qedhere
\end{proof}
\begin{lemma}[Estimates for zero sequences]
\label{COMPACT:lemma:EstimateForZeroSequences}
Assume $\sigma(\bar{x})>0$. Then the constants
\begin{equation}
\begin{split}
c&:=\sup_{k\geq1}{\Big(
      \lvert H_kg_{k+1}^{k+1}\rvert,
      \lvert H_k(\tilde{g}_p^k+\bar{\kappa}^{k+1}\tilde{\hat{g}}_p^k)\rvert,
      \sqrt{\tau_k}
      \Big)}
\komma\quad
\cbeqDVItwo
\delta:=\tfrac{\sigma(\bar{x})}{2}
\komma\quad
\bar{c}:=\delta\tfrac{1-m_R}{4c}
\komma
\\
\tilde{c}
&:=
\sup_{k\geq1}{
(\lvert g_{k+1}^{k+1}\rvert+\lvert\tilde{g}_p^k+\bar{\kappa}^{k+1}\tilde{\hat{g}}_p^k\rvert)
}
\komma\quad
C_6
:=
\tilde{c}C_5
\max{(2c,1,\tfrac{1}{2}\tilde{c}C_5)}
\punkt
\label{BundleSQPmitQCQP:GlobaleKonvergenz:Satz:Thm3.8:Bew:DefcCOMPACT}
\end{split}
\end{equation}
are finite and there exists $\bar{k}\geq0$ such that
\begin{equation}
\begin{split}
\tfrac{1}{2}  
\bar{c}^2&>
4c\lvert\Delta_k\rvert+\tfrac{\lvert\Delta_k\rvert^2}{2}
+\big\lvert(\alpha_p^{k+1}-\tilde{\alpha}_p^k)+\bar{\kappa}^{k+1}(A_p^{k+1}-\tilde{A}_p^k)+\nonumber\\
&\hspace{0.55\textwidth}\bar{\kappa}^{k+1}\big(F(x_k)-F(x_{k+1})\big)\big\rvert
\cbmathBLUE
\\
\tfrac{1}{2}\bar{c}^2&>
C_6
(
\lvert\bar{\kappa}^{k+2}-\bar{\kappa}^{k+1}\rvert
+
\lvert\Delta_k\rvert
\cdot
\lvert\bar{\kappa}^{k+2}-\bar{\kappa}^{k+1}\rvert
+
\lvert\bar{\kappa}^{k+2}-\bar{\kappa}^{k+1}\rvert^2
)
\label{BundleSQPmitQCQP:GlobaleKonvergenz:Satz:Thm3.8:Bew:AbschaetzungFuercbarQuadratCOMPACT}
\end{split}
\end{equation}
\cbendDVI
hold for all $k>\bar{k}$.
\end{lemma}
\begin{proof}
Then $c$ is finite due to (\ref{BundleSQPmitQCQP:GlobaleKonvergenz:Satz:Thm3.8:Bew:DefcCOMPACT}),
\cbstartDVI
(\ref{BundleSQPmitQCQP:GlobaleKonvergenz:Satz:Thm3.8:Bew:HkIsBounded}),
\cbendDVI
(\ref{BundleSQPmitQCQP:GlobaleKonvergenz:Satz:Thm3.8:Bew:gkkUNDHkgkkUNDalphakkBeschraenkt}) and (\ref{BundleSQPmitQCQP:GlobaleKonvergenz:Satz:Thm3.8:Bew:wkUNDHkgtildepkUNDgtildepkUNDalphatildepkBeschraenkt}).
Furthermore, we have $c>0$ (If we had $c=0$, then using (\ref{BundleSQPmitQCQP:GlobaleKonvergenz:Satz:Thm3.8:Bew:DefcCOMPACT}),
\cbstartDVI
(\ref{BundleSQPmitQCQP:GlobaleKonvergenz:Satz:Thm3.8:Bew:ADDON:Deftauk})
\cbendDVI
and (\ref{BundleSQPmitQCQP:Alg:wk}) would imply $w_k=0$ for all $k\geq1$, which is a contradiction to assumption (\ref{BundleSQPmitQCQP:GlobaleKonvergenz:Vor:AlgTerminiertNicht})).
Due to
(\ref{BundleSQPmitQCQP:GlobaleKonvergenz:Satz:Thm3.8:Bew:DefcCOMPACT}),
$\sigma(\bar{x})>0$ and
$1-m_R>0$ (cf.~the initialization of Algorithm \ref{BundleSQPmitQCQP:Alg:GesamtAlgMitQCQP}),
we have $\bar{c}=\tfrac{\sigma(\bar{x})}{2}\cdot\tfrac{1-m_R}{4c}$, where $\sigma(\bar{x})>0$ implies $\bar{c}>0$, and Lemma \ref{BundleSQPmitQCQP:GlobaleKonvergenz:Satz:Thm3.8:Bew:sigmaendlichCOMPACT} implies $\bar{c}<\infty$.
\cbstartDVI
Due to (\ref{BundleSQPmitQCQP:GlobaleKonvergenz:Satz:Thm3.8:Bew:DefcCOMPACT}), (\ref{BundleSQPmitQCQP:GlobaleKonvergenz:Satz:Thm3.8:Bew:gkkUNDHkgkkUNDalphakkBeschraenkt}), (\ref{BundleSQPmitQCQP:GlobaleKonvergenz:Satz:Thm3.8:Bew:wkUNDHkgtildepkUNDgtildepkUNDalphatildepkBeschraenkt}) and the assumption of the boundedness of $\lbrace\bar{\kappa}^{k+1}\rbrace$, $\tilde{c}\geq0$ is bounded. Therefore, (\ref{BundleSQPmitQCQP:GlobaleKonvergenz:Satz:Thm3.8:Bew:DefcCOMPACT}) and (\ref{KonvergenzErweiterung:Satz:MatrixAbschaetzungMitHigham:ZkAbschaetzungCOMPACT}) imply $0\leq C_6<\infty$.
\cbendDVI
Since
\cbstartBLUE
$4c\lvert\Delta_k\rvert+\tfrac{\lvert\Delta_k\rvert^2}{2}+\big\lvert(\alpha_p^{k+1}-\tilde{\alpha}_p^k)+\bar{\kappa}^{k+1}(A_p^{k+1}-\tilde{A}_p^k)+\bar{\kappa}^{k+1}\big(F(x_k)-F(x_{k+1})\big)\big\rvert\xrightarrow{K}0$
\cbendBLUE
due to (\ref{BundleSQPmitQCQP:GlobaleKonvergenz:Satz:Thm3.8:Bew:fpk+1-ftildepk0COMPACT}) and
Lemma \ref{BundleSQPmitQCQP:GlobaleKonvergenz:Satz:Thm3.8:Bew:alphapk+1-alphatildepkCOMPACT}
\cbstartDVI
and since $C_6
(
\lvert\bar{\kappa}^{k+2}-\bar{\kappa}^{k+1}\rvert
+
\lvert\Delta_k\rvert
\cdot
\lvert\bar{\kappa}^{k+2}-\bar{\kappa}^{k+1}\rvert
+
\lvert\bar{\kappa}^{k+2}-\bar{\kappa}^{k+1}\rvert^2
)\rvert\xrightarrow{K}0$
due to (\ref{BundleSQPmitQCQP:GlobaleKonvergenz:Satz:Thm3.8:Bew:ADDON:DefDiffbarkappa}), there exists $\bar{k}\geq0$
such that (\ref{BundleSQPmitQCQP:GlobaleKonvergenz:Satz:Thm3.8:Bew:AbschaetzungFuercbarQuadratCOMPACT}) holds
for all $k>\bar{k}$.
\qedhere
\end{proof}
\begin{lemma}[Estimate with error term]
We define for $k\geq1$
\begin{equation}
\begin{split}
q_k
&:=
H_kg_{k+1}^{k+1}
\komma\quad
p_k
:=
H_k(\tilde{g}_p^k+\bar{\kappa}^{k+1}\tilde{\hat{g}}_p^k)
\\
e_k
&:=
(2c+\lvert\Delta_k\rvert)
\tilde{c}\lvert E_k\rvert
+
\tfrac{1}{2}
\tilde{c}^2\lvert E_k\rvert^2
\komma\quad
E_k
:=
H_{k+1}-H_k
\punkt
\label{BundleSQPmitQCQP:GlobaleKonvergenz:Satz:Thm3.8:Bew:ADDON:DefqkCOMPACT}
\end{split}
\end{equation}
Then we have for all $\nu\in[0,1]$ and for all $k\geq1$
\begin{equation}
\tfrac{1}{2}\lvert
\nu H_{k+1}g_{k+1}^{k+1}+(1-\nu)H_{k+1}(g_p^{k+1}+\bar{\kappa}^{k+1}\hat{g}_p^{k+1})
\rvert^2
\leq
\tfrac{1}{2}\lvert
\nu q_k+(1-\nu)(p_k+\Delta_k)
\rvert^2
+
e_k
\punkt
\label{BundleSQPmitQCQP:GlobaleKonvergenz:Satz:Thm3.8:Bew:ADDON:AbschaetzungMitEk}
\end{equation}
\end{lemma}
\begin{proof}
Setting $z_k:=E_kg_{k+1}^{k+1}$, we obtain $H_{k+1}g_{k+1}^{k+1}=q_k+z_k$ due to (\ref{BundleSQPmitQCQP:GlobaleKonvergenz:Satz:Thm3.8:Bew:ADDON:DefqkCOMPACT}). Setting $\hat{z}_k:=E_k(\tilde{g}_p^k+\bar{\kappa}^{k+1}\tilde{\hat{g}}_p^k)$, we obtain
$H_{k+1}(g_p^{k+1}+\bar{\kappa}^{k+1}\hat{g}_p^{k+1})
=
p_k+\Delta_k+\hat{z}_k$
due to (\ref{BundleSQPmitQCQP:GlobaleKonvergenz:Satz:Thm3.8:Bew:DefDeltak}) and (\ref{BundleSQPmitQCQP:GlobaleKonvergenz:Satz:Thm3.8:Bew:ADDON:DefqkCOMPACT}). Furthermore, we estimate for all $\nu\in[0,1]$
\begin{equation*}
\big(\nu q_k+(1-\nu)(p_k+\Delta_k)\big)^T\big(\nu z_k+(1-\nu)\hat{z}_k\big)
\leq
(2c+\lvert\Delta_k\rvert)
\lvert\nu z_k+(1-\nu)\hat{z}_k\rvert
\end{equation*}
due to the Cauchy-Schwarz inequality,
(\ref{BundleSQPmitQCQP:GlobaleKonvergenz:Satz:Thm3.8:Bew:ADDON:DefqkCOMPACT}) and
(\ref{BundleSQPmitQCQP:GlobaleKonvergenz:Satz:Thm3.8:Bew:DefcCOMPACT})
as well as
$\lvert\nu z_k+(1-\nu)\hat{z}_k\rvert
\leq
\tilde{c}\lvert E_k\rvert$
due to (\ref{BundleSQPmitQCQP:GlobaleKonvergenz:Satz:Thm3.8:Bew:DefcCOMPACT}). Hence we obtain
(\ref{BundleSQPmitQCQP:GlobaleKonvergenz:Satz:Thm3.8:Bew:ADDON:AbschaetzungMitEk})
due to and (\ref{BundleSQPmitQCQP:GlobaleKonvergenz:Satz:Thm3.8:Bew:ADDON:DefqkCOMPACT}).
\qedhere
\end{proof}
\begin{lemma}[Index construction]
\label{BundleSQPmitQCQP:GlobaleKonvergenz:Satz:Thm3.8:Bew:kEQk0PLUSiCOMPACT}
Assume $\sigma(\bar{x})>0$ and define
\begin{equation}
\begin{split}
\hat{r}:=\tfrac{3}{2}\cdot\tfrac{c^2}{\bar{c}^2}+i_m
\komma\quad
r:=i_l+\hat{r}
\punkt
\label{BundleSQPmitQCQP:GlobaleKonvergenz:Satz:Thm3.8:Bew:rhat:CORRECTIONCOMPACT}
\end{split}
\end{equation}
Then there exists a finite index $k_0\in K$ such that
\begin{align}
w_k&\geq\delta
\komma\quad
t_L^k<t_0^k
\label{BundleSQPmitQCQP:GlobaleKonvergenz:Satz:Thm3.8:Bew:wkGeqdeltaCOMPACT}
\\
i_n&>i_l+i_m
\label{BundleSQPmitQCQP:GlobaleKonvergenz:Satz:Thm3.8:Bew:inGROESSERim}
\end{align}
hold for
$
k:=k_0+i_l+i
$
with
$i\in[i_m,\hat{r}]\cap\lbrace0,1,\dots\rbrace$.
\end{lemma}
\begin{proof}
We obtain
$
r
\geq
i_l+i_m
\geq
i_l
\geq
0
$
due to
(\ref{BundleSQPmitQCQP:GlobaleKonvergenz:Satz:Thm3.8:Bew:rhat:CORRECTIONCOMPACT}) and
the initialization of Algorithm \ref{BundleSQPmitQCQP:Alg:GesamtAlgMitQCQP}.
Therefore,
$[i_l,r]$ is a well-defined interval and since $i_l\geq0$ is a natural number (cf.~Algorithm \ref{BundleSQPmitQCQP:Alg:GesamtAlgMitQCQP}),
there exists
$i\in[i_l,r]\cap\lbrace0,1,\dots\rbrace\subseteq[0,r]$.
Furthermore, $[i_m,\hat{r}]$ is a well-defined interval and since $i_m\geq0$ is a natural number (cf.~Algorithm \ref{BundleSQPmitQCQP:Alg:GesamtAlgMitQCQP}), there exists
$
i\in[i_m,\hat{r}]\cap\lbrace0,1,\dots\rbrace\subseteq[0,\hat{r}]
$.
The assumptions of Proposition \ref{BundleSQPmitQCQP:GlobaleKonvergenz:Lemma3.7} are satisfied --- (\ref{BundleSQPmitQCQP:GlobaleKonvergenz:Vor:AlgTerminiertNicht}) holds
\cbstartDVI
due to assumption,
$\lbrace H_k\rbrace$ is bounded due (\ref{BundleSQPmitQCQP:GlobaleKonvergenz:Satz:Thm3.8:Bew:HkIsBounded}),
\cbendDVI
we have $x_k\xrightarrow{K}\bar{x}$ due to (\ref{BundleSQPmitQCQP:GlobaleKonvergenz:Satz:Thm3.8:Bew:xkGegenxbar}),
we have $\sigma(\bar{x})>0$ due to assumption and
Lemma \ref{BundleSQPmitQCQP:GlobaleKonvergenz:Satz:Thm3.8:Bew:sigmaendlichCOMPACT},
$r\geq0$ is a fixed number due to (\ref{BundleSQPmitQCQP:GlobaleKonvergenz:Satz:Thm3.8:Bew:rhat:CORRECTIONCOMPACT}),
the choice
$
\varepsilon_0:=t_0^{\inf}>0
$
yields a fixed positive number $\varepsilon_0$ due to (\ref{Theorem3.8:Presumptiont0infGT0}) ---
and therefore we can apply Proposition \ref{BundleSQPmitQCQP:GlobaleKonvergenz:Lemma3.7}: For $r$ defined in (\ref{BundleSQPmitQCQP:GlobaleKonvergenz:Satz:Thm3.8:Bew:rhat:CORRECTIONCOMPACT}) there exists $\tilde{k}\geq0$ with
\begin{equation}
w_{k+i}
\geq
\tfrac{\sigma(\bar{x})}{2}
=
\delta
\komma\quad
t_L^{k+i}
<
\varepsilon_0
=
t_0^{\inf}
\label{BundleSQPmitQCQP:GlobaleKonvergenz:Satz:Thm3.8:Bew:Lemma3.7Abschaetzungen}
\end{equation}
for all $k>\tilde{k}$, $k\in K$ and for all $0\leq i\leq r$ due to
(\ref{BundleSQPmitQCQP:GlobaleKonvergenz:Lemma3.7:AussagetLk+it0COMPACT}) and
(\ref{BundleSQPmitQCQP:GlobaleKonvergenz:Satz:Thm3.8:Bew:DefcCOMPACT}). Since $K$ is an infinite set due to (\ref{BundleSQPmitQCQP:GlobaleKonvergenz:Satz:Thm3.8:Bew:xkGegenxbar}) ($K\subset\lbrace1,2,\dots,\rbrace$), we can choose $k_0\in K$ with $k_0>\max{(\tilde{k},\bar{k})}\geq\tilde{k}$ ($\bar{k}$ was introduced in Lemma \ref{COMPACT:lemma:EstimateForZeroSequences}). Hence, (\ref{BundleSQPmitQCQP:GlobaleKonvergenz:Satz:Thm3.8:Bew:Lemma3.7Abschaetzungen}) holds in particular for all $k\geq k_0$ and hence for $k=k_0$,
i.e.~
$
w_{k_0+i}
\geq
\delta
$
and
$
t_L^{k_0+i}
<
t_0^{\inf}
$
for all $0\leq i\leq r$. Because of $t_L^{k_0+i}\leq t_0^{k_0+i}$ for all $0\leq i\leq r$ due to (\ref{Theorem3.8:Presumptiont0infGT0}), we obtain
\begin{equation}
w_{k_0+i}
\geq
\delta
\komma\quad
t_L^{k_0+i}
<
t_0^{k_0+i}
\label{BundleSQPmitQCQP:GlobaleKonvergenz:Satz:Thm3.8:Bew:Lemma3.7Abschaetzungen:CORRECTION3}
\end{equation}
for all $0\leq i\leq r$. Due to
(\ref{BundleSQPmitQCQP:GlobaleKonvergenz:Satz:Thm3.8:Bew:rhat:CORRECTIONCOMPACT}), (\ref{BundleSQPmitQCQP:GlobaleKonvergenz:Satz:Thm3.8:Bew:Lemma3.7Abschaetzungen:CORRECTION3}) holds in particular for all $i\in[i_l,r]=i_l+[0,\hat{r}]$ which yields
$
w_{k_0+i_l+i}
\geq
\delta
$
and
$
t_L^{k_0+i_l+i}
<
t_0^{k_0+i_l+i}
$
with $i\in[0,\hat{r}]$. In particular,
these last two inequalities
hold for all $i\in[i_m,\hat{r}]\cap\lbrace0,1,\dots,\rbrace$ and now setting
$
k:=k_0+i_l+i
$
yields
the desired index and
that
(\ref{BundleSQPmitQCQP:GlobaleKonvergenz:Satz:Thm3.8:Bew:wkGeqdeltaCOMPACT}) holds
after step 6 (line search) of Algorithm \ref{BundleSQPmitQCQP:Alg:GesamtAlgMitQCQP}.
Due to (\ref{BundleSQPmitQCQP:GlobaleKonvergenz:Satz:Thm3.8:Bew:Lemma3.7Abschaetzungen:CORRECTION3}), we have $t_L^{k_0+i}<t_0^{k_0+i}$ in particular for all $0\leq i\leq i_l+i_m$. Consequently, the case (\ref{BundleSQPmitQCQP:Alg:inUpdate}) always occurs for the $i_l+i_m+1$ subsequent iterations $k_0+0,\dots,k_0+i_l,\dots,k_0+i_l+i_m$ (Remember: $i_n\geq0$ denotes the number of subsequent short and null steps according to the initialization of Algorithm \ref{BundleSQPmitQCQP:Alg:GesamtAlgMitQCQP}) and therefore
(\ref{BundleSQPmitQCQP:GlobaleKonvergenz:Satz:Thm3.8:Bew:inGROESSERim}) holds
at the end of iteration $k_0+i_l+i_m$ (even if the initial value of $i_n$ is zero at the beginning of iteration $k_0+0$) after step 6 (line search) of Algorithm \ref{BundleSQPmitQCQP:Alg:GesamtAlgMitQCQP}.
\qedhere
\end{proof}
\begin{lemma}[Error estimate]
\label{BundleSQPmitQCQP:GlobaleKonvergenz:Satz:Thm3.8:Bew:ekAbschaetzungStattHk+1=HkCOMPACT}
For $k$ defined in Lemma \refH{BundleSQPmitQCQP:GlobaleKonvergenz:Satz:Thm3.8:Bew:kEQk0PLUSiCOMPACT} we have
$
e_k
<
\tfrac{1}{2}\bar{c}^2
$.
\end{lemma}
\begin{proof}
Since $i_n>i_l+i_m$ due to (\ref{BundleSQPmitQCQP:GlobaleKonvergenz:Satz:Thm3.8:Bew:inGROESSERim}) and since $i_n$ increases at most by one at each iteration due to (\ref{BundleSQPmitQCQP:Alg:inUpdate}), we have at iteration $k$ at least $i_n\geq i_l+i_m$ and hence either the case (\ref{BundleSQPmitQCQP:Alg:ModifikationVonGhatbarjkBeiGleichheitsCase}) or (\ref{BundleSQPmitQCQP:Alg:KeineModifikationVonGhatbarjk}) occurs (at iteration $k$). Furthermore, since $i_n>i_l+i_m$ due to (\ref{BundleSQPmitQCQP:GlobaleKonvergenz:Satz:Thm3.8:Bew:inGROESSERim}), the cases (\ref{BundleSQPmitQCQP:Alg:KeineModifikationVonGpk}) and (\ref{BundleSQPmitQCQP:Alg:KeineModifikationVonGhatbarjk}) occur at iteration $k+1$. Therefore, combining these facts yields $E_k=Z_k(\bar{\kappa}^{k+2})-Z_k(\bar{\kappa}^{k+1})$ due to (\ref{BundleSQPmitQCQP:GlobaleKonvergenz:Satz:Thm3.8:Bew:ADDON:DefqkCOMPACT}), (\ref{BundleSQPmitQCQP:Alg:Hk}), (\ref{BundleSQPmitQCQP:Alg:Defkappabark+1}),
\cbstartMIFFLIN
the fact that $\sum_{j\in J_k}{\lambda_j^k}+\lambda_p^k=1=\sum_{j\in J_{k+1}}{\lambda_j^{k+1}}+\lambda_p^{k+1}$,
\cbendMIFFLIN
and (\ref{KonvergenzErweiterung:Satz:MatrixAbschaetzungMitHigham:DefZk}). Since $\lbrace\bar{\kappa}^{k+1}\rbrace$ is bounded and $\lbrace(\widebar{W}_p^k)^{-\mathalf}\rbrace$ is bounded as well as uniformly positive definite (by assumption), we can make use of Proposition \ref{KonvergenzErweiterung:Satz:MatrixAbschaetzungMitHigham:ZkAbschaetzungSatz} and hence we obtain
$\lvert E_k\rvert
\leq
C_5\lvert\bar{\kappa}^{k+2}-\bar{\kappa}^{k+1}\rvert$
due to (\ref{KonvergenzErweiterung:Satz:MatrixAbschaetzungMitHigham:ZkAbschaetzungCOMPACT}). Consequently, we
obtain
he desired estimate
due to (\ref{BundleSQPmitQCQP:GlobaleKonvergenz:Satz:Thm3.8:Bew:ADDON:DefqkCOMPACT}), (\ref{BundleSQPmitQCQP:GlobaleKonvergenz:Satz:Thm3.8:Bew:DefcCOMPACT}) and (\ref{BundleSQPmitQCQP:GlobaleKonvergenz:Satz:Thm3.8:Bew:AbschaetzungFuercbarQuadratCOMPACT}).
\qedhere
\end{proof}
\begin{lemma}[Termination criterion estimate]
\label{BundleSQPmitQCQP:GlobaleKonvergenz:Satz:Thm3.8:Bew:wk+1LeqMinQCOMPACT}
For $k$ defined in Lemma \refH{BundleSQPmitQCQP:GlobaleKonvergenz:Satz:Thm3.8:Bew:kEQk0PLUSiCOMPACT}
a short or null step which changes the model of the objective function is executed
and
\begin{align*}
w_{k+1}
&\leq
\big\lvert(\alpha_p^{k+1}-\tilde{\alpha}_p^k)+\bar{\kappa}^{k+1}(A_p^{k+1}-\tilde{A}_p^k)+\bar{\kappa}^{k+1}\big(F(x_k)-F(x_{k+1})\big)\big\rvert
+e_k
\cbmathBLUE
\\
&\hspace{20pt}+\min_{\nu\in[0,1]}
\tfrac{1}{2}\lvert
\nu q_k+(1-\nu)(p_k+\Delta_k)
\rvert^2
+\nu\alpha_{k+1}^{k+1}
+(1-\nu)
\tau_k
\punkt
\end{align*}
\end{lemma}
\begin{proof}
Combining (\ref{BundleSQPmitQCQP:GlobaleKonvergenz:Satz:Thm3.8:Bew:wkGeqdeltaCOMPACT}) with step 6 (line search) of Algorithm \ref{BundleSQPmitQCQP:Alg:GesamtAlgMitQCQP} and considering the case $i_n>i_l+i_m\geq i_l$ due to (\ref{BundleSQPmitQCQP:GlobaleKonvergenz:Satz:Thm3.8:Bew:inGROESSERim}) in the line search (Algorithm \ref{BundleSQP:AlgNB:LinesearchMitQCQP}), we obtain that at iteration $k$ a short or null step which changes the model of the objective function is executed. Furthermore, $i_s$ is unchanged (since no serious step is executed), i.e.~$i_s\leq i_r$ (no bundle reset) still holds (If $i_s>i_r$, then we would have had a serious step at iteration $k$, as a bundle reset can only occur after a serious step). Therefore, $(\lambda,\lambda_p,\mu,\mu_p)\in\mathbb{R}^{2(\lvert J_{k+1}\rvert+1)}$ with
\begin{align}
\lambda_j
&:=
\left\lbrace
\begin{array}{ll}
\nu & \textnormal{for }j=k+1\\
0   & \textnormal{for }j\in J_{k+1}\setminus\lbrace k+1\rbrace
\end{array}
\right\rbrace
~,\nonumber\\
\lambda_p
&:=1-\nu
~,~
\mu_j
:=0
~~~\forall j\in J_{k+1}
~,~
\mu_p
:=(1-\nu)\bar{\kappa}^{k+1}
\komma
\label{BundleSQPmitQCQP:GlobaleKonvergenz:Satz:Thm3.8:Bew:lambdanuZulaessig}
\end{align}
where $\nu\in[0,1]$, is feasible for the $(k+1)$st (dual) problem (\ref{Luksan:Alg:QCQPTeilproblem:DualesProblem}) (Note: This problem is written as a minimization problem) and, hence, due to
(\ref{BundleSQPmitQCQP:Satz:GlobaleKonvergenz:Lemma3.5:wkCOMPACT}),
(\ref{BundleSQPmitQCQP:GlobaleKonvergenz:Lemma3.5:Defwhat}),
(\ref{BundleSQPmitQCQP:Alg:gtildepkCOMPACT}),
\cbstartBLUE
(\ref{BundleSQPmitQCQP:Alg:ghattildepkCOMPACT})\cbendBLUE,
(\ref{BundleSQPmitQCQP:GlobaleKonvergenz:Lemma3.5:DefalphahatCOMPACT}),
(\ref{BundleSQPmitQCQP:Alg:Defkappabark+1}),
inserting the feasible point from (\ref{BundleSQPmitQCQP:GlobaleKonvergenz:Satz:Thm3.8:Bew:lambdanuZulaessig}),
\cbstartDVI
(\ref{BundleSQPmitQCQP:GlobaleKonvergenz:Satz:Thm3.8:Bew:ADDON:Deftauk}),
\cbendDVI
taking into account that $\nu\in[0,1]$
\cbstartDVI
and (\ref{BundleSQPmitQCQP:GlobaleKonvergenz:Satz:Thm3.8:Bew:ADDON:AbschaetzungMitEk}),
\cbendDVI
we estimate (Note: $\hat{w}_{k+1}$ in (\ref{BundleSQPmitQCQP:GlobaleKonvergenz:Lemma3.5:Defwhat}) is the optimal function value of (\ref{Luksan:Alg:QCQPTeilproblem:DualesProblem}))
\cbstartDVI
\begin{align*}
w_{k+1}
&
\leq
\tfrac{1}{2}\lvert
\nu q_k+(1-\nu)(p_k+\Delta_k)
\rvert^2
+
e_k
+\nu\alpha_{k+1}^{k+1}
+(1-\nu)
\tau_k
\\
&\hspace{20pt}
+\big\lvert(\alpha_p^{k+1}-\tilde{\alpha}_p^k)+\bar{\kappa}^{k+1}(A_p^{k+1}-\tilde{A}_p^k)+\bar{\kappa}^{k+1}\big(F(x_k)-F(x_{k+1})\big)\big\rvert
\cbmathBLUE
\end{align*}
and consequently, since $\nu\in[0,1]$ is arbitrary, we obtain the desired estimate.
\cbendDVI
\qedhere
\end{proof}
\begin{lemma}[Termination criterion is shrinking]
\label{BundleSQPmitQCQP:GlobaleKonvergenz:Satz:Thm3.8:Bew:wk+1ltwk-cbarQuadratCOMPACT}
For $k$ defined in Lemma \refH{BundleSQPmitQCQP:GlobaleKonvergenz:Satz:Thm3.8:Bew:kEQk0PLUSiCOMPACT} we have
$
w_{k+1}<w_k-\bar{c}^2
$.
\end{lemma}
\begin{proof}
Since for
$p:=p_k$,
$g:=q_k$,
$\Delta:=\Delta_k$,
$
v:=
v_k-\tfrac{1}{2}d_k^T\big(\sum_{j\in J_k}
\lambda_j^k\widebar{G}_j^k+\lambda_p^k\widebar{G}^k+
\mu_j^k\widebar{\hat{G}}_j^k+\mu_p^k\widebar{\hat{G}}^k
\big)d_k
$,
$w:=w_k$,
$\beta:=\alpha_{k+1}^{k+1}$,
$m:=m_R$ and
$\alpha:=\tau_k$,
the assumptions of Proposition \ref{Luksan:GlobaleKonvergenz:Lemma3.4} are satisfied
and since we have $\delta^2\tfrac{(1-m_R)^2}{8c^2}=2\bar{c}^2$ due to (\ref{BundleSQPmitQCQP:GlobaleKonvergenz:Satz:Thm3.8:Bew:DefcCOMPACT}),
now applying Proposition \ref{Luksan:GlobaleKonvergenz:Lemma3.4} yields the desired estimate due to
Lemma \ref{BundleSQPmitQCQP:GlobaleKonvergenz:Satz:Thm3.8:Bew:wk+1LeqMinQCOMPACT},
(\ref{BundleSQPmitQCQP:GlobaleKonvergenz:Satz:Thm3.8:Bew:AbschaetzungFuercbarQuadratCOMPACT}),
Lemma \ref{BundleSQPmitQCQP:GlobaleKonvergenz:Satz:Thm3.8:Bew:ekAbschaetzungStattHk+1=HkCOMPACT} and
(\ref{BundleSQPmitQCQP:GlobaleKonvergenz:Satz:Thm3.8:Bew:wkGeqdeltaCOMPACT}).
\qedhere
\end{proof}
\begin{lemma}[Contradiction]
\label{ASUS:lemma:Contradiction}
For $k_0$ from Lemma \refH{BundleSQPmitQCQP:GlobaleKonvergenz:Satz:Thm3.8:Bew:kEQk0PLUSiCOMPACT}
we have $w_{k_0+n+1}<0$.
\end{lemma}
\begin{proof}
Set $n:=\max_{z\leq\hat{r},z\in\lbrace0,1,\dots\rbrace}{z}$ (Note that $\hat{r}>0$ due to (\ref{BundleSQPmitQCQP:GlobaleKonvergenz:Satz:Thm3.8:Bew:rhat:CORRECTIONCOMPACT})), then we have $n+1>\hat{r}$ and hence (\ref{BundleSQPmitQCQP:GlobaleKonvergenz:Satz:Thm3.8:Bew:rhat:CORRECTIONCOMPACT}) implies $-\bar{c}^2(n+1-i_m)<-\tfrac{3}{2}c^2$. Now, applying
Lemma \ref{BundleSQPmitQCQP:GlobaleKonvergenz:Satz:Thm3.8:Bew:wk+1ltwk-cbarQuadratCOMPACT}
$(n-i_m)+1$ times as well as using (\ref{BundleSQPmitQCQP:Alg:wk}),
\cbstartDVI
(\ref{BundleSQPmitQCQP:GlobaleKonvergenz:Satz:Thm3.8:Bew:ADDON:Deftauk})
\cbendDVI
and (\ref{BundleSQPmitQCQP:GlobaleKonvergenz:Satz:Thm3.8:Bew:DefcCOMPACT}) yields
\begin{equation*}
w_{k_0+n+1}
<
w_{k_0+i_m}-(n+1-i_m)\bar{c}^2
<
w_{k_0+i_m}-\tfrac{3}{2}c^2
\leq
\tfrac{1}{2}c^2+c^2-\tfrac{3}{2}c^2
=
0
\punkt
\qedhere
\end{equation*}
\end{proof}
\begin{theorem}
\label{BundleSQPmitQCQP:GlobaleKonvergenz:Satz:Thm3.8:Bew:Theorem}
Let Assumption \ref{COMPACT:presumption:AssumptionForMainTheorem} be satisfied.
Then there exists $\bar{\kappa}\in\Rpos$
such that \refh{AOB:Satz:KarushJohnUnglgsNBNichtGlattAlternativformGlgsSystem} holds for $(\bar{x},\bar{\kappa})$,
i.e.~each accumulation point of the sequence of iteration points $\lbrace x_k\rbrace$ is stationary for the optimization problem \refh{BundleSQP:OptProblem}.
\end{theorem}
\begin{proof}
(by contradiction) Since $\lbrace(x_k,\bar{\kappa}^{k+1})\rbrace$ is bounded 
\cbstartDVI
and $\lbrace(\widebar{W}_p^k)^{-\mathalf}\rbrace$ is uniformly positive definite (both due to assumption) ,
\cbendDVI
the statement follows from Proposition \ref{BundleSQPmitQCQP:GlobaleKonvergenz:Satz:Lemma3.6}, if we can show $\sigma(\bar{x})=0$. We suppose this is false, i.e.~we have due to (\ref{BundleSQPmitQCQP:GlobaleKonvergenz:Def:sigma})
$
\sigma(\bar{x})>0
$
or
$
\sigma(\bar{x})=\infty
$.
Due to Assumption \ref{COMPACT:presumption:AssumptionForMainTheorem}, we can make use of
Lemma \ref{COMPACT:lemma:BoundedBasicSequences}--\ref{BundleSQPmitQCQP:GlobaleKonvergenz:Satz:Thm3.8:Bew:sigmaendlichCOMPACT}, which implies that only the case $\sigma(\bar{x})>0$ occurs.
Therefore, we can use Lemma \ref{COMPACT:lemma:CauchySequences}--\ref{ASUS:lemma:Contradiction},
which yields a contradiction to the non-negativity of $w_k$ for all $k\geq1$ due to (\ref{BundleSQPmitQCQP:Satz:GlobaleKonvergenz:Lemma3.5:wkCOMPACT}).
\qedhere
\end{proof}
\begin{remark}\label{BundleSQPmitQCQP:GlobaleKonvergenz:Bem:Thm3.8}
In examples that do not satisfy the nonsmooth constraint qualification (\ref{AOB:Satz:KarushJohnUnglgsNBNichtGlattAlternativform:CQ}), $\bar{\kappa}^{k+1}$ became very large in Algorithm \ref{BundleSQPmitQCQP:Alg:GesamtAlgMitQCQP} (Note that Theorem \ref{BundleSQPmitQCQP:GlobaleKonvergenz:Satz:Thm3.8:Bew:Theorem} has in particular the assumption that $\bar{\kappa}^{k+1}$ is bounded).

The assumption (\ref{BundleSQPmitQCQP:GlobaleKonvergenz:Satz:Thm3.8:Bew:ADDON:DefDiffbarkappa}) of Theorem \ref{BundleSQPmitQCQP:GlobaleKonvergenz:Satz:Thm3.8:Bew:Theorem} was satisfied in all numerical examples
in \citet[\GenaueAngabeFive]{HannesPaperA} in which the termination criterion of Algorithm \ref{BundleSQPmitQCQP:Alg:GesamtAlgMitQCQP} was satisfied.

If $t_0^k$ is only modified in, e.g., finitely many iterations of Algorithm \ref{BundleSQPmitQCQP:Alg:GesamtAlgMitQCQP}, then (\ref{Theorem3.8:Presumptiont0infGT0}) is satisfied (cf.~Remark \ref{remark:t0infWithGraphic}).

For an unconstrained optimization problem we obtain in the proof of
Lemma \ref{BundleSQPmitQCQP:GlobaleKonvergenz:Satz:Thm3.8:Bew:ekAbschaetzungStattHk+1=HkCOMPACT}
that $E_k=\zeroMatrix{N}$ which implies that $e_k=0$ due to (\ref{BundleSQPmitQCQP:GlobaleKonvergenz:Satz:Thm3.8:Bew:ADDON:DefqkCOMPACT}). Therefore,
Lemma \ref{BundleSQPmitQCQP:GlobaleKonvergenz:Satz:Thm3.8:Bew:ekAbschaetzungStattHk+1=HkCOMPACT}
is trivially satisfied in the unconstrained case, since $\bar{c}$ from \ref{BundleSQPmitQCQP:GlobaleKonvergenz:Satz:Thm3.8:Bew:DefcCOMPACT} is positive.

\cbstartMIFFLIN
If we demand that all assumptions in the proof of convergence, which we imposed on $\widebar{W}_p^k$, are satisfied for $\sum_{j\in J_k}{\lambda_j^k\widebar{G}_j^k+\lambda_p^k\widebar{G}^k}$, then the convergence result also holds in the case $\widebar{W}_p^k=\zeroMatrix{N}$. This is important, since first numerical results in the unconstrained case showed a better performance for the choice $\widebar{W}_p^k=\zeroMatrix{N}$, which is due to the fact that otherwise for a smooth, convex objective function $f$ the Hessian information in the QCQP (\ref{Luksan:Alg:QCQPTeilproblem}) is distorted --- this can be seen by putting the constraints of the QCQP (\ref{Luksan:Alg:QCQPTeilproblem}) into its objective function, which is then given by
\begin{multline*}
\max_{j\in J_k}{\big(-\alpha_j^k+d^Tg_j^k+\tfrac{1}{2}d^T\widebar{G}_j^kd\big)}+\tfrac{1}{2}d^T\widebar{W}_p^kd
\\=
\max_{j\in J_k}{\big(-\alpha_j^k+d^Tg_j^k+\tfrac{1}{2}d^T(\widebar{G}_j^k+\widebar{W}_p^k)d\big)}
\punkt
\end{multline*}
\cbendMIFFLIN
\end{remark}

\section{Conclusion}
In this paper we investigated the possibility of extending the SQP-approach of the bundle-Newton method for nonsmooth unconstrained minimization by\citet{Luksan} to nonsmooth nonlinearly constrained optimization problems, where we did not use a penalty function or a filter or an improvement function to handle the constraints. Instead --- after the commitment to only accept strictly feasible points as iteration points, while trial points do not need to have this property --- we computed the search direction by solving a convex QCQP in the hope to obtain preferably feasible points that yield a good descent. Since the duality gap for such problems is zero, if the iteration point is strictly feasible, we were able to establish a global convergence result under certain assumptions. Furthermore, we discussed the presence of $t_0^k$ in the line search, we explained why this should not be a problem when we use the solution of the QCQP as the search direction and we referred to \citet[\GenaueAngabeSix]{HannesPaperA} that this turns out to be true in practice for at least many examples of the Hock-Schittkowski collection by \citet{Schittkowski,Schittkowski2}.



\ifthenelse{\boolean{UseBibTex}}
{


\bibliographystyle{plainnatnew}
\bibliography{PublicationsALL}
\addcontentsline{toc}{section}{References}

}
{

}

\end{document}